\documentclass{amsart}
\usepackage{amsmath, graphicx}
\usepackage{comment}
\usepackage{stackrel}
\usepackage{amssymb, MnSymbol}
\usepackage{color}
\usepackage{amscd}
\input{xy}
\xyoption{all}
\usepackage{tikz}
\usepackage{graphicx}
\usepackage{tikz-cd}

\newtheorem{thm}{Theorem}[section]
\newtheorem{dfn}[thm]{Definition}
\newtheorem{lemma}[thm]{Lemma}
\newtheorem{prop}[thm]{Proposition}

\newtheorem*{dfn*}{Definition}
\newtheorem*{thm*}{Theorem}
\newtheorem*{prop*}{Proposition}
\newtheorem*{lemma*}{Lemma}

\newtheorem*{corol*}{Corollary}

\newcommand{\G}{\mathcal{G}}
\newcommand{\D}{\mathcal{D}}
\newcommand{\J}{\mathcal{J}}
\newcommand{\K}{\mathcal{K}}
\renewcommand{\O}{\mathcal{O}}
\newcommand{\T}{\mathcal{T}}
\newcommand{\DG}{\mathcal{D}(\mathcal{G})}
\newcommand{\OG}{\mathcal{O}_{\mathcal{G}}}
\newcommand{\KG}{\mathcal{K}_{\mathcal{G}}}
\newcommand{\ccinfty}{C_c^\infty}
\newcommand{\cinfty}{C^\infty}
\newcommand{\pro}{\mathrm{pro}}

\DeclareMathOperator{\Mor}{Mor}

\DeclareMathOperator{\im}{im}
\DeclareMathOperator{\id}{id}
\DeclareMathOperator{\Hom}{Hom}

\DeclareMathOperator{\Alg}{Alg}

\title{Equivariant periodic cyclic homology for ample groupoids}

\author{Francesco Pagliuca}
\address{School of Mathematics and Statistics \\
         University of Glasgow \\
         University Place \\
         Glasgow G12 8QQ \\
         United Kingdom 
}

\email{francesco.pagliuca@glasgow.ac.uk}

\author{Christian Voigt}
\address{School of Mathematics and Statistics \\
         University of Glasgow \\
         University Place \\
         Glasgow G12 8QQ \\
         United Kingdom 
}

\email{christian.voigt@glasgow.ac.uk}

\begin{document}

\begin{abstract}
We define and study bivariant equivariant periodic cyclic homology for actions of ample groupoids. In analogy to the group case, we show that the theory satisfies homotopy invariance, stability, and excision in both variables. We also prove an analogue of the Green-Julg theorem for actions of proper groupoids. 
\end{abstract}

\maketitle

\section{Introduction}

Cyclic homology is a homology theory for algebras which was introduced inde\-pendently by Connes and Tsygan. It can be viewed as an analogue of de Rham cohomology in the realm of noncommutative geometry \cite{CONNES_noncommutativedifferentialgeometry}, \cite{CONNES_noncommutativegeometry}. 
Groupoids, in turn, provide a vast range of examples of noncommutative spaces. Among other things, they arise naturally in topological dynamics and the classification of simple $ C^* $-algebras, see for instance \cite{LI_classifiablecartan}. 

The aim of this paper is to introduce a version of cyclic homology which takes into account groupoid symmetries. More precisely, we shall define and study bivariant equivariant periodic cyclic homology with respect to actions of ample Hausdorff groupoids on complex algebras. This is based on the Cuntz-Quillen approach to cyclic homology  \cite{CUNTZ_QUILLEN_algebraextensions}, \cite{CUNTZ_QUILLEN_nonsingularity}, \cite{CUNTZ_QUILLEN_excision}, and the group equivariant cyclic theory introduced in \cite{VOIGT_thesis}, \cite{VOIGT_equivariantperiodiccyclichomology}. 

Our main motivation comes from topological dynamics. Matui's HK-conjecture \cite{MATUI_etalegroupoidshifts} 
asserts that the $ K $-theory of a certain class of ample groupoids is given by groupoid homology in the sense of Crainic and Moerdijk \cite{CRAINIC_MOERDIJK_homologyetalegroupoids}. This conjecture is known to hold in a range of cases, see for instance
\cite{MATUI_homologyandtoplogicalfullgroups}, \cite{MATUI_etalegroupoidshifts},  
\cite{FARSI_KUMJIAN_PASK_SIMS_amplegroupoidshomology}, \cite{PROIETTI_YAMASHITA_homologyktheory1}, \cite{BOENICKE_DELLAIERA_GABE_WILLETT_dynamicasymptoticdimension}. 
As highlighted in \cite{PROIETTI_YAMASHITA_homologyktheory1}, it is natural to approach the HK-conjecture, and more generally the problem of calculating the $ K $-theory of ample groupoids, through the Baum-Connes conjecture \cite{TU_feuilletagesmoyennables}. From this perspective, one may hope to gain insight into the $ K $-theory of ample groupoids via a bivariant Chern character as in \cite{VOIGT_chern}. This is precisely where equivariant cyclic homology enters the picture. 

An important feature of group equivariant cyclic homology is that the basic object of the theory, the equivariant $ X $-complex, is not a chain complex in the usual sense of homological algebra. More precisely, the square of the differential of the equivariant $ X $-complex fails to be zero. This failure is controlled in a precise way, and after passing to mapping complexes one obtains honest chain complexes. The same is true in the groupoid case, and for this reason we only define a generalisation of bivariant periodic cyclic homology, and not of ordinary cyclic homology. 

We remark that the constructions in \cite{VOIGT_equivariantperiodiccyclichomology}, on which our work builds, are carried out in the category of bornological vector spaces. In the present paper we restrict ourselves to the purely algebraic framework of complex vector spaces without additional structure. 
This makes the algebraic nature of the theory as transparent as possible. However, let us remark that the constructions carry over to the bornological setting with minimal changes. 
In contrast, in order to treat more general classes of \'etale groupoids some of our techniques would need to be modified, and working with bornological vector spaces would then be the most natural starting point. 

Already in the ample case, removing our assumption that all groupoids under consideration are Hausdorff would require additional work as well. The key issue is that the standard space of ``locally constant functions'' on a non-Hausdorff ample groupoid fails to be closed under pointwise multiplication. This directly affects a basic ingredient in our theory, namely the definition of anti-Yetter-Drinfeld modules. It is conceivable that the theory of Hausdorff covers of non-Hausdorff groupoids, as developed in \cite{BRIX_GONZALES_HUME_LI_hausdorffcovers}, could be used to address this problem, but we shall not pursue this direction here.    

Let us now explain how this paper is organised. In section \ref{section:preliminaries} we collect some background material on groupoids and their actions. This includes a review of basic properties of locally constant functions on totally disconnected locally compact Hausdorff spaces and balanced tensor products. Section \ref{section:dgmodules} is an exposition of key facts about the category of $ \G $-modules for an ample Hausdorff groupoid $ \G $ and its monoidal structure. In particular, we provide an alternative description of $ \G $-modules in terms of what we call $ \ccinfty(\G) $-comodules. We define $ \G $-algebras, give some examples, and introduce the notion of a $ \G $-anti-Yetter-Drinfeld module. 
In section \ref{section:epch} we first review some facts about pro-categories and the concept of a paracomplex. We then introduce the main ingredients in the construction of equivariant cyclic homology, namely equivariant differential forms, the periodic tensor algebra, and the equivariant $ X $-complex. After these preparations we state the  definition of the bivariant equivariant cyclic homology groups $ HP^\G_*(A,B) $ for a pair of $ \G $-algebras $ A,B $. Section \ref{section:properties} is devoted to establishing fundamental properties of the resulting homology theory. We show that $ HP^\G_* $ is homotopy invariant with respect to smooth homotopies, stable under tensoring with algebras of finite rank operators on $ \G $-modules with invariant pairings, and satisfies excision in both variables. The proofs are similar to the group equivariant case, but additional arguments are required in particular in connection with stability. Finally, in section \ref{section:greenjulg} we establish an analogue in periodic cyclic homology of the Green-Julg theorem for the $ K $-theory of proper groupoids, see \cite{TU_novikovhyperbolique}. Here we make use of localisation techniques, which allows us to reduce the assertion to the group equivariant case studied in \cite{VOIGT_thesis}. 

We conclude with some comments on notation and terminology. In this paper we work over the complex numbers, and we write $ \otimes $ for the algebraic tensor product of complex vector spaces. Throughout, by a locally compact space we mean a topological space with a locally compact Hausdorff topology. In particular, all groupoids are assumed to be Hausdorff. 

It is a pleasure to thank Christian B\"onicke and Xin Li for instructive discussions about ample groupoids. We would also like to thank Julian Kranz for pointing out a gap in our original argument for stability. 

\section{Preliminaries} \label{section:preliminaries}

In this section we review some definitions and facts about ample groupoids and fix our notation. For further details and background we refer the reader to \cite{PATERSON_groupoidsinversesemigroups}, \cite{RENAULT_groupoidapproach}, \cite{TU_nonhausdorffgroupoids}.  

\subsection{Groupoids}

By definition, a groupoid $ \G $ is a small category in which all morphisms are invertible. We will frequently identify the set $ \G^{(0)} $ of objects of $ \G $ with the set of identity morphisms via the unit map $ u: \G^{(0)} \rightarrow \G $. The source and range maps are denoted by $ s,r: \G \rightarrow \G^{(0)} $, respectively. The 
composition $ m: \G^{(2)} \rightarrow \G $, given by $ m(\alpha, \beta) = \alpha \beta $, is defined on the 
set $ \G^{(2)} = \{(\alpha, \beta) \in \G \times \G \mid s(\alpha) = r(\beta) \} $ of composable morphisms. Finally, the inverse map $ i: \G \rightarrow \G $ is given by $ i(\alpha) = \alpha^{-1} $. We will use the notations $ \G_x = s^{-1}(x), \G^x = r^{-1}(x) $ and $ \G^y_x = \G_x \cap \G^y $ for $ x, y \in \G^{(0)} $. Moreover, if $ U, V \subseteq \G $ are subsets we write $ UV = \{\alpha \beta \mid \alpha \in U, \beta \in V, s(\alpha) = r(\beta) \} $. 

A topological groupoid is a groupoid $ \G $ such that $ \G $ and $ \G^{(0)} $ are equipped with topologies and all structure maps $ u,s,r,m,i $ are continuous. In this paper we will only consider topological groupoids which are Hausdorff and locally compact. A topological groupoid $ \G $ is \'etale if the range map $ r $ is a local homeomorphism. In this case
the unit map identifies $ \G^{(0)} $ with an open subset of $ \G $, and the maps $ u, s, m $ are local homeomorphisms as well. 

\begin{dfn}
An ample groupoid is an \'etale groupoid $ \G $ such that the base space $ \G^{(0)} $ is totally disconnected. 
\end{dfn}

An open source section of an \'etale groupoid $ \G $ is an open subset $ U \subseteq \G $ such that the restriction of the source map induces a homeomorphism $ U \rightarrow s(U) $. 
Similarly one defines open range sections. An open bisection of $ \G $ is an open subset of $ \G $ which is both a source and a range section. Ample groupoids can be equivalently characterised as those \'etale groupoids for which the set of all compact open bisections forms a basis for the topology. 

Let $ X,Y,Z $ topological spaces and let $ p: X \rightarrow Z $ and $ q : Y \rightarrow Z $ be continuous maps. The fibre product
$$
X \times_{p,q} Y = \{(x,y) \in X \times Y \mid p(x) = q(y) \}
$$
is equipped with the subspace topology from $ X \times Y $. If $ X, Y $ are locally compact and $ Z $ is Hausdorff then $ X \times_{p,q} Y $ is again locally compact. 

Let $ \G $ be an \'etale groupoid and let $ X $ be a locally compact space. A left action of $ \G $ on $ X $ consists of continuous maps
\begin{itemize}
\item [(i)] $ \pi: X \rightarrow \G^{(0)} $, called the anchor map,
\item[(ii)] $ m: \G \times_{s,\pi} X \rightarrow X $, denoted $ m(\alpha,x) = \alpha \cdot x $,
\end{itemize}
such that if $ (\alpha, \beta) \in \G^{(2)} $ and $ (\alpha,x) \in \G \times_{s,\pi} X $ then $ (\beta, \alpha \cdot x) \in \G \times_{s,\pi} X $ and $ \beta \cdot (\alpha\cdot x) = (\beta\alpha)\cdot x $, and $ u(\pi(x)) \cdot x = x $ for all $ x \in X $. 

One defines right actions of $ \G $ in an analogous fashion. A locally compact space with an action of $ \G $ will be called a $ \G $-space. 
A basic example of a left $ \G $-space is $ X = \G^{(0)} $ with anchor map $ \pi = \id $ and the action $ \alpha \cdot s(\alpha) = r(\alpha) $ for $ \alpha \in \G $. 

Let $ X $ be a left $ \G $-space $ X $ with anchor map $ \pi $. One defines $ \G \backslash X $ as the quotient of $ X $ by the equivalence relation $ \sim $ such that $ x \sim y $ if and only if there exists an element $ \alpha \in \G $ such that $ \pi(x) = s(\alpha) $ and $ y = \alpha \cdot x $. For an \'etale groupoid $ \G $ the projection map $ X \rightarrow \G \backslash X $ is always open, and the quotient topology on $ \G \backslash X $ is locally compact but not necessarily Hausdorff.

\subsection{Algebras and multipliers} 

By an algebra we shall mean an associative but not necessarily unital algebra over the complex numbers. We will typically work with algebras $ A $ which are essential in the sense that the multiplication map induces an isomorphism $ A \otimes_A A \cong A $. We will also be mostly interested in algebras with a nondegenerate multiplication in the sense that $ ab = 0 $ for all $ a \in A $ implies $ b = 0 $ and $ ab = 0 $ for all $ b \in A $ implies $ a = 0 $. Note that every unital algebra is essential and has a nondegenerate multiplication. 

The algebraic multiplier algebra $ M(A) $ of an algebra of $ A $ consists of all two-sided multipliers $ (L,R) $ of $ A $, compare for instance \cite[Appendix]{VANDAELE_multiplierhopfalgebras}. By definition, a two-sided multiplier $ (L,R) $ of $ A $ is a pair of a left multiplier $ L $, that is, a right $ A $-linear map $ L: A \rightarrow A $, and a right multiplier $ R $, that is, a left $ A $-linear map $ R: A \rightarrow A $, such that $ R(a)b = a(L(b)) $ for all $ a,b \in A $. The vector space $ M(A) $ becomes a unital algebra via composition of maps and unit $ 1 = (\id, \id) $. Every element $ a \in A $ defines a two-sided multiplier $ \iota(a) = (L_a, R_a) $ where $ L_a(b) = ab, R_a(b) = ba $. If the multiplication in $ A $ is nondegenerate then the canonical homomorphism $ \iota: A \rightarrow M(A) $ is injective, in which case we identify $ A $ with a subalgebra of $ M(A) $ in this way. 

A (left) $ A $-module $ M $ is called essential if the canonical map $ A \otimes_A M \rightarrow M $ induced by the module action is an isomorphism. Note that in this case $ AM $, the linear span of all elements $ a \cdot m $ for $ a \in A $ and $ m \in M $, equals $ M $. 
An algebra homomorphism $ f: A \rightarrow M(B) $ is called essential if $ B $ becomes an essential left and right module via $ f $. 
If $ A $ is unital then an $ A $-module $ M $ is essential iff it is unital in the sense that $ 1 \cdot m = m $ for all $ m \in M $, 
and an algebra homomorphism $ f: A \rightarrow M(B) $ is essential iff it is unital. For the following fact compare \cite[Proposition A.5]{VANDAELE_multiplierhopfalgebras}. 

\begin{lemma} 
Let $ f: A \rightarrow M(B) $ be an essential algebra homomorphism. If the multiplication in $ B $ is nondegenerate there exists a unique unital algebra homo\-morphism $ F: M(A) \rightarrow M(B) $ such that $ F \iota = f $. 
\end{lemma}

We say that an algebra $ A $ has local units if for every finite set of elements $ a_1, \dots, a_n $ of $ A $ there exists $ e \in A $ such that $ e a_i = a_i = a_i e $ for all $ i $. The multiplication in such an algebra is clearly nondegenerate. Let us record the following well-known fact. 

\begin{lemma} \label{localunitsessential}
Let $ A $ be an algebra with local units. Then a left $ A $-module is essential iff $ AM = M $. A similar statement holds for right modules. 
\end{lemma}

\begin{proof} 
We consider only the case of left modules. Let $ M $ be an essential left $ A $-module and let $ m \in M $. By assumption, there are elements $ a_i \in A $ and $ m_i \in M $ such that $ \sum a_i \cdot m_i = m $. Since $ A $ has local units we find $ e \in A $ such that $ e a_i = a_i $ for all $ i $, and hence the canonical map $ m: A \otimes_A M \rightarrow M $ sends $ \sum_i e \otimes a_i \cdot m $ to $ m $. We conclude that $ m $ is surjective. 

Now assume that $ \sum a_i \otimes m_i \in A \otimes_A M $ satisfies $ \sum a_i \cdot m_i = 0 $. By assumption we find $ e \in A $ such that $ e a_i = a_i $ for all $ i $ and hence 
$$
\sum a_i \otimes m_i = \sum ea_i \otimes m_i = \sum e \otimes a_i m_i = 0 
$$
in $ A \otimes_A M $. This means that the canonical map is injectve. 
\end{proof}

Note that Lemma \ref{localunitsessential} implies in particular that an algebra with local units is essential. 

\subsection{Function spaces}

Let $ X $ be a totally disconnected locally compact space. We define $ \ccinfty(X) $ 
as the space of all locally constant functions $ X \rightarrow \mathbb{C} $ with compact support. 
This is equal to the linear span of all characteristic functions $ \chi_K $ for compact open subsets $ K \subseteq X $. More precisely, we have the following basic fact. 

\begin{lemma} \label{lemma:funct.local.disc.tensor}
Let $ X $ be a totally disconnected locally compact space. 
Then every element $ f \in \ccinfty(X) $ can be written as a linear combination $ f = \sum_i c_i \chi_{U_i} $ for a finite family of pairwise disjoint compact open subsets $ U_k \subseteq X $ and coefficients $ c_k \in \mathbb{C} $. 
\end{lemma}

\begin{proof}
Since $ f $ is locally constant and the support of $ f $ is compact, the image $ f(X) $ is a finite subset of $ \mathbb{C} $. If we denote by $ c_1, \dots c_n $ the nonzero 
elements of $ f(X) $ and set $ U_k = f^{-1}(c_k) $ then each $ U_k \subseteq X $ is compact open, the sets $ U_1, \dots, U_n $ are pairwise disjoint, and we 
have $ f = \sum_i c_i \chi_{U_i} $.
\end{proof}

The vector space $ \ccinfty(X) $ is naturally an algebra with pointwise multiplication. This algebra has local units, 
and hence is essential, see the remark after Lemma \ref{localunitsessential}. We will write $ C^\infty(X) $ for the algebraic multiplier algebra of $ \ccinfty(X) $. This can be identified canonically with the algebra of all locally constant functions $ f: X \rightarrow \mathbb{C} $. 

Recall that a continuous map $ \varphi: X \rightarrow Y $ between locally compact spaces $ X, Y $ is proper iff $ \varphi^{-1}(K) $ is compact for every compact subset $ K \subseteq Y $. 

\begin{lemma} \label{lemma: non degeneracy module structure at level of functions}
Let $ X, Y $ be totally disconnected locally compact spaces and let $ \varphi: X \rightarrow Y $ be a continuous map. Then $ \varphi^*: \ccinfty(Y) \rightarrow C^\infty(X) = M(\ccinfty(X)), \varphi^*(f) = f \circ \varphi $ is a well-defined essential algebra homomorphism. If $ \varphi $ is proper then the image of $ \varphi^* $ is contained in $ \ccinfty(X) $. 
\end{lemma}

\begin{proof}
For $ f \in \ccinfty(Y) $ the function $ \varphi^*(f) = f \circ \varphi $ is locally constant since it is the composition of a continuous function and a locally constant function. It follows that $ \varphi^*: \ccinfty(Y) \rightarrow C^\infty(X) = M(\ccinfty(X)) $ is well-defined. Moreover this map is clearly an algebra homomorphism. 

In order to show that $ \varphi^* $ is essential let $ f \in \ccinfty(X) $. Since $ K = \operatorname{supp}(f) $ is compact the same is true for $ \varphi(K) $. We can 
cover $ \varphi(K) $ by finitely many compact open subsets of $ Y $, and if $ \chi $ denotes the characteristic function of the union of these sets then $ f = \varphi^*(\chi)f $ 
is contained in $ \varphi^*(\ccinfty(Y)) \ccinfty(X) $. 

Finally, assume that $ \varphi $ is proper and let $ g \in \ccinfty(Y) $. Then the preimage $ \varphi^{-1}(K) $ of the compact open set $ K = \operatorname{supp}(g) $ is compact open. If we write $ e \in \ccinfty(X) $ for the characteristic function of $ \varphi^{-1}(K) $ then we get $ \varphi^*(g) = \varphi^*(g) e = e \varphi^*(g) \in \ccinfty(X) $ as required because $ \varphi^* $ is essential. 
\end{proof}

\begin{prop} \label{Lemma:cartesian loc. const. function}
Let $ X, Y $ be totally disconnected locally compact spaces. Then the canonical linear map 
$$ 
\gamma: \ccinfty(X) \otimes \ccinfty(Y) \rightarrow \ccinfty(X \times Y), 
$$ 
given by $ \gamma(f \otimes g)(x,y) = f(x) g(y) $, is an isomorphism. 
\end{prop}

\begin{proof}
Assume $ F = \sum f_i \otimes g_i \in \ccinfty(X) \otimes \ccinfty(Y) $ satisfies $ \gamma(F) = 0 $. By Lemma \ref{lemma:funct.local.disc.tensor} we can write each $ f_i $ as a linear combination of characteristic functions $ \chi_{U_{ij}} $ for mutually disjoint compact open subsets $ U_{ij} \subseteq X $, and similarly each $ g_i $ as a linear combination of characteristic functions $ \chi_{V_{ik}} $ for mutually disjoint compact open subsets of $ Y $. Upon taking intersections of these subsets, it follows that $ F $ can be written in the form $ F = \sum_k c_k \chi_{U_k} \otimes \chi_{V_k} $, where $ U_1, \dots, U_n $ and $ V_1, \dots, V_n $ are mutually disjoint compact open subsets of $ X $ and $ Y $, respectively. Without loss of generality we may assume that these sets are all nonempty. For every index $ k $ pick $ (x_k,y_k) \in U_k \times V_k $. Then the relation
$$ 
0 = \gamma(F)(x_k,y_k) = c_k \chi_{U_k}(x_k)\chi_{V_k}(y_k) = c_k 
$$ 
gives $ c_k = 0 $. Hence $ F = 0 $, and it follows that $ \gamma $ is injective. 
    
To show surjectivity, it suffices to verify that the characteristic function $ \chi_W $ of an arbitrary compact open subset $ W \subseteq X \times Y $ is contained in the image of $ \gamma $. For this it is enough to write $ W $ as a disjoint union of sets of the form $ U \times V $ where $ U \subseteq X $ and $ V \subseteq Y $ are compact open. In order to obtain such a decomposition of $ W $, note first that since $ X $ and $ Y $ are totally disconnected and locally compact they both have a basis for their topology made up of compact open sets. 
In particular, for every point $ w = (x, y) \in W $ we find compact open neighborhoods $ U_w \subseteq X $ of $ x $ and $ V_w \subseteq Y $ of $ y $ such that the rectangle $ R_w = U_w \times V_w $ is contained in $ W $. Since $ W $ is compact we obtain a finite cover of $ W $ by rectangles $ R_{w_1}, \dots, R_{w_n} $ for some $ w_1, \dots, w_n \in W $. Upon taking intersections of the compact open sets $ U_{w_i} $ and $ V_{w_i} $ making up the rectangles $ R_{w_i} $, we can refine this to a finite cover of $ W $ consisting of mutually disjoint compact open rectangles as required. 
\end{proof}

\begin{lemma} \label{lemma:restriction is epimorphism}
Let $ X $ be a totally disconnected locally compact space and let $ K \subseteq X $ be a closed subset. Then the canonical restriction map $ \ccinfty(X) \rightarrow \ccinfty(K) $, mapping $ f $ to $ f_{\mid K} $, is surjective.
\end{lemma}

\begin{proof}
For any given $ f \in \ccinfty(K) $ we have to construct a function $ F \in \ccinfty(X) $ such that $ F_{|K} = f $. Since every element of $ \ccinfty(K) $ is a linear combination of characteristic functions it suffices to consider the case that $ f = \chi_U $ for some compact open set $ U \subseteq K $. In this case, since $ U $ is open, there exists an open set $ V \subseteq X $ such that $ V \cap K = U $. Using that $ V $ is open and $ X $ is totally disconnected we can write $ V $ as a union of compact open subsets of $ X $. Since we have $ U \subseteq V $, these sets are in particular an open cover of the compact set $ U $. This means that we can find finitely many compact open subsets $ W_1, \dots, W_n \subseteq X $ such that $ W_i \subseteq V $ for all $ i $ and the union $ W $ of the $ W_i $ satisfies $ W \cap K = U $. It follows that the function $ F = \chi_W $ has the desired properties.  
\end{proof}

Let $ X, Y, Z $ be totally disconnected locally compact spaces and let $  p: X \rightarrow Z, q: Y \rightarrow Z $ be continuous maps. Then $ \ccinfty(X) $ and $ \ccinfty(Y) $ become $ \ccinfty(Z) $-modules via the pullback homomorphisms $ p^*, q^* $, see Lemma \ref{lemma: non degeneracy module structure at level of functions}. The balanced tensor product of $ \ccinfty(X) $ and $ \ccinfty(Y) $ over $ \ccinfty(Z) $ with respect to $ p,q $ is defined as the quotient
$$
\ccinfty(X) \stackrel{p,q}{\otimes}\ccinfty(Y) = (\ccinfty(X) \otimes \ccinfty(Y))/R,
$$
where $ R $ is the linear subspace spanned by all elements of the form $ f p^*(h) \otimes g - f \otimes q^*(h) g $ for $ f \in \ccinfty(X), g \in \ccinfty(Y) $ 
and $ h \in \ccinfty(Z) $. 

\begin{prop} \label{prop:isomorphisms ccinfty}
Let $ X, Y, Z $ be totally disconnected locally compact spaces and let $ p: X \rightarrow Z, q: Y \rightarrow Z $ be continuous maps. Then the canonical $ \ccinfty(Z) $-linear map
$$
\ccinfty(X) \stackrel{p,q}{\otimes }\ccinfty(Y) \rightarrow \ccinfty(X \times_{p,q} Y) 
$$
is an isomorphism. 
\end{prop}

\begin{proof}
It is straightforward to check that the composition of the canonical homo\-morphism $ \gamma: \ccinfty(X) \otimes \ccinfty(Y) \rightarrow \ccinfty(X \times Y) $ with the restriction homomorphism $ \ccinfty(X \times Y) \rightarrow \ccinfty(X \times_{p,q} Y) $ factorises through $ \ccinfty(X) \stackrel{p,q}{\otimes }\ccinfty(Y) $. We shall write $ \gamma_{p,q} $ for the resulting $ \ccinfty(Z) $-linear map $ \ccinfty(X) \stackrel{p,q}{\otimes }\ccinfty(Y) \rightarrow \ccinfty(X \times_{p,q} Y) $. Due to Proposition \ref{Lemma:cartesian loc. const. function} and Lemma \ref{lemma:restriction is epimorphism} the map $ \gamma_{p,q} $ is surjective,  
and it remains only to show that $ \gamma_{p,q} $ is injective. 

Assume that $ F \in \ccinfty(X) \stackrel[]{p,q}{\otimes} \ccinfty(Y)$ satisfies $ \gamma_{p,q}(F) = 0 $. As explained in the proof of Proposition \ref{Lemma:cartesian loc. const. function}, we can represent $ F $ as a linear combination $ F = \sum_k c_k \chi_{U_k} \otimes \chi_{V_k} $ where $ U_1, \dots, U_n $ and $ V_1, \dots, V_n  $ are mutually disjoint compact open subsets of $ X $ and $ Y $, respectively. If there exists an index $ k $ and points $ x \in U_k, y \in V_k $ such that $ p(x) = q(y) $ then $ (x,y) \in X\stackrel[]{p, q}{\times} Y $ and $ c_k = c_k \chi_{U_k}(x) \chi_{V_k}(y) = \gamma_{p,q}(F)(x,y) = 0 $. 
Therefore we may assume without loss of generality that $ p(U_k) \cap q(V_k) = \emptyset $ for all $ k $. Using that $ Z $ is locally compact and hence regular we can then find compact open sets $ E_k \subseteq Z $ such that $ p(U_k) \subseteq E_k $ and $ E_k \cap q(V_k) = \emptyset $ for all $ k $. 
It follows that $ e_k = \chi_{E_k} \in \ccinfty(Z) $ satisfies $ \chi_{U_k} \cdot e_k = \chi_{U_k} $ and $ e_k \cdot \chi_{V_k} = 0 $ for all $ k $. Hence we conclude 
$$ 
F = \sum_k c_k \chi_{U_k} \otimes \chi_{V_k} = \sum_k c_k \chi_{U_k} \cdot e_k \otimes \chi_{V_k} - c_k \chi_{U_k} \otimes e_k \cdot \chi_{V_k} = 0 
$$ 
as required. 
\end{proof}

\subsection{Proper groupoids} 

In this subsection we review some basic facts about proper groupoids, compare \cite{RENAULT_groupoidapproach}, \cite{TU_novikovhyperbolique}.  

Let us first recall that a (left) Haar system on an \'etale groupoid $ \G $ is a family $ (\lambda^x)_{x \in \G^{(0)}} $ of positive regular Borel measures $ \lambda^x $ on $ \G $ such that 
\begin{enumerate}
\item the support of $ \lambda^x $ is $ \G^x $ for all $ x \in \G^{(0)} $, 
\item for every $ f \in C_c(\G) $ the function $ \lambda(f): \G^{(0)} \rightarrow \mathbb{C} $ given by 
$$
\lambda(f)(x) = \int_{\G^x} f(\beta) d\lambda^x(\beta) 
$$
is contained in $ C_c(\G^{(0)}) $,
\item we have 
$$
\int_{\G^{s(\alpha)}} f(\alpha \beta) d\lambda^{s(\alpha)}(\beta) = \int_{\G^{r(\alpha)}} f(\beta) d\lambda^{r(\alpha)}(\beta)
$$
for all $ f \in C_c(\G^{(0)}) $ and $ \alpha \in \G^{(0)} $. 
\end{enumerate}

Every \'etale groupoid $ \G $ admits a canonical Haar system $ (\lambda^x)_{x \in \G^{(0)}} $ given by the counting measures on the range fibers, so that 
$$
\int_{\G^x} f(\beta) \lambda^x(\beta) = \sum_{\beta \in \G^x} f(\beta) 
$$
for $ x \in \G^{(0)} $ and $ f \in C_c(\G) $. 
This Haar system behaves well with respect to integration of locally constant functions in the following sense. 

\begin{lemma} \label{lemma: integration over fiber is continuous for functions with proper support}
Let $ \G $ be an ample groupoid and let $ f: \G \rightarrow \mathbb{C} $ be a locally constant function such that $ \operatorname{supp}(f) \cap r^{-1}(K) $ is compact for all compact 
sets $ K \subseteq \G^{(0)} $. Then the function $ \lambda(f): \G^{(0)} \rightarrow \mathbb{C} $ defined by 
$$ 
\lambda(f)(x) = \sum_{\alpha\in \G^x}f(\alpha)
$$ 
is locally constant.
\end{lemma}

\begin{proof}
Let $ x \in \G^{(0)} $ and let $ V $ be a compact open neighbourhood of $ x $. By assumption the set $ \operatorname{supp}(f) \cap r^{-1}(V)\subseteq \G $ is compact. Hence $ g = f \chi_{r^{-1}(V)} $ is contained in $ \ccinfty(\G) $. Writing $ g $ as a linear combination of characteristic functions of compact open bisections of $ \G $ it is straightforward to check that $ \lambda(g) $ is locally constant. 
Since by construction the functions $ \lambda(g) $ and $ \lambda(f) $ agree on $ V $ it follows that $ \lambda(f) $ is locally constant in a neighborhood of $ x $, and since $ x $ was arbitrary this yields the claim. 
\end{proof}

Let us now come to the definition of properness for groupoids. Recall that we write $ s, r $ for the source and range maps, respectively. 

\begin{dfn}
An \'etale groupoid $ \G $ is called proper if $ (s,r): \G \rightarrow \G^{(0)} \times \G^{(0)} $ is a proper map. 
\end{dfn}

If $ \G $ is a proper \'etale groupoid then the quotient space $ \G \backslash \G^{(0)} $ is Hausdorff, 
and if $ \G $ is ample then $ \G \backslash \G^{(0)} $ is a totally disconnected locally compact space. 
The canonical quotient map $ \G^{(0)} \rightarrow \G \backslash \G^{(0)} $ induces an embedding $ \ccinfty(\G\backslash \G^{(0)}) \rightarrow C^{\infty}(\G^{(0)}) $ in this case. As a consequence, every essential $ \ccinfty(\G^{(0)}) $-module becomes an essential $ \ccinfty(\G \backslash \G^{(0)}) $-module in a canonical way. 

Next we review the concept of a cut-off function for \'etale groupoids, compare \cite[Definition 6.7]{TU_novikovhyperbolique}. 

\begin{dfn}
Let $ \G $ be an \'etale groupoid. A cut-off function for $ \G $ is a continuous function $ c: \G^{(0)} \rightarrow [0, \infty) $ such that
\begin{enumerate}
\item for every $ x \in \G^{(0)} $ we have $ \sum_{\alpha \in \G^x} cs(\alpha) = 1 $,
\item the map $ r :\operatorname{supp}(cs) \rightarrow \G^{(0)} $ is proper.
\end{enumerate}
\end{dfn}

It is shown in \cite[Proposition 6.10 and Proposition 6.11]{TU_novikovhyperbolique} that the existence of cut-off functions is closely related to properness. 
The following result is a variant of \cite[Proposition 6.11]{TU_novikovhyperbolique}. 

\begin{prop} \label{smoothcutoff}
Let $ \G $ be a proper ample groupoid with $ \G \backslash \G^{(0)} $ paracompact. Then $ \G $ admits a locally constant cut-off function. If $ \G \backslash \G^{(0)} $ is compact 
then $ \G $ admits a locally constant cut-off function with compact support.
\end{prop}

\begin{proof}
The quotient $ \G \backslash \G^{(0)} $ is a totally disconnected locally compact Hausdorff space. By assumption it is also paracompact, and hence can be written as a disjoint union of a family of open $ \sigma $-compact totally disconnected locally compact Hausdorff spaces, see \cite[Section 9.10]{BOURBAKI_generaltopology1}. 
Every $ \sigma $-compact totally disconnected locally compact space, in turn, can be written as a disjoint union of a countable family of compact open subsets. 
As a consequence, there is a cover $ (V_i)_{i \in I} $ of $ \G \backslash \G^{(0)} $ consisting of mutually disjoint compact open subsets. 

Since the projection map $ \pi: \G^{(0)} \rightarrow \G \backslash \G^{(0)} $ is open we can find $ n_i \in \mathbb{N} $ and compact open subsets $ U_{i,1} \dots, U_{i, n_i} $ 
of $ \G^{(0)} $ for each $ i \in I $ such that $ \pi(U_{ij}) \subseteq V_i $ for all $ j $ and $ \pi(U_{i,1}) \cup \cdots \cup \pi(U_{i, n_i}) = V_i $. 
Without loss of generality we can arrange the sets $ U_{i,j} $ to be mutually disjoint. 
We then define $ d: \G^{(0)} \rightarrow [0, \infty) $ by 
$$ 
d = \sum_{i \in I} \sum_{j = 1}^{n_i} \chi_{U_{i,j}}.
$$
By construction $ d $ is well-defined and locally constant. In fact, $ d $ is the characteristic function of the union of the sets $ U_{i,j} $. 

If $ K \subseteq \G^{(0)} $ is compact then $ \pi(K) \cap V_i $ is nonempty only for finitely many $ i \in I $, and hence 
$ \operatorname{supp}(d) \cap \pi^{-1}(\pi(K)) $ is compact. As a consequence, 
$$ 
\mathrm{supp}(d s) \cap r^{-1}(K) = (s\times r)^{-1}(\operatorname{supp}(d) \times K) = (s\times r)^{-1}(\operatorname{supp}(d) \cap \pi^{-1}(\pi(K)) \times K) 
$$ 
is compact by properness of $ \G $. According to Lemma \ref{lemma: integration over fiber is continuous for functions with proper support} it follows that the function $ \lambda(ds) $ is locally constant. 

Note that for every $ x \in \G^{(0)} $ there exists an index $ i \in I $ such that $ \pi(x) \in V_i $. This implies that there exists some $ 1 \leq j \leq n_i $ 
and an element $ \alpha \in \G^x $ such that $ s(\alpha) \in U_{i,j} $, 
and we conclude that $ \lambda(ds)(x) = \sum_{\alpha\in \G^x} d(s(\alpha)) > 0 $.
It is then straightforward to check that $ c(x) = d(x)/\lambda(ds)(x) $ is a locally constant cut-off function for $ \G $.

Finally, if $ \G \backslash \G^{(0)} $ is compact then the index set $ I $ in the above construction can be taken to be a singleton, and then both $ d $ and $ c $ have compact support. 
\end{proof}

\section{The category of $ \G $-modules} \label{section:dgmodules}

In this section we discuss $ \G $-modules for an ample groupoid $ \G $ and show that the category of $ \G $-modules is monoidal. This allows us to introduce $ \G $-algebras as algebra objects in this category. We also discuss the notion of a $ \G $-anti-Yetter-Drinfeld module. 

\subsection{$ \G $-modules}

Let us write $ \DG $ for the Steinberg algebra of $ \G $, that is, the vector space $ \ccinfty(\G) $ equipped with the convolution product given by 
$$
(f * g)(\alpha) = \sum_{\beta\in\G^{r(\alpha)}} f(\beta) g(\beta^{-1}\alpha)
= \sum_{\gamma\in\G_{s(\alpha)}} f(\alpha \gamma^{-1}) g(\gamma). 
$$
The algebra $ \DG $ has local units, given by the characteristic functions of compact open subsets of the unit space $ \G^{(0)} $, and it is unital iff $ \G^{(0)} $ is compact, compare \cite{STEINBERG_discreteinversesemigroup}. 

\begin{dfn}
A left $ \G $-module is an essential left module over $ \DG $. A linear map $ f: M \rightarrow N $ between left $ \G $-modules is called $ \G $-equivariant if it is $ \DG $-linear. 
\end{dfn}

One defines right $ \G $-modules in the same way, namely as essential right modules over $ \DG $. 
We denote by $ \G \operatorname{-Mod} $ the category of essential left $ \G $-modules and $ \G $-equivariant linear maps. Since we will be mostly working with left modules, we refer to the objects of $ \G \operatorname{-Mod} $ simply as $ \G $-modules. 

Our first aim is to provide an alternative description of $ \G $-modules, inspired by the discussion in \cite{BUSS_HOLKAR_MEYER_universalproperty}. 
Consider the maps $ d_0,d_1,d_2: \G^{(2)} \rightarrow \G^{(1)} = \G $ given by 
\begin{equation*} \label{equation: v_0,v_1,v_2}
d_0(\alpha,\beta) = \beta, \quad d_1(\alpha,\beta) = \alpha\beta,\quad d_2(\alpha,\beta) = \alpha
\end{equation*}
for $ (\alpha,\beta) \in \G^{(2)} = \G \times_{s,r} \G $. 
Each of these maps can be used to turn $ \ccinfty(\G^{(2)}) $ into a $ \ccinfty(\G) $-module with the action by pointwise multiplication, that is,  
$$
(f \cdot_i g)(\alpha, \beta) = f(\alpha, \beta) g(d_i(\alpha, \beta)).  
$$
Let us write $ \ccinfty(\G^{(2)}, d_i) $ for the resulting $ \ccinfty(\G) $-module. 
Then if $ P, Q $ are $ \ccinfty(\G) $-modules and $ T: P \rightarrow Q $ is a $ \ccinfty(\G) $-linear map we get induced linear maps  
$$
\id \otimes T: \ccinfty(\G^{(2)}, d_i) \otimes_{\ccinfty(\G)} P \rightarrow \ccinfty(\G^{(2)}, d_i) \otimes_{\ccinfty(\G)} Q
$$
for $ i = 0,1,2 $. We will denote these maps by $ d_i^*(T) $ in the sequel, in order to distinguish the different module structures on $ \ccinfty(\G^{(2)}) $ used in the construction. 
Consider the special case $ P = \ccinfty(\G)\stackrel{r,\id}{\otimes} M,\; Q =  \ccinfty(\G)\stackrel{s,\id}{\otimes} M $ for a $ \ccinfty(\G^{(0)}) $-module $ M $, where both $ P $ and $ Q $ are viewed as $ \ccinfty(\G) $-modules with the action by pointwise multiplication in the first tensor factor. If we write 
\begin{equation*} 
v_0 = r d_1 = r d_2, \quad v_1 = r d_0 = s d_2, \quad v_2 = s d_0 = s d_1,   
\end{equation*}
then we can view the tensor product maps $ \id \otimes T $ for $ i = 0,1,2 $ in this case as maps
\begin{align*}
d_0^*(T)&:\ccinfty(\G^{(2)})\stackrel{v_1,\id}{\otimes} M \rightarrow \ccinfty(\G^{(2)})\stackrel{v_2,\id}{\otimes} M,\\
d_1^*(T)&:\ccinfty(\G^{(2)})\stackrel{v_0,\id}{\otimes} M \rightarrow \ccinfty(\G^{(2)})\stackrel{v_2,\id}{\otimes} M,\\
d_2^*(T)&:\ccinfty(\G^{(2)})\stackrel{v_0,\id}{\otimes} M \rightarrow \ccinfty(\G^{(2)})\stackrel{v_1,\id}{\otimes} M,
\end{align*}
taking into account the canonical identifications of the respective tensor products. These identifications will be tacitly used in the following definition. 

\begin{dfn}
A $ \ccinfty(\G) $-comodule is an essential $ \ccinfty(\G^{(0)}) $-module $ M $ together with a $ \ccinfty(\G) $-linear isomorphism 
$$
T_M:\ccinfty(\G)\stackrel{r,\id}{\otimes} M \rightarrow \ccinfty(\G)\stackrel{s,\id}{\otimes} M
$$
satisfying the coaction identity 
$$ 
d_0^*(T_M) d_2^*(T_M) = d_1^*(T_M).
$$
A morphism of $ \ccinfty(\G) $-comodules is a $ \ccinfty(\G^{(0)}) $-linear map $ f: M \rightarrow N $ such that $ (\id \otimes f)T_M = T_N (\id \otimes f) $. 
\end{dfn}

The terminology used here is inspired by the duality theory for actions of groups and quantum groups. We will write $ \ccinfty(\G) \operatorname{-Comod} $ for the category of $ \ccinfty(\G) $-comodules. 

Our aim is to show that $ \G $-modules and $ \ccinfty(\G) $-comodules are essentially the same thing.  Let us first explain how to pass from $ \G $-modules to $ \ccinfty(\G) $-comodules. Consider $ M = \ccinfty(\G) = \DG $ as a left module over itself. The homeomorphism  
$$ 
t: \G \times_{s,r} \G \rightarrow \G \times_{r,r} \G, \qquad t(\alpha, \beta) = (\alpha, \alpha \beta)  
$$ 
induces a linear isomorphism $ T: \ccinfty(\G) \stackrel{r,r}{\otimes} \ccinfty(\G)\rightarrow \ccinfty(\G)\stackrel{s,r}{\otimes} \ccinfty(\G) $, given by 
$$
T(f)(\alpha, \beta) = f(\alpha, \alpha \beta)
$$
for $ f \in \ccinfty(\G) \stackrel{r,r}{\otimes} \ccinfty(\G) = \ccinfty(\G \times_{r,r} \G) $. Note that $ T $ is right $ \DG $-linear with respect to the multiplication action on the second tensor factor because $ t $ is right $ \G $-equivariant. Moreover, $ T $ is $ \ccinfty(\G) $-linear with respect to the pointwise multiplication action on the first tensor factor on both sides. We will also refer to $ T $ as the canonical map of $ M = \ccinfty(\G) $. 

Let us show that the canonical map turns $ \ccinfty(\G) = \DG $ into a $ \ccinfty(\G) $-comodule. To this end, note that for each $ i = 0,1,2 $ we have homeomorphisms
\begin{align*} 
\G^{(2)} \times_{d_0, \pi}(\G \times_{s,r} \G) \cong \G^{(2)} \times_{v_2, r} \G, \\
\G^{(2)} \times_{d_0, \pi}(\G \times_{r,r} \G) \cong \G^{(2)} \times_{v_1, r} \G, \\ 
\G^{(2)} \times_{d_1, \pi}(\G \times_{s,r} \G) \cong \G^{(2)} \times_{v_2, r} \G, \\
\G^{(2)} \times_{d_1, \pi}(\G \times_{r,r} \G) \cong \G^{(2)} \times_{v_0, r} \G, \\
\G^{(2)} \times_{d_2, \pi}(\G \times_{s,r} \G) \cong \G^{(2)} \times_{v_1, r} \G, \\
\G^{(2)} \times_{d_2, \pi}(\G \times_{r,r} \G) \cong \G^{(2)} \times_{v_0, r} \G, 
\end{align*} 
where $ \pi $ denotes the projection to the first copy of $ \G $ in either case. 
Using these homeomorphisms we can identify the maps induced by $ t $ on these fibre products as 
\begin{align*}
(\id \times_{d_0,\pi} t)&:\G^{(2)} \times_{v_2,r}\G \rightarrow \G^{(2)} \times_{v_1,r}\G, \quad (\id \times_{d_0,\pi} t)(\alpha,\beta,\gamma) = (\alpha, \beta, \beta \gamma), \\
(\id \times_{d_1,\pi} t)&:\G^{(2)} \times_{v_2,r}\G \rightarrow \G^{(2)} \times_{v_0,r}\G, \quad (\id \times_{d_1,\pi} t)(\alpha,\beta,\gamma) = (\alpha, \beta, \alpha\beta\gamma),\\
(\id \times_{d_2,\pi} t)&:\G^{(2)} \times_{v_1,r}\G \rightarrow \G^{(2)} \times_{v_0,r}\G, \quad (\id \times_{d_2,\pi} t)(\alpha,\beta,\gamma)=(\alpha, \beta, \alpha \gamma).
\end{align*}
From this description it is obvious that $ (\id \times_{d_2,\pi} t) (\id \times_{d_0,\pi} t) = (\id \times_{d_1,\pi} t) $. Since $ d_i^*(T) $ is the transpose of $ \id \times_{d_i,\pi} t $ this yields the coaction identity $ d_0^*(T) d_2^*(T) = d_1^*(T) $ for $ T $.

Now let $ M $ be an arbitrary $ \G $-module. Then $ M $ is a non-degenerate $ \ccinfty(\G^{(0)}) $-module with the restricted action along the inclusion $ \G^{(0)} \rightarrow \G $. Moreover, we obtain a $ \ccinfty(\G) $-linear isomorphism $ T_M:\ccinfty(\G)\stackrel{r,\id}{\otimes} M \rightarrow \ccinfty(\G)\stackrel{s,\id}{\otimes} M $  
as the unique map fitting into the commutative diagram 
\begin{center}
\begin{tikzcd}
(\ccinfty(\G) \stackrel{r,r}{\otimes} \ccinfty(\G)) \otimes_{\DG} M \arrow[r,"T \otimes \id"]\arrow[d,"\cong"] & \arrow[d,"\cong"] 
(\ccinfty(\G) \stackrel{s,r}{\otimes} \ccinfty(\G)) \otimes_{\DG} M \\
\ccinfty(\G)\stackrel{r,\id}{\otimes} M \arrow[r,"T_M"] & \ccinfty(\G)\stackrel{s,\id}{\otimes} M. 
\end{tikzcd}
\end{center}
Here we use the identification $ \DG \otimes_{\DG} M \cong M $ and right $ \DG $-linearity of $ T $. 
From the construction, it is immediate that we obtain analogous commutative diagrams linking $ d_i^*(T_M) $ and $ d_i^*(T) \otimes \id $ for $ i = 0,1,2 $ and hence the coaction identity 
$ d_0^*(T_M) d_2^*(T_M) = d_1^*(T_M) $ holds. We will refer to $ T_M $ as the canonical map of $ M $ in the sequel. 

These constructions are clearly compatible with $ \G $-equivariant linear maps. That is, if $ f: M \rightarrow N $ is $ \G $-equivariant, then it is a morphism of $ \ccinfty(\G) $-comodules with respect to the canonical maps $ T_M, T_N $. 
In summary, we obtain the following result. 

\begin{lemma} \label{modtocomod}
Let $ M $ be a $ \G $-module. Then $ M $ becomes a $ \ccinfty(\G) $-comodule via the canonical map $ T_M $ defined above. This assignment defines a functor $ A: \G \operatorname{-Mod} \rightarrow \ccinfty(\G) \operatorname{-Comod} $. 
\end{lemma}

Let us now discuss how to pass from $ \ccinfty(\G) $-comodules to $ \G $-modules. To this end consider the integration map $ \lambda: \ccinfty(\G) \rightarrow\ccinfty(\G^{(0)}) $ defined by 
$$ 
\lambda(f)(x) = \sum_{\alpha\in \G^x} f(\alpha).
$$
This map is $ \ccinfty(\G^{(0)}) $-linear with respect to the action of $ \ccinfty(\G^{(0)}) $ on $ \ccinfty(\G) $ induced by the range map $ r $ and the obvious multiplication action on $ \ccinfty(\G^{(0)}) $. Indeed, for $ f \in \ccinfty(\G), h \in \ccinfty(\G^{(0)}) $ we have
$$
\lambda(h \cdot f)(x) = \sum_{\alpha\in \G^x} h(r(\alpha))f(\alpha)  = \sum_{\alpha\in \G^x} h(x) f(\alpha) = h(x) \lambda(f)(x) = (h \cdot \lambda(f))(x) 
$$ 
as required. In fact, we have the following equivariance property. 

\begin{lemma} \label{integrationmap}
The integration map $ \lambda: \ccinfty(\G) \rightarrow\ccinfty(\G^{(0)}) $ is $ \G $-equivariant with respect to the left multiplication action on $ \ccinfty(\G) = \DG $. 
\end{lemma}

\begin{proof}
Let $ f \in \ccinfty(\G) $ and let $ U \subseteq \G $ be a compact open bisection. We calculate 
\begin{align*}
\lambda(\chi_U * f)(x) &= \sum_{\beta \in \G^x} (\chi_U * f)(\beta) \\
&= \sum_{\alpha, \beta \in \G^x} \chi_U(\alpha) f(\alpha^{-1} \beta) \\
&= \sum_{\alpha \in \G^x} \chi_U(\alpha) \sum_{\gamma \in \G^{s(\alpha)}} f(\gamma) \\
&= \sum_{\alpha \in \G^x} \chi_U(\alpha) \lambda(f)(s(\alpha)) \\
&= \chi_U \cdot \lambda(f)(x),  
\end{align*}
and this yields the claim. 
\end{proof}

Now assume that $ M $ is a $ \ccinfty(\G) $-comodule with canonical map $ T_M $ and define $ \mu_M: \DG \otimes M \rightarrow M $ by the formula 
$$ 
\mu_M = (\lambda \otimes \id) T^{-1}_M q_M, 
$$
where $ q_M: \DG \otimes M = \ccinfty(\G) \otimes M \rightarrow \ccinfty(\G)\stackrel{s,\id}{\otimes} M $ is the quotient map. 
Consider the diagram 
\begin{center}
\begin{tikzcd}
\ccinfty(\G^{(2)}) \stackrel{v_2, \id}{\otimes} M \arrow[d," \cong"] \arrow[r,"d_0^*(T^{-1}_M)"] &\ccinfty(\G^{(2)})\stackrel[]{v_1,\id}{\otimes} M \arrow[d,"\cong"] \arrow[r,"d_2^*(T^{-1}_M)"]& \ccinfty(\G^{(2)})\stackrel[]{v_0,\id}{\otimes} M \arrow[d,"\cong"] \\
\ccinfty(\G) \stackrel[]{s, r}{\otimes} \ccinfty(\G) \stackrel[]{s, \id}{\otimes} M 
\arrow[r,"\id \otimes T^{-1}_M"] &\ccinfty(\G) \stackrel[]{s, r}{\otimes} \ccinfty(\G)\stackrel[]{r,\id}{\otimes} M \arrow[d,"\id \otimes \lambda \otimes \id"] \arrow[r,"(T^{-1}_M)_{13}"]
& (\ccinfty(\G) \stackrel[]{s, r}{\otimes} \ccinfty(\G))\stackrel[]{v_0,\id}{\otimes} M \arrow[d,"\id \otimes \lambda \otimes \id"] \\
& \ccinfty(\G) \stackrel[]{s, \id}{\otimes} M \arrow[r,"T^{-1}_M"] & \ccinfty(\G) \stackrel[]{r, \id}{\otimes} M.
\end{tikzcd}
\end{center}    
Here $ (T^{-1}_M)_{13} $ is the map $ T^{-1}_M $ applied to the first and third tensor factors. 
The two top squares are commutative by the definition of $ d_0^*(T^{-1}_M) $ and $ d_2^*(T^{-1}_M) $. The bottom right square commutes trivially. As a consequence, we obtain
$$
(\lambda \otimes \id)(\id \otimes \lambda \otimes \id) d_2^*(T^{-1}_M)d_0^*(T^{-1}_M)(f \otimes g \otimes m) = \mu_M(\id \otimes \mu_M)(f \otimes g \otimes m)
$$
for all $ f,g \in \ccinfty(\G) $ and $ m \in M $. Similarly, we have a commutative diagram 
\begin{center}
\begin{tikzcd}
\ccinfty(\G^{(2)}) \stackrel[]{v_2, \id}{\otimes} M \arrow[d," T^{-1} \otimes \id"] \arrow[r,"d_1^*(T^{-1}_M)"] &\ccinfty(\G^{(2)})\stackrel[]{v_0,\id}{\otimes} M \arrow[d,"T^{-1} \otimes \id"] \\
(\ccinfty(\G) \stackrel[]{r, r}{\otimes} \ccinfty(\G)) \stackrel[]{s \pi,\id}{\otimes} M \arrow[d," \lambda \otimes \id \otimes \id"] \arrow[r,"\id \otimes T^{-1}_M"] &(\ccinfty(\G) \stackrel[]{r, r}{\otimes} \ccinfty(\G))\stackrel[]{r \pi,\id}{\otimes} M \arrow[d,"\lambda \otimes \id \otimes \id"] \arrow[r,"T \otimes \id"] &(\ccinfty(\G) \stackrel[]{s, r}{\otimes} \ccinfty(\G))\stackrel[]{v_0,\id}{\otimes} M \arrow[d,"\id \otimes \lambda \otimes \id"]\\
\ccinfty(\G) \stackrel[]{s, \id}{\otimes} M \arrow[r,"T^{-1}_M"]& \ccinfty(\G) \stackrel[]{r, \id}{\otimes} M \arrow[d,"\lambda \otimes \id"]
& \ccinfty(\G) \stackrel[]{r, \id}{\otimes} M \arrow[d,"\lambda \otimes \id"] \\
& M \arrow[r,"="] &M, 
\end{tikzcd}
\end{center}
where $ \pi $ is the projection onto the second factor. Observing that $ (\lambda \otimes \id)T^{-1}(f \otimes g) = f * g = \mu(f\otimes g) $ is the convolution product, 
we obtain
$$
(\lambda \otimes \id)(\id \otimes \lambda \otimes \id) d_1^*(T^{-1}_M) (f \otimes g \otimes m) 
= \mu_M(\mu \otimes \id)(f \otimes g \otimes m)
$$
for all $ f,g \in \ccinfty(\G) $ and $ m \in M $. Applying the coaction identity $ d_2^*(T^{-1}_M)d_0^*(T^{-1}_M) = d_1^*(T^{-1}_M) $ we conclude that $ \mu_M $ turns $ M $ into a 
left $ \DG $-module. Since both $ \lambda $ and $ q_M $ are surjective we see that $ \mu_M(\DG \otimes M) = M $, and it follows that the resulting module structure is essential. 

The construction of $ \mu_M $ is compatible with morphisms, that is, if $ f: M \rightarrow N $ is a morphism of $ \ccinfty(\G) $-comodules then $ f $ is $ \G $-equivariant for the corresponding $ \G $-module structures. This completes the proof of the following result.

\begin{lemma} \label{comodtomod}
Let $ M $ be a $ \ccinfty(\G) $-comodule. Then $ M $ becomes a $ \G $-module via the action $ \mu_M $ defined above. This assignment defines a functor $ B: \ccinfty(\G) \operatorname{-Comod} \rightarrow \G \operatorname{-Mod} $. 
\end{lemma}

Upon combining Lemma \ref{modtocomod} and Lemma \ref{comodtomod} we are now ready to establish the correspondence between $ \G $-modules and $ \ccinfty(\G) $-comodules. 

\begin{prop} \label{theorem: unification theorem} 
Let $ \G $ be an ample groupoid. The constructions described above implement an isomorphism of categories between the category $ \G \operatorname{-Mod} $ of $ \G $-modules and the category $ \ccinfty(\G) \operatorname{-Comod} $ of $ \ccinfty(\G) $-comodules. 
\end{prop}

\begin{proof}
We have already constructed functors $ A: \G \operatorname{-Mod} \rightarrow \ccinfty(\G) \operatorname{-Comod} $ and $ B: \ccinfty(\G) \operatorname{-Comod} \rightarrow \G \operatorname{-Mod} $, and it suffices to show that the compositions $ B A $ and $ A B $ equal the identity on $ \G \operatorname{-Mod} $ and $ \ccinfty(\G) \operatorname{-Comod} $, respectively. 

For the $ \G $-module $ \DG $, the $ \G $-module $ BA(\DG) $ is obtained by passing from the canonical map $ T: \ccinfty(\G) \stackrel{r,r}{\otimes} \ccinfty(\G)\rightarrow \ccinfty(\G)\stackrel{s,r}{\otimes} \ccinfty(\G) $ to the $ \G $-module structure $ \DG \times \DG \rightarrow \DG $ given by 
$$
(\lambda \otimes \id)T^{-1} q(f \otimes g)(\alpha) = (\lambda \otimes \id)T^{-1}(f \otimes g)(\alpha) = \sum_{\beta \in \G^{r(\alpha)}} f(\beta) g(\beta^{-1} \alpha)
$$
for $f,g\in\DG$ and $\alpha\in \G$.
This coincides with the left multiplication action of $ \DG $ on itself. It follows that $ BA(\DG) = \DG $ as $ \G $-modules. Using this, combined with the canonical identification $ \DG \otimes_{\DG} M \cong M $, we see that we have in fact $ BA(M) = M $ as $ \G $-modules for every $ M \in \G \operatorname{-Mod} $. 

Now let $ N $ be a $ \ccinfty(\G) $-comodule with canonical map $ T_N:\ccinfty(\G)\stackrel{r,\id}{\otimes} N \rightarrow \ccinfty(\G)\stackrel{s,\id}{\otimes} N $. In order to describe the canonical map $ T_{AB(N)} $ we start from the module structure $ \mu_{B(N)} = (\lambda \otimes \id) T^{-1}_N q_N $ on $ B(N) $. Then, by construction, 
\begin{align*}
T^{-1}_{AB(N)}(\chi_U\otimes n) &= \chi_U \otimes \mu_{B(N)}(\chi_{U} \otimes n) \\ 
&=\chi_U\otimes (\lambda \otimes \id) T^{-1}_N q_N(\chi_U \otimes n) \\ 
&=\chi_U\otimes (\lambda \otimes \id) T^{-1}_N(\chi_U \otimes n)
\end{align*} 
for any compact open bisection $ U \subseteq \G $ and $ n \in N $. 
Let us write $ T^{-1}_N(\chi_U \otimes n) = \sum_i \chi_{U_i} \otimes n_i $ for compact open bisections $ U_i \subseteq \G $ and elements $ n_i \in N $. Since $ T^{-1}_N $ is $ \ccinfty(\G) $-linear we have
\begin{align*}
\sum_i\chi_{U_i}\otimes n_i &= T^{-1}_N(\chi_U\otimes n) \\
&= T^{-1}_N(\chi_U \chi_U \otimes n) \\
&= \chi_U \cdot (\sum_i\chi_{U_i}\otimes n_i) \\
&= \sum_i \chi_U \chi_{U_i} \otimes n_i \\
&= \sum_i \chi_{U \cap U_i} \otimes n_i.
\end{align*}
Hence we can assume without loss of generality that $ U_i \subseteq U $ for all $ i $. 
With this understood, we calculate 
\begin{align*} 
T^{-1}_{AB(N)}(\chi_U\otimes n) &=\chi_U\otimes (\lambda \otimes \id)(\sum_i\chi_{U_i}\otimes n_i) \\
&= \sum_i \chi_U\otimes \lambda(\chi_{U_i}) n_i \\
&= \sum_i \chi_U \cdot \lambda(\chi_{U_i}) \otimes n_i \\
&= \sum_i \chi_U\cdot \chi_{r(U_i)} \otimes n_i \\ 
&= \sum_i \chi_{U\cap U_i} \otimes n_i \\
&= T^{-1}_N(\chi_U\otimes n), 
\end{align*}
using that the action of $ \ccinfty(\G^{(0)}) $ on the left tensor factor is given by the range map in the penultimate step. Since $ U $ and $ n $ were arbitrary we conclude $ T_{AB(N)} = T_N $ as required. 
\end{proof}

Theorem \ref{theorem: unification theorem} can be phrased as saying that a $ \G $-module $ M $ is the same thing as the data of an essential $ \ccinfty(\G^{(0)}) $-module structure on $ M $ together with a $ \ccinfty(\G) $-linear map $ T_M:\ccinfty(\G)\stackrel{r,\id}{\otimes} M \rightarrow \ccinfty(\G)\stackrel{s,\id}{\otimes} M $ satisfying the coaction identity. This is how this result will be used in the sequel.

\subsection{Tensor products}

The category $ \G \operatorname{-Mod} $ admits a natural tensor product operation. More precisely, if $ M, N $ are $ \G $-modules then they are in particular essential $ \ccinfty(\G^{(0)}) $-modules by restriction of the action, and we obtain a natural $ \G $-module structure on the tensor product $ M \otimes_{\ccinfty(\G^{(0)})} N $. 

We shall describe this action using the comodule picture developed in the previous section. According to Theorem \ref{theorem: unification theorem}, 
it suffices to define a $ \ccinfty(\G) $-linear isomorphism 
$$ 
T_{M\otimes N}: \ccinfty(\G) \stackrel{r,\id}{\otimes}(M\otimes_{\ccinfty(\G^{(0)})} N) \rightarrow \ccinfty(\G)\stackrel{s,\id}{\otimes}(M\otimes_{\ccinfty(\G^{(0)})} N),
$$
satisfying the coaction identity.
We start with the case $ M = N = \DG $ and consider the homeomorphism 
$$
t_{\G \times \G}: \G \stackrel{}{\times_{s,r}}(\G\stackrel{}{\times_{r,r}} \G) \rightarrow \G \stackrel{}{\times_{r,r}}(\G \stackrel{}{\times_{r,r}} \G), 
\quad t_{\G \times \G}(\alpha,(\beta,\gamma)) = (\alpha,(\alpha\beta, \alpha\gamma)). 
$$
With the notation used in the proof of Lemma \ref{modtocomod} we obtain natural identifications 
\begin{align*} 
\G^{(2)} \times_{d_0, \pi}(\G \stackrel{}{\times_{s,r}}(\G \stackrel{}{\times_{r,r}} \G)) \cong \G^{(2)} \stackrel{}{\times_{v_2,r}}(\G\stackrel{}{\times_{r,r}} \G), \\
\G^{(2)} \times_{d_0, \pi}(\G \stackrel{}{\times_{r,r}}(\G\stackrel{}{\times_{r,r}} \G)) \cong \G^{(2)} \stackrel{}{\times_{v_1,r}}(\G\stackrel{}{\times_{r,r}} \G), \\ 
\G^{(2)} \times_{d_1, \pi}(\G \stackrel{}{\times_{s,r}}(\G \stackrel{}{\times_{r,r}} \G)) \cong \G^{(2)} \stackrel{}{\times_{v_2,r}}(\G\stackrel{}{\times_{r,r}} \G), \\
\G^{(2)} \times_{d_1, \pi}(\G \stackrel{}{\times_{r,r}}(\G\stackrel{}{\times_{r,r}} \G)) \cong \G^{(2)} \stackrel{}{\times_{v_0,r}}(\G\stackrel{}{\times_{r,r}} \G), \\
\G^{(2)} \times_{d_2, \pi}(\G \stackrel{}{\times_{s,r}}(\G \stackrel{}{\times_{r,r}} \G)) \cong \G^{(2)} \stackrel{}{\times_{v_1,r}}(\G\stackrel{}{\times_{r,r}} \G), \\
\G^{(2)} \times_{d_2, \pi}(\G \stackrel{}{\times_{r,r}}(\G\stackrel{}{\times_{r,r}} \G)) \cong \G^{(2)} \stackrel{}{\times_{v_0,r}}(\G\stackrel{}{\times_{r,r}} \G), 
\end{align*} 
and one checks that the induced maps 
\begin{align*}
\id \times_{d_0,\pi} t_{\G \times \G}&:\G^{(2)} \stackrel{}{\times_{v_2,r}}(\G\stackrel{}{\times_{r,r}} \G)\rightarrow \G^{(2)} \stackrel{}{\times_{v_1,r}}(\G\stackrel{}{\times_{r,r}} \G),\\
\id \times_{d_1,\pi} t_{\G \times \G}&:\G^{(2)} \stackrel{}{\times_{v_2,r}}(\G\stackrel{}{\times_{r,r}} \G) \rightarrow \G^{(2)} \stackrel{}{\times_{v_0,r}}(\G\stackrel{}{\times_{r,r}} \G),\\
\id \times_{d_2,\pi} t_{\G \times \G}&:\G^{(2)} \stackrel{}{\times_{v_1,r}}(\G\stackrel{}{\times_{r,r}} \G) \rightarrow \G^{(2)} \stackrel{}{\times_{v_0,r}}(\G\stackrel{}{\times_{r,r}} \G), 
\end{align*}
are given by 
\begin{align*}
(\id \times_{d_0,\pi} t_{\G \times \G})(\alpha,\beta,\gamma, \delta) &= (\alpha, \beta, (\beta \gamma, \beta \delta)), \\
(\id \times_{d_1,\pi} t_{\G \times \G})(\alpha,\beta,\gamma, \delta) &= (\alpha, \beta, (\alpha\beta\gamma, \alpha\beta \delta)),\\
(\id \times_{d_2,\pi} t_{\G \times \G})(\alpha,\beta,\gamma, \delta) &=(\alpha, \beta, (\alpha\gamma, \alpha \delta)), 
\end{align*}
respectively. It follows that the linear isomorphism
\begin{align*}
T_{\DG \otimes \DG}: \ccinfty(\G)\stackrel{r,r}{\otimes} (\ccinfty(\G)\stackrel{r,r}{\otimes}\ccinfty(\G))&\rightarrow \ccinfty(\G)\stackrel{s,r}{\otimes}(\ccinfty(\G)\stackrel{r,r}{\otimes}\ccinfty(\G)), 
\\
T_{\DG \otimes \DG}(f)(\alpha, \beta, \gamma) &=f(\alpha,\alpha\beta,\alpha\gamma)
\end{align*}
induced by transposition of $ t_{\G \times \G} $ satisfies 
$$ 
d_0^*(T_{\DG \otimes \DG}) d_2^*(T_{\DG \otimes \DG}) = d_1^*(T_{\DG \otimes \DG}) 
$$ 
as required. We also note that $ T_{\DG \otimes \DG} $ is right $ \DG \otimes \DG $-linear with respect to the action on the second and third tensor factors. 

Now let $ M, N $ be arbitrary $ \G $-modules. Then the canonical isomorphism 
$$
(\DG \otimes \DG) \otimes_{\DG \otimes \DG} (M\otimes N) \rightarrow M\otimes N
$$
induces a linear isomorphism 
$$
(\DG \stackrel{r,r}{\otimes} \DG)\otimes_{\DG \otimes \DG} (M\otimes N) \rightarrow M \otimes_{\ccinfty(\G^{(0)})} N. 
$$
Using the left $ \DG $-action on $ \DG \stackrel{r,r}{\otimes} \DG = \DG \otimes_{\ccinfty(\G^{(0)})} \DG $ corresponding to the $ \ccinfty(\G) $-comodule structure constructed above we obtain the desired $ \G $-module structure on $ M \otimes_{\ccinfty(\G^{(0)})} N $ through this identification. 

For calculations it is useful to know how the tensor product $ \G $-module structure looks in terms of compact open bisections.

\begin{lemma} \label{lemma:how diagonal action works}
Let $ M, N $ be $ \G $-modules. Then the $ \DG $-module structure on the tensor product  $ M \otimes_{\ccinfty(\G^{(0)}} N $ satisfies 
$$ 
\chi_U \cdot(m \otimes n) = (\chi_U \cdot m) \otimes (\chi_U \cdot n), 
$$ 
for any compact open bisection $ U $ of $ \G $. 
\end{lemma} 

\begin{proof}
In view of the construction of the tensor product action it suffices to consider the case $ M = N = \DG $. Given compact open bisections $ U,V,W $ of $ \G $ and  $ \alpha, \beta, \gamma \in \G $ such that $ r(\alpha) = r(\beta) = r(\gamma) $ we compute
\begin{align*}
T^{-1}_{\DG \otimes \DG}(\chi_U\otimes \chi_V\otimes \chi_W)(\alpha,\beta,\gamma) &= (\chi_U\otimes \chi_V\otimes \chi_W)(\alpha,\alpha^{-1}\beta,\alpha^{-1}\gamma)\\
&=\chi_U(\alpha) \chi_V(\alpha^{-1}\beta)\chi_W(\alpha^{-1}\gamma).
\end{align*}
It follows that $ T^{-1}_{\DG \otimes \DG}(\chi_U\otimes \chi_V\otimes \chi_W) $ is the characteristic function of the set $ U \times U V \times UW $. 
Applying the integration map $ \lambda $ to the first tensor factor therefore gives 
$$
\chi_U \cdot (\chi_V \otimes \chi_W) = \chi_{UV} \otimes \chi_{UW} = (\chi_U \cdot \chi_V) \otimes (\chi_U \cdot \chi_W),  
$$
and since characteristic functions of compact open bisections span $ \DG $ this yields the claim. 
\end{proof}

We will always view the tensor product $ M \otimes_{\ccinfty(\G^0)} N $ of $ \G $-modules $ M, N $ as a $ \G $-module with the action defined above. Let us summarise our discussion as follows. 

\begin{prop} \label{tensorcategory}
The category $ \G\mbox{-}\operatorname{Mod} $ with the tensor product operation defined above is a monoidal category.  
\end{prop}

\begin{proof}
Let $ M, N $ and $ P $ be $ \G $-modules. Then we have a canonical isomorphism 
$$ 
(M\otimes_{\ccinfty(\G^{(0)})} N) \otimes_{\ccinfty(\G^{(0)})} P 
\cong M \otimes_{\ccinfty(\G^{(0)})} N \otimes_{\ccinfty(\G^{(0)})} P 
\cong M \otimes_{\ccinfty(\G^{(0)})} (N\otimes_{\ccinfty(\G^{(0)})} P) 
$$ 
of $ \ccinfty(\G^{(0)}) $-modules, and using Lemma \ref{lemma:how diagonal action works} we see that the action on either side is given by 
$$
\chi_U \cdot(m \otimes n \otimes p) = (\chi_U \cdot m) \otimes (\chi_U \cdot n) \otimes (\chi_U \otimes p) 
$$
for $ m \in M, n \in N, p \in P $. We conclude that the above isomorphism is $ \G $-equivariant, thus giving the required associativity constraint. 

The tensor unit is given by the trivial $ \G $-module $ \ccinfty(\G^{(0)}) $, with the action 
$$ 
f \cdot h = \lambda(f s^*(h)) 
$$ 
for $ f \in \DG, h \in \ccinfty(\G^{(0)}) $. It is straightforward to check that this turns $ \ccinfty(\G^{(0)}) $ indeed into a $ \G $-module. To check the behaviour of the trivial $ \G $-module under taking tensor products with other $ \G $-modules recall that the canonical identification $ \DG \cong  \ccinfty(\G^{(0)}) \otimes_{\ccinfty(\G^0)} \DG $ is given by $ \phi(h \otimes f) = r^*(h) f $. Then if $ U \subseteq \G $ is a bisection we calculate
\begin{align*}
\phi(\chi_U \cdot h \otimes \chi_U \cdot f)(\alpha) &= (\chi_U \cdot h)(r(\alpha)) (\chi_U * f)(\alpha) \\
&= \lambda(\chi_U s^*(h))(r(\alpha)) \sum_{\beta \in \G^{r(\alpha)}} \chi_U(\beta) f(\beta^{-1} \alpha) \\
&= \sum_{\gamma \in \G^{r(\alpha)}} \chi_U(\gamma) h(s(\gamma)) \sum_{\beta \in \G^{r(\alpha)}} \chi_U(\beta) f(\beta^{-1} \alpha) \\
&= \sum_{\beta \in \G^{r(\alpha)}} \chi_U(\beta) h(s(\beta)) f(\beta^{-1} \alpha) \\
&= \sum_{\beta \in \G^{r(\alpha)}} \chi_U(\beta) \phi(h \otimes f)(\beta^{-1} \alpha) \\
&=(\chi_U \cdot \phi(h \otimes f))(\alpha)
\end{align*}
for $ f \in \DG $ and $ h \in \ccinfty(\G^{(0)}) $. 
Due to Lemma \ref{lemma:how diagonal action works}, it follows easily that the canonical identifications $ M \otimes_{\ccinfty(\G^{(0)})} \ccinfty(\G^{(0)}) \cong M \cong  \ccinfty(\G^{(0)}) \otimes_{\ccinfty(\G^0)} M $ are $ \G $-equivariant for every $ \G $-module $ M $. 

With these structures in place, the axioms for a monoidal category are readily verified. 
\end{proof}

\subsection{$ \G $-algebras}

Our main object of study in this paper are $ \G $-algebras over an ample groupoid $ \G $ in the following sense. 

\begin{dfn} \label{Galgebra}
A $ \G $-algebra is a $ \G $-module $ A $ together with a $ \G $-equivariant linear map $ m: A \otimes_{\ccinfty(\G^{(0)})} A \rightarrow A $, written $ m(a \otimes b) = ab $, such that $ (ab)c = a(bc) $ for all $ a,b,c \in A $. 
\end{dfn}

This can be phrased equivalently as saying that a $ \G $-algebra is a non-unital algebra object in the monoidal category $ \G\mbox{-}\operatorname{Mod} $. 

By definition, if $ A, B $ are $ \G $-algebras then a $ \G $-equivariant algebra homomorphism $ \phi: A \rightarrow B $ is a $ \G $-equivariant linear map of the underlying $ \G $-modules such that $ \phi(ab) = \phi(a) \phi(b) $ for all $ a,b \in A $. 

Let us list some examples and  constructions with $ \G $-algebras. 

\subsubsection{Algebras of functions}

We obtain examples of commutative $ \G $-algebras from actions of $ \G $ on totally disconnected locally compact spaces.  

\begin{prop} \label{functionsgmodule}
Let $ X $ be a totally disconnected locally compact $ \G $-space. Then $ \ccinfty(X) $ with pointwise multiplication is a $ \G $-algebra in a natural way. 
\end{prop}

\begin{proof}
Let us denote the anchor map of $ X $ by $ \pi $. Then we obtain a $ \ccinfty(\G^{(0)}) $-linear map 
$$ 
T_{\ccinfty(X)}: \ccinfty(\G) \stackrel{r,\pi}{\otimes} \ccinfty(X) \rightarrow \ccinfty(\G) \stackrel{s,\pi}{\otimes} \ccinfty(X) 
$$ 
by setting 
$$
T_{\ccinfty(X)}(f)(\alpha, x) = f(\alpha, \alpha \cdot x). 
$$
An argument analogous to the discussion before Lemma \ref{modtocomod} shows that this map satisfies the coaction identity, thus turning $ \ccinfty(X) $ into a $ \ccinfty(\G) $-comodule. Hence the claim follows from Proposition \ref{theorem: unification theorem}. 
\end{proof} 

The construction in Proposition \ref{functionsgmodule} is compatible with tensor products. More precisely, if $ X, Y $ are totally disconnected locally compact $ \G $-spaces, with anchor maps $\pi_1$ and $\pi_2$, respectively, then 
the canonical map 
$$
\ccinfty(X) \otimes_{\ccinfty(\G^{(0)})} \ccinfty(Y) \rightarrow \ccinfty(X \times_{\pi_1,\pi_2} Y) 
$$
is an isomorphism of $ \G $-algebras, compare Proposition \ref{prop:isomorphisms ccinfty}.  

\subsubsection{Algebras associated with pairings} 

Another class of examples of $ \G $-algebras comes from $ \G $-modules equipped with $ \G $-equivariant pairings in the following sense. 

\begin{dfn} \label{defpairing}
Let $ E $ be a $ \G $-module. A $ \G $-equivariant pairing on $ E $ is a $ \G $-equivariant linear map $ h: E \otimes_{\ccinfty(\G^{(0)})} E \rightarrow \ccinfty(\G^{(0)}) $. 
\end{dfn}

We may equivalently view a $ \G $-equivariant pairing as in Definition \ref{defpairing} as a $ \ccinfty(\G^{(0)}) $-bilinear map $ h: E \times E \rightarrow \ccinfty(\G^{(0)}) $ such that 
$$
h(\chi_U \cdot e, \chi_U \cdot f) = \chi_U \cdot h(e,f)
$$
for all $ e,f \in E $ and all compact open bisections $ U \subseteq \G $. If $ E,F $ are $ \G $-modules equipped with $ \G $-equivariant pairings $ h_E, h_F $, respectively, then the tensor product pairing 
$$
h_{E \otimes F}: E \otimes_{\ccinfty(\G^{(0)})} F \otimes_{\ccinfty(\G^{(0)})} E \otimes_{\ccinfty(\G^{(0)})} F \rightarrow \ccinfty(\G^{(0)})
$$
given by 
$$
h_{E \otimes F}(e_1 \otimes f_1, e_2 \otimes f_2) = h_E(e_1, e_2) h_F(f_1, f_2) 
$$
is a $ \G $-equivariant pairing on $ E \otimes_{\ccinfty(\G^{(0)})} F $. 

Given a $ \G $-equivariant pairing on $ E $, consider the tensor product 
$$ 
\K(E) = E \otimes_{\ccinfty(\G^{(0)})} E 
$$ 
as a $ \G $-module with the diagonal action. Then $ \K(E) $ becomes a $ \G $-algebra with the multiplication defined by
\begin{align*} 
(e_1 \otimes f_1) (e_2 \otimes f_2) = e_1 \otimes h(f_1, e_2) f_2 = e_1 h(f_1, e_2) \otimes f_2
\end{align*} 
for $ e_1, e_2, f_1, f_2 \in E $. Note that the multiplication in $ \K(E) $ depends on the pairing $ h $, so it would be more accurate to write $ \K(E, h) $ for the resulting $ \G $-algebra. However, in the sequel the pairings we use will always be clear from the context. 

The most important example of a $ \G $-equivariant pairing is the \emph{regular pairing} 
$ \lambda: \DG \times\DG \rightarrow \ccinfty(\G^{(0)}) $ defined by  
\begin{align*}
\lambda(f,g)(x) = \sum_{\alpha\in \G^x} f(\alpha) g(\alpha),
\end{align*}
for $ x \in \G^{(0)} $. We will simply write $ \K_\G = \K(\DG) $ for the associated $ \G $-algebra. 

\subsubsection{$ \ccinfty(X) $-algebras} 

Given an algebra $ A $ we write $ ZM(A) $ for the center of the multiplier algebra of $ A $. 

\begin{dfn}
Let $ X $ be a totally disconnected locally compact space. A $ \ccinfty(X) $-algebra is an algebra $ A $ together with an essential homomorphism $ \ccinfty(X) \rightarrow M(A) $ which takes values in $ ZM(A) $. 
\end{dfn}

Let us record the following observation, in analogy to the study of groupoid actions in the $ C^* $-algebra setting. 

\begin{lemma} 
Let $ \G $ be an ample groupoid and let $ A $ be a $ \G $-algebra. If the multiplication in $ A $ is nondegenerate then $ A $ is canonically a $ \ccinfty(\G^{(0)}) $-algebra. \end{lemma} 

\begin{proof}
The underlying $ \ccinfty(\G^{(0)}) $-module structure of $ A $ determines an essential homomorphism $ \iota: \ccinfty(\G^{(0)}) \rightarrow M(A) $ such that $ \iota(f) a = f \cdot a $ for $ a \in A $. If $ m \in M(A) $ is an arbitrary multiplier and $ f \in \ccinfty(\G^{(0)}) $ we 
get $ a m \iota(f) b = am f\cdot b = f \cdot (a m b) = a f \cdot (m b) = a \iota(f) m b $ for all $ a, b \in A $, and using nondegeneracy of the multiplication we conclude 
that $ \iota(f) $ is contained in $ ZM(A) $. 
\end{proof}

We also note that if the ample groupoid $ \G = \G^{(0)} = X $ is obtained by viewing a totally disconnected locally compact space $ X $ as a groupoid then every $ \ccinfty(X) $-algebra is canonically a $ \G $-algebra. 

\subsubsection{Unitarisation} 

By a unital $ \G $-algebra object we shall mean a unital algebra object in the category of $ \G $-modules, that is, a $ \G $-algebra $ A $ in the sense of Definition \ref{Galgebra} together with a $ \G $-equivariant homomorphism $ u: \ccinfty(\G^{(0)}) \rightarrow A $ such that $ u(f) a = a u(f) = f \cdot a $ for $ f \in  \ccinfty(\G^{(0)}), a \in A $. 
A $ \G $-equivariant algebra homomorphism between unital $ \G $ algebra objects is called unital if it commutes with the unit maps in the obvious way. 

A basic example of a unital $ \G $-algebra object is $ A = \ccinfty(\G^{(0)}) $ with the trivial $ \G $-action and $ u = \id $. This example shows already that a unital $ \G $-algebra object does not need to have a unit element in general. For this reason, we speak of unital $ \G $-algebra objects and not of unital $ \G $-algebras. 

\begin{dfn}
Let $ A $ be a $ \G $-algebra. The $ \G $-unitarisation of $ A $ is defined as 
$$
A^+= A \oplus \ccinfty(\G^{(0)})
$$ 
viewed as $ \G $-module with the given action on $ A $ and the trivial action on $ \ccinfty(\G^{(0)}) $, and the multiplication given by
$$
(a, f) \cdot(b, g) = (ab + g \cdot a + f \cdot b, fg)
$$ 
for $ a,b \in A $ and $ f,g \in \ccinfty(\G^{(0)}) $. Here the dot product denotes the $ \ccinfty(\G^{(0)}) $-structure on $ A $ induced from its $ \G $-module structure.    
\end{dfn}

Let us write $ \Alg_\G(A,B) $ for the set of all $ \G $-equivariant algebra homomorphisms between $ \G $-algebras $ A, B $. If $ A, B $ are unital $ \G $-algebra objects then we denote by $ \Alg^u_\G(A,B) $ the set of all unital $ \G $-equivariant algebra homomorphisms. 

With this notation in place, let us show that the $ \G $-unitarisation of a $ \G $-algebra satisfies the following universal property. 

\begin{lemma}
Let $ A $ be a $ \G $-algebra. Then $ A^+ $ is a unital $ \G $-algebra object, and there is a natural bijection
$$
\Alg^u_\G( A^+, B) \cong \Alg_\G(A, B) 
$$
for every unital $ \G $-algebra object $ B $. 
\end{lemma}

\begin{proof}
By construction, the embedding into the first summand is a $ \G $-equivariant homomorphism $ \ccinfty(\G^{(0)}) \rightarrow A^+ $ which turns $ A^+ $ into a unital $ \G $-algebra object. 

If $ \phi: A^+ \rightarrow B $ is a $ \G $-equivariant unital algebra homomorphism then the restriction of $ \phi $ to $ A \subset A^+ $ defines a $ \G $-equivariant algebra homomorphism $ \phi: A \rightarrow B $. Conversely, if $ \psi: A \rightarrow B $ is a $ \G $-equivariant algebra homomorphism then we obtain a unital $ \G $-equivariant algebra homomorphism $ \psi^+: A^+ \rightarrow B $ by setting 
$$
\psi^+(a, f) = \psi(a) + u(f) 
$$
where $ u: \ccinfty(\G^{(0)}) \rightarrow B $ is the unit morphism. 

It is straightforward to check that these assignments determine mutually inverse bijections as claimed. 
\end{proof}

\subsubsection{Crossed products} 

The algebraic crossed product $ A \rtimes \G $ of a $ \G $-algebra $ A $ can be defined in a similar way as for discrete groups. 

More precisely, we define $ A \rtimes \G = A \otimes_{\ccinfty(\G^{(0)})} \DG $ as a vector space, with the left action of $ \ccinfty(\G^{(0)}) $ on $ \DG $ given via the range map. The multiplication in $ A \rtimes \G $ is determined by 
$$
(a \otimes \chi_U)(b \otimes g) = a \chi_U \cdot b \otimes \chi_U * g 
$$
for $ a, b \in A $, compact open bisections $ U \subseteq \G $ and $ g \in \DG $. It is straightforward to check that this turns $ A \rtimes \G $ into an algebra. 

The crossed product admits algebra homomorphisms $ i_A: A \rightarrow M(A \rtimes \G) $ and $ i_\G: \DG \rightarrow M(A \rtimes \G) $ such that $ i_A(a) i_\G(f) = a \otimes f $ 
for all $ a \in A, f \in \DG $. 
By definition, a covariant representation of $ (A, \G) $ on an algebra $ B $ is a pair of essential homomorphisms $ \phi: A \rightarrow M(B), \pi: \DG \rightarrow M(B) $ such that $ \phi(f \cdot a) \pi(g) = \phi(a) \pi(f * g) $ for all $ f \in \ccinfty(\G^{(0)}), a \in A, g \in \DG $ and 
$$
\phi(\chi_U \cdot a) \pi(\chi_U) = \pi(\chi_U) \phi(a) 
$$
for all compact open bisections $ U \subseteq \G $ and $ a \in A $. Clearly the maps $ i_A, i_{\G} $ define a covariant representation of $ (A, \G) $ on $ A \rtimes \G $. 

\begin{prop} \label{crossedproductuniversal}
Let $ A $ be a $ \G $-algebra. The crossed product $ A \rtimes \G $ is universal for covariant representations of $ (A, \G) $, that is, for every algebra $ B $ and every covariant representation $ (\phi, \pi) $ of $ (A,\G) $ on $ B $ there exists a unique essential algebra homomorphism $ \psi: A \rtimes \G \rightarrow M(B) $ such that $ \phi = \psi i_A $ 
and $ \pi = \psi i_\G $. 
\end{prop}

\begin{proof} 
We define $ \psi(a \otimes f) = \phi(a) \pi(f) $. This gives a well-defined linear map $ A \rtimes \G \rightarrow M(B) $ by the first part of the covariance condition. Using the second part of the covariance condition we calculate 
\begin{align*}
\psi(a \otimes \chi_U) \psi(b \otimes \chi_V) &= \phi(a) \pi(\chi_U) \phi(b) \pi(\chi_V) \\
&= \phi(a) \phi(\chi_U \cdot b) \pi(\chi_U) \pi(\chi_V) = \psi((a \otimes \chi_U)(b \otimes \chi_V))
\end{align*}
for $ a,b \in A $ and all compact open bisections $ U,V \subseteq \G $, and it follows that $ \psi $ is a homomorphism. It is straightforward to check that $ \psi $ is essential, satisfying $ \phi = \psi i_A, \pi = \psi i_\G $, and since $ A \rtimes \G = i_A(A) i_\G(\DG) $ it is uniquely determined. 
\end{proof}

\subsection{Anti-Yetter-Drinfeld modules}

Consider the set of loops 
$$
\G_{ad} = \{\alpha\in \G \mid r(\alpha)=s(\alpha)\} 
$$ 
in $ \G $. This is a subgroupoid of  $\G $ with the same objects, often called the isotropy subgroupoid or inertia groupoid. 
Noting that $ \G_{ad} = (s,r)^{-1}(\Delta) $, where 
$ \Delta \subseteq \G^{(0)} \times \G^{(0)} $ is the diagonal, we see that $ \G_{ad} $ is a closed subset of $ \G $, and thus a totally disconnected locally compact space with the subspace topology.

Due to Lemma \ref{lemma:restriction is epimorphism}, a function $ f \in \ccinfty(\G_{ad}) $ can be represented as a linear combination of restrictions to $ \G_{ad} $ of characteristic functions of compact open bisections of $ \G $. 
We will often view the characteristic function $ \chi_U $ of a compact open bisection $ U \subseteq \G $ as an element of $ \ccinfty(\G_{ad}) $ in this way. 

It is straightforward to check that $ \G_{ad} $ is a $ \G $-space with the adjoint action 
$$
\alpha \cdot \beta = \alpha \beta \alpha^{-1}
$$
for $ \alpha \in \G, \beta \in \G_{ad} $, with anchor map $ \pi = r = s: \G_{ad} \rightarrow \G^{(0)} $. According to Lemma \ref{functionsgmodule} we therefore obtain a natural $ \G $-algebra structure on $ \ccinfty(\G_{ad}) $. We will write $ \OG $ for this $ \G $-algebra in the sequel. 

\begin{dfn}
A $ \G $-anti-Yetter-Drinfeld module is a $ \G $-module $ M $ which is also an essential $ \OG $-module such that the module action induces a $ \G $-equivariant linear map $ \OG \otimes_{\ccinfty(\G^{(0)})} M \rightarrow M $. A morphism of $ \G $-anti-Yetter-Drinfeld modules is a $ \G $-equivariant linear map which is also $ \OG $-linear. 
\end{dfn}

The terminology used here is motivated from the theory of quantum groups, compare \cite{VOIGT_equivariantperiodiccyclicquantum}. 
A basic example of a $ \G $-anti-Yetter-Drinfeld module is obtained by considering $ M = \OG \otimes_{\ccinfty(\G^{(0)})} E $ for a $ \G $-module $ E $, with the diagonal action of $ \G $ and the action of $ \OG $ by multiplication on the first tensor factor. 

One can view $ \G $-anti-Yetter-Drinfeld modules equivalently as essential modules over the crossed product $ A(\G) = \OG \rtimes \G $. This observation is a special case of Proposition \ref{crossedproductuniversal}, note that both $ \OG $ and $ \DG $ are subalgebras of the multiplier algebra $ M(A(\G)) $, and composition with the inclusion maps gives the asserted equivalence. We will frequently use this identification between $ \G $-anti-Yetter-Drinfeld modules and $ A(\G) $-modules in the sequel. 

Given a $ \G $-anti-Yetter-Drinfeld module $ M $ our goal is to define a certain canonical automorphism $ T = T_M: M \rightarrow M $. We start with $ M = A(\G) = \OG \otimes_{\ccinfty(\G^{(0)})} \ccinfty(\G) = \ccinfty(\G_{ad} \times_{\pi, r} \G) $, in which case we define $ T $ by the formula 
$$
T(f)(\alpha, \beta) = f(\alpha, \alpha \beta)
$$
for $ f \in \ccinfty(\G_{ad} \times_{\pi, r} \G) $, in a similar way as in the discussion of $ \ccinfty(\G) $-comodules. 

\begin{lemma} \label{Tlemma}
The map $ T: A(\G) \rightarrow A(\G) $ defined above is an isomorphism of $ A(\G) $-bimodules. 
\end{lemma} 

\begin{proof}
It is clear that $ T $ is bijective with inverse given by $ T^{-1}(f)(\alpha, \beta) = f(\alpha, \alpha^{-1} \beta) $. The left and right $ \OG $-module structures on $ A(\G) $ are given by 
$$
(h \cdot f)(\alpha, \beta) = h(\alpha) f(\alpha, \beta), \qquad 
(f \cdot h)(\alpha, \beta) = h(\beta^{-1}\alpha \beta) f(\alpha, \beta)
$$
for $ h \in \OG, f \in A(\G) $. We thus obtain 
\begin{align*}
(h \cdot T(f))(\alpha, \beta) &= h(\alpha) f(\alpha, \alpha\beta) 
= (h \cdot f)(\alpha, \alpha \beta) = T(h \cdot f)(\alpha, \beta) 
\end{align*}
and 
\begin{align*}
(T(f) \cdot h)(\alpha, \beta) &= h(\beta^{-1}\alpha \beta) f(\alpha, \alpha\beta) 
= (f \cdot h)(\alpha, \alpha \beta) = T(f \cdot h)(\alpha, \beta), 
\end{align*}
which shows that $ T $ is both left and right $ \OG $-linear. 
For $ g \in \DG $ we have
\begin{align*}
(g \cdot T(f))(\alpha, \beta) &= \sum_{\gamma \in \G^{r(\alpha)}} g(\gamma) f(\gamma^{-1}\alpha\gamma, \gamma^{-1} \alpha \gamma \gamma^{-1}\beta) 
= T(g \cdot f)(\alpha, \beta) 
\end{align*}
and 
\begin{align*}
(T(f) \cdot g)(\alpha, \beta) &= \sum_{\gamma \in \G_{s(\beta)}} 
f(\alpha, \alpha \beta \gamma^{-1}) g(\gamma) 
= T(f \cdot g)(\alpha, \beta), 
\end{align*}
and it follows that $ T $ is left and right $ \DG $-linear. Combining these observations yields the claim. 
\end{proof}

In view of Lemma \ref{Tlemma} we can define $ T_M: M \rightarrow M $ for a $ \G $-anti-Yetter-Drinfeld module $ M $ by 
$$
T_M = m_M(T \otimes \id)m_M^{-1}, 
$$
where $ m_M: A(\G) \otimes_{A(\G)} M \rightarrow M $ is the canonical isomorphism. This defines an automorphism of the $ \G $-anti-Yetter-Drinfeld module $ M $. 

\begin{lemma} \label{automaticcommutation}
Let $ \G $ be an ample groupoid and let $ \phi: M \rightarrow N $ be a morphism of $ \G $-anti-Yetter-Drinfeld modules. Then $ T_N \phi = \phi T_M $. 
\end{lemma}

\begin{proof}
Using the canonical isomorphisms and $ m_N (\id \otimes \phi) = \phi m_M $ we compute
\begin{align*}
T_N \phi &= m_N (T \otimes \id) (\id \otimes \phi) m_M^{-1} 
= m_N (\id \otimes \phi) (T \otimes \id) m_M^{-1} = \phi T_M 
\end{align*}
as required. 
\end{proof}

In calculations it is useful to have an explicit formula for the action of the canonical automorphism. We will only need this for $ \G $-anti-Yetter-Drinfeld modules of the form $ M = \OG \otimes_{\ccinfty(\G^{(0)})} E $ for a $ \G $-module $ E $, in which case we get 
$$
T_M(\chi_U \otimes e) = \chi_U \otimes \chi_{U^{-1}} \cdot e
$$
for any compact open bisection $ U \subseteq \G $ and $ e \in E $.

\section{Equivariant periodic cyclic homology} \label{section:epch}

In this section we define bivariant equivariant periodic cyclic homology with respect to an ample groupoid $ \G $ which is fixed throughout. 

\subsection{Projective systems}

Following Cuntz and Quillen \cite{CUNTZ_QUILLEN_excision}, in the sequel we will not only consider $ \G $-algebras but work in greater generality with pro-$ \G $-algebras. This s convenient when it comes to discussing quasifreeness, even if one is only interested in $ \G $-algebras. Accordingly, we shall consider projective systems of $ \G $-modules 
and $ \G $-anti-Yetter-Drinfeld modules. We briefly discuss these concepts here and fix our notation, and refer 
to \cite{CUNTZ_QUILLEN_excision}, \cite{VOIGT_thesis}, \cite{VOIGT_equivariantperiodiccyclichomology} for more details. 

To an additive category $ \mathcal{C} $ one associates the pro-category $ \pro(\mathcal{C}) $ of projective systems over $ \mathcal{C} $ as follows. A projective system 
over $ \mathcal{C} $ consists of a directed index set $ I $, objects $ V_i $ for all $ i \in I $, and morphisms $ p_{ij}: V_j \rightarrow V_i $ for all $ j \geq i $. These 
morphisms are assumed to satisfy $ p_{ij} p_{jk} = p_{ik} $ if $ k \geq j \geq i $. By definition, the objects of $ \pro(\mathcal{C}) $ are projective systems over $ \mathcal{C} $. 
The space of morphisms in $ \pro(\mathcal{C}) $ between $ (V_i)_{i \in I} $ and $ (W_j)_{j \in J} $ is defined by 
\begin{equation*}
\Mor((V_i),(W_j)) = \varprojlim_{j} \varinjlim_i
\Mor_\mathcal{C}(V_i,W_j), 
\end{equation*}
where the limits are taken in the category of abelian groups. Any object of $ \mathcal{C} $ can be viewed as a constant pro-object, and this gives rise to a fully faithful embedding 
of $ \mathcal{C} $ into $ \pro(\mathcal{C}) $. Moreover the category $ \pro(\mathcal{C}) $ is again additive in a canonical way. 

If we apply these general constructions to the category of $ \G $-modules we obtain the category $ \pro(\G \operatorname{-Mod}) $ of pro-$ \G $-modules.
A morphism in $ \pro(\G \operatorname{-Mod}) $ will be called a $ \G $-equivariant pro-linear
map. Similarly, we have the category of pro-$ \G $-anti-Yetter-Drinfeld modules. 

If $ \mathcal{C} $ is an additive monoidal category with tensor product functor $ \mathcal{C} \times \mathcal{C} \rightarrow \mathcal{C} $ one defines the tensor 
product $ V \otimes W $ of pro-objects $ V = (V_i)_{i \in I} $ and $ W =  (W_j)_{j \in J} $ by
\begin{equation*}
(V_i)_{i \in I} \otimes (W_j)_{j \in J} =
(V_i \otimes W_j)_{(i,j) \in I \times J},
\end{equation*}
equipped with the obvious structure maps. Here $ I \times J $ is ordered using the product ordering. 
With this tensor product the category $ \pro(\mathcal{C}) $ becomes additive monoidal and the embedding functor $ \mathcal{C} \rightarrow \pro(\mathcal{C}) $ is monoidal. The existence of a tensor product in $ \pro(\mathcal{C}) $ yields a natural notion of algebra objects and algebra homo\-morphisms in this category, which we refer to as pro-algebras and their homomor\-phisms. 

According to Proposition \ref{tensorcategory} the category $ \G\operatorname{-Mod} $ is additive monoidal. 
By definition, a pro-$ \G $-algebra is an algebra object in $ \pro(\G \operatorname{-Mod}) $, in the same way as $ \G $-algebras are algebra objects in $ G \operatorname{-Mod} $, compare section \ref{section:dgmodules}. An algebra homo\-morphism $ f: A \rightarrow B $ in $ \pro(\G \operatorname{-Mod}) $ will simply be called a $ \G $-equivariant homomor\-phism. Clearly, every $ \G $-algebra is a pro-$ \G $-algebra in a canonical way. 

Occasionally we will encounter unital pro-$ \G $-algebras. The $ \G $-unitarisation $ A^+ $ of a pro-$ \G $-algebra $ A $ is defined in the same way as for $ \G $-algebras. Similarly, the construction of crossed products for $ \G $-algebras carries over to pro-$ \G $-algebras.

Let $ \mathcal{C} $ be any additive category. An admissible extension in $ \pro(\mathcal{C}) $ consists of objects $ K, E, Q \in \pro(\mathcal{C}) $ and morphisms $ \iota: K \rightarrow E, \pi: E \rightarrow Q $ in $ \pro(\mathcal{C}) $, such that there exist morphisms $ \sigma: Q \rightarrow E, \rho: E \rightarrow K $ in $ \pro(\mathcal{C}) $ satisfying $ \rho \iota = \id, \pi \sigma = \id $ and $ \iota \rho + \sigma \pi = \id $. In other words, we require that $ E $ decomposes into a direct sum of 
$ K $ and $ Q $ in $ \pro(\mathcal{C}) $. We will  write
\begin{equation*}
   \xymatrix{
     K \;\; \ar@{>->}[r]^{\iota} & E \ar@{->>}[r]^{\pi} & Q 
     }
\end{equation*}
or $ 0 \rightarrow K \rightarrow E \rightarrow Q 
\rightarrow 0 $ for an admissible extension. 

Let $ K, E $ and $ Q $ be pro-$ \G $-algebras. An admissible extension of pro-$ \G $-algebras is an admissible extension 
$ 0 \rightarrow K \rightarrow E \rightarrow Q 
\rightarrow 0 $ in $ \pro(\G \operatorname{-Mod}) $ such that $ \iota $ and $ \pi $ are $ \G $-equivariant algebra homomorphisms. In connection with excision we will only need that the underlying objects of $ K,E, Q $ in $ \pro(\ccinfty(\G^{(0)}) \operatorname{-Mod}) $ form an admissible extension, or equivalently, the splitting $ \sigma $ for the homomorphism $ \pi $ is a morphism of pro-$ \ccinfty(\G^{(0)}) $-modules but not necessarily $ \G $-equivariant. 

We will frequently describe morphisms between pro-objects by writing down explicit formulas involving ``elements'' in subsequent sections. This can be justified by observing that these formulas encode identities between abstractly defined mor\-phisms. 

\subsection{Paracomplexes} \label{secpara}

Let us next review the notion of a paracomplex in a para-additive category, compare \cite{VOIGT_equivariantperiodiccyclichomology}.  

\begin{dfn} 
A para-additive category is an additive category $ \mathcal{C} $ 
together with a natural isomorphism $ T $ of the identity functor $ \id: \mathcal{C} \rightarrow \mathcal{C} $.   
\end{dfn}

In other words, we are given invertible morphisms $ T(M): M \rightarrow M $ for all 
objects $ M \in \mathcal{C} $ such that $ \phi  T(M) = T(N) \phi $ for all morphisms $ \phi: M \rightarrow N $. In the sequel we will 
simply write $ T $ instead of $ T(M) $. 

Any additive category is para-additive by setting $ T = \id $. The para-additive categories that are relevant for us are the category of $ \G $-anti-Yetter-Drinfeld modules and its pro-category, compare Lemma \ref{automaticcommutation}. 

\begin{dfn} 
Let $ \mathcal{C} $ be a para-additive category. 
A paracomplex $ C = C_0 \oplus C_1 $ in $ \mathcal{C} $ is a given by objects $ C_0 $ and $ C_1 $ in $ \mathcal{C} $ together with  morphisms $ \partial_0: C_0 \rightarrow C_1 $ and $ \partial_1: C_1 \rightarrow C_0 $ such that
$$ 
\partial^2 = \id - T, 
$$
where the differential $ \partial: C \rightarrow C_1 \oplus C_0 \cong C $ is the composition of $ \partial_0 \oplus \partial_1 $
with the canonical flip map. A chain map $ \phi: C \rightarrow D $ between two paracomplexes is a morphism from $ C $ to $ D $ that commutes with
the differentials.
\end{dfn}

In general it does not make sense to speak about the homology of a paracomplex, but one can give meaning to the statement that two paracomplexes are homotopy equivalent, by using the standard formulas for chain homotopies. 

The paracomplexes we are interested in arise from paramixed complexes in the following sense. 

\begin{dfn} 
Let $ \mathcal{C} $ be a para-additive category. A paramixed complex $ M $ in $ \mathcal{C} $ is a sequence of objects $ M_n $ together with differentials $ b $ of degree $ -1 $ and $ B $ of degree $ + 1 $ satisfying
$ b^2 = 0 $, $ B^2 = 0 $ and
\begin{equation*}
[b,B] = bB + Bb = \id - T.
\end{equation*}
\end{dfn}

One can define Hochschild homology of a paramixed complex in the usual way since the 
Hochschild operator $ b $ satisfies $ b^2 = 0 $. However, we will not study Hochschild homology in this paper and focus entirely on periodic cyclic homology. 

\subsection{Equivariant differential forms}

For a pro-$ \G $-algebra $ A $ we define the space of differential forms over $ A $ by $ \Omega^0_{\G^{(0)}}(A) = A $ and the iterated tensor products over $ \ccinfty(\G^{(0)}) $ given by 
$$
\Omega^n_{\G^{(0)}}(A) = A^+ \otimes_{\ccinfty(\G^{(0)})} A^{\otimes_{\ccinfty(\G^{(0)})} n} \cong A^{\otimes_{\ccinfty(\G^{(0)})} n + 1} \oplus A^{\otimes_{\ccinfty(\G^{(0)})} n}
$$
for all $ n > 0 $, where we recall that $ A^+ $ denotes the $ \G $-unitarisation defined in section \ref{section:dgmodules}. 
We always view $ \Omega^n_{\G^{(0)}}(A) $ as a pro-$ \G $-module with the diagonal action. Elements of $ \Omega^n_{\G^{(0)}}(A) $ contained in the first summand of the above decomposition, that is, tensors of the form $ a^0 \otimes a^1 \otimes \cdots \otimes a^n $ with $ a^0, a^1, \dots, a^n \in A $, will usually be written $ a^0da^1 \cdots da^n $. Similarly, elements in the second summand  will be denoted $ da^1 \cdots da^n $. If we want to cover both cases at the same time we shall write $ \langle a^0 \rangle da^1 \cdots da^n $, following \cite{MEYER_thesis}. 

The pro-$ \G $-module $ \Omega^n_{\G^{(0)}}(A) $ becomes an $ A $-$ A $-bimodule object in $ \pro(\G \operatorname{-Mod}) $ with the left $ A $-module structure 
$$
a \cdot (\langle a^0 \rangle da^1 \cdots da^n) = a\langle a^0 \rangle da^1 \cdots da^n, 
$$ 
and the right $ A $-module determined by the Leibniz rule, that is, 
\begin{align*}
(\langle a^0 \rangle da^1\cdots \, da^n)\cdot a &= \langle a^0 \rangle da^1 \cdots d(a^na) + \sum_{j = 1}^{n - 1}(-1)^{n - j} \langle a^0 \rangle da^1\cdots d(a^ja^{j + 1}) \cdots da^n da \\
&\quad +(-1)^{n} \langle a^0 \rangle a^1da^2\cdots da^nda. 
\end{align*}
We note that the $ A $-bimodule $ \Omega^n_{\G^{(0)}}(A) $ can be identified with the $ n $-fold tensor product of $ \Omega^1_{\G^{(0)}}(A) $ over $ A $ in the category of pro-$ \ccinfty(\G^{(0)}) $-modules. From this description it is easy to see that 
$$
\Omega_{\G^{(0)}}(A) = \bigoplus_{n = 0}^\infty \Omega^n_{\G^{(0)}}(A)
$$
is a pro-$ \G $-algebra in a natural way. 
Let us also define the $ \ccinfty(\G^{(0)}) $-linear operator $ d: \Omega^n_{\G^{(0)}}(A) \rightarrow \Omega^{n+1}_{\G^{(0)}}(A)$ by
$$
d(a^0da^1 \cdots da^n) = da^0 da^1 \cdots da^n, \qquad d(da^1\cdots da^n)=0
$$ 
for $ a^0, a^1, \dots, a^n \in A $. By construction one has $ d^2 = 0 $. 

Next we introduce the $ \G $-equivariant differential forms over $ A $. We set $ \Omega^0_\G(A) = \OG \otimes_{\ccinfty(\G^{(0)})} A $ and 
$$
\Omega^n_\G(A) = \OG \otimes_{\ccinfty(\G^{(0)})} \Omega^n_{\G^{(0)}}(A)
$$ 
for $ n > 0 $, where we recall that $ \OG $ is the $ \G $-algebra of functions on $ \G_{ad} $ with the adjoint action. This becomes a pro-$ \G $-module with the diagonal action, and a pro-$ \OG $-module with the multiplication action on the first tensor factor. These actions turn $ \Omega^n_\G(A) $ into a pro-$ \G $-anti-Yetter-Drinfeld module. We write $ \Omega_\G(A) $ for the direct sum of all $ \Omega^n_\G(A) $ for $ n \geq 0 $. 

We need several operators on equivariant differential forms. Let us define $ d_\G: \Omega^n_\G(A) \rightarrow \Omega^{n + 1}_\G(A) $ by 
$$
d_\G(f\otimes \omega) = f\otimes d\omega,
$$ 
where $ f \in \ccinfty(\G_{ad}) $ and $ \omega \in \Omega^n_{\G^{(0)}}(A) $.
The equivariant Hochschild operator $ b_\G: \Omega^n_\G(A) \rightarrow \Omega^{n - 1}_\G(A) $ is defined by $ b_\G = 0 $ for $ n = 0 $ and 
\begin{align*}
b_\G(f\otimes \omega da) &= (-1)^{n - 1} (f \otimes \omega a - (\id \otimes \mu)(T(f \otimes a) \otimes \omega) 
\end{align*} 
for $ n > 1 $, where $ \mu $ denotes multiplication in $ \Omega_{\G^{(0)}}(A) $ and $ T $ is the canonical map. If $ U \subseteq \G $ is a compact open bisection then we can write this in the form 
\begin{align*}
b_\G(\chi_U \otimes \omega da) &= (-1)^{n - 1} (\chi_U \otimes \omega a - \chi_U \otimes (\chi_{U^{-1}} \cdot a)\omega), 
\end{align*} 
or explicitly, 
\begin{align*}
b_\G(\chi_U \otimes \langle a^0 \rangle da^1&\cdots da^n) = \chi_U \otimes \langle a^0 \rangle a^1da^2\cdots da^n \\
&+ \sum_{j = 1}^{n - 1}(-1)^j \chi_U \otimes \langle a^0 \rangle da^1\cdots d(a^ja^{j+1}) \cdots da^n \\
&+ (-1)^n \chi_U \otimes (\chi_{U^{-1}} \cdot a^n)\langle a^0 \rangle d a^1 \cdots d a^{n-1}
\end{align*}
for $ \langle a^0 \rangle a^1da^2\cdots da^n \in \Omega^n_{\G^{(0)}}(A) $. 
With these formulas at hand one verifies in the same way as in the group equivariant case that $ b_\G^2 = 0 $. 

Starting from $ d_\G $ and $ b_\G $ we define the equivariant Karoubi operator $ \kappa_\G $ by
$$
\kappa_\G = \id - (b_\G d_\G + d_\G b_\G),
$$ 
and the equivariant Connes operator $ B_\G $ by 
$$
B_\G = \sum_{j = 0}^n \kappa_\G^j d_\G
$$  
on $ \Omega^n_{\G^{(0)}}(A) $. Using $ d_\G^2 = 0 $ we see that $ d_\G $ and $ \kappa_{\G} $ commute and that $ B_\G^2 = 0 $.
Explicitly, for $ n > 0 $ and a compact open bisection $ U \subseteq \G $ we obtain 
$$
\kappa_\G(\chi_U \otimes \omega da) = (-1)^{n - 1} \chi_U \otimes (\chi_{U^{-1}} \cdot da) \omega,
$$
and for $ n = 0 $ we get $ \kappa_\G(\chi_U \otimes a) = \chi_U \otimes \chi_{U^{-1}} \cdot a $. For $ B_\G $ one calculates
$$
B_\G(\chi_U \otimes a^0da^1\cdots da^n) = \sum_{i = 0}^n (-1)^{ni} \chi_U \otimes (\chi_{U^{-1}} \cdot (da^{n+1-i} \cdots da^n)) da^0 \cdots da^{n - i}.
$$
The canonical operator $ T $ on $ \Omega_\G(A) $ is given by 
$$
T(\chi_U \otimes \omega) = \chi_U \otimes \chi_{U^{-1}} \cdot \omega,  
$$
compare the discussion at the end of section \ref{section:dgmodules}. All the operators introduced above are morphisms of pro-$ \G $-anti-Yetter-Drinfeld modules, and therefore commute with $ T $ by Lemma \ref{automaticcommutation}. 

The following formulas are verified in a similar way as in the group case, see \cite[Lemma 7.2]{VOIGT_equivariantperiodiccyclichomology}.

\begin{lemma} \label{lemma: properties of bounderies}
The following relations hold on $ \Omega^n_\G(A) $. 
\begin{enumerate}
\item $ \kappa_\G^{n + 1} d_\G = T d_\G $
\item $ \kappa_\G^n = T + b_{\G} \kappa_\G^n d_\G $
\item $ b_\G \kappa_\G^n = b_\G T $
\item $ \kappa_\G^{n + 1} = (\id - d_\G b_\G) T $
\item $ (\kappa_\G^{n + 1} - T)(\kappa_\G^n - T) = 0 $
\item $ B_\G b_\G + b_\G B_\G = \id - T. $
\end{enumerate}
\end{lemma}

Note that the final formula of Lemma \ref{lemma: properties of bounderies} can equivalently be phrased as saying that $ \Omega_\G(A) $ together with the operators $ b_\G, B_\G $ defines a paramixed complex in the category of pro-$ \G $-anti-Yetter-Drinfeld modules. 

\subsection{Quasifree pro-$ \G $-algebras} 

Let us next discuss the main definitions and facts related to quasifreeness. For background information and more details we refer 
to \cite{CUNTZ_QUILLEN_algebraextensions}, \cite{MEYER_thesis}, \cite{VOIGT_thesis}, \cite{VOIGT_equivariantperiodiccyclichomology}. 

We endow the pro-$ \G $-module $ \Omega_{\G^{(0)}}(A) $ of differential forms over a pro-$ \G $-algebra $ A $ with the Fedosov product, defined by  
$$ 
\omega \circ \eta = \omega \eta -(-1)^m d\omega d\eta
$$ 
for forms $ \omega \in \Omega_{\G^{(0)}}^m(A), \eta \in \Omega_{\G^{(0)}}^n(A) $. By definition, the periodic tensor algebra $ \T A $  of $ A $ is the pro-$ \G $-algebra obtained as projective limit of the projective system $ \left(\T A /(\J A)^n\right)_{n \in \mathbb{N}} $, 
where $ T A /(\J A)^n = A \oplus \Omega_{\G^{(0)}}^2(A) \oplus \cdots \oplus \Omega_{\G^{(0)}}^{2n}(A) $ and the structure maps are the canonical projections. 
Similarly one defines the pro-$ \G $-algebra $ \J A $ as the projective limit of the projective 
system $ (\J A /(\J A)^n)_{n \in \mathbb{N}} $, where $ \J A /(\J A)^n = \Omega_{\G^{(0)}}^2(A) \oplus \cdots \oplus \Omega_{\G^{(0)}}^{2n}(A) $.

The natural projection from $ \T A$ to the first term of the projective system gives a $ \G $-equivariant homomorphism $ \tau_A: \T A \rightarrow A $. Moreover, for 
every $ n \in \mathbb{N} $ we have natural inclusions $ A \rightarrow A \oplus \Omega_{\G^{(0)}}^2(A) \oplus \cdots \oplus \Omega_{\G^{(0)}}^{2 n}(A) $, and these maps 
assemble to a $ \G $-equivariant pro-linear section $ \sigma_A $ for $ \tau_A $.
It follows that we obtain an admissible extension
\begin{equation*} \label{eqext}
   \xymatrix{
     \J A\;\; \ar@{>->}[r]^{\iota_A} & \T A \ar@{->>}[r]^{\tau_A} & A 
     }
\end{equation*}

If $ B $ is a pro-$ \G $-algebra we shall write $ m^n_B: B^{\otimes_{\ccinfty(\G^{(0)})} n} \rightarrow B $ for the iterated multiplication map, defined on the $ n $-fold balanced tensor product over $ \ccinfty(\G^{(0)}) $. 
A $ \G $-equivariant pro-linear map $ l: A \rightarrow B $ between pro-$ \G $-algebras $ A $ and $ B $ is called a $ \G $-lonilcur if its curvature $ \omega_l: A \otimes_{\ccinfty(\G^{(0)})} A \rightarrow B $, defined by $ \omega_l(a, b) = l(a b) - l(a) l(b) $, is locally nilpotent in the sense that for every $ \G $-equivariant pro-linear map $ f: B \rightarrow C $ with constant range $ C $ there exists $ n \in \mathbb{N} $ such that $ f m_B^n \omega_l^{\otimes_{\ccinfty(\G^{(0)})} n} = 0 $.

Every $ \G $-equivariant homomorphism $ l: A \rightarrow B $ between pro-$ \G $-algebras is a lonilcur since $ \omega_l(a, b) = l(a b) - l(a) l(b) = 0 $. 
More importantly, the splitting map $ \sigma_A: A \rightarrow \T A$ is a $ \G $-lonilcur for any $ A $, and the following result is verified in the same way as \cite[Proposition 3.3]{VOIGT_thesis}. 

\begin{prop} \label{prop: universal property of tensor algebra}
Let $ A,B $ be pro-$ \G $-algebras. The pro-$ \G $-algebra $ \T A $ and the $ \G $- equivariant pro-linear map $ \sigma_A: A \rightarrow \T A $ satisfy the following universal property. If $ l: A \rightarrow B $ is a $ \G $-lonilcur there exists a unique $ \G $-equivariant homomorphism $ [[l]]: \T A \rightarrow B $ such that the diagram
\begin{center}
\begin{tikzcd}
A\arrow[r,"\sigma_A"]\arrow[rd,"l"]& \T A \arrow[d,"{[[l]]}"]\\
& B
\end{tikzcd}
\end{center}
is commutative.
\end{prop}

The periodic tensor algebra plays a central role in the definition of quasifree pro-$ \G $-algebras. 

\begin{dfn}
A pro-$ \G $-algebra $ R $ is called quasifree if there exists a $ \G $-equivariant splitting homomorphism $ R \rightarrow \T R $ for the canonical projection $ \tau_R $.  
\end{dfn}

Let us list without proof a number of equivalent characterisations of the class of quasifree pro-$ \G $-algebras.

\begin{thm}\label{thm: characterization of quasi-free algebras}
Let $ R $ be a pro-$ \G $-algebra. Then the following conditions are equivalent:
\begin{enumerate}
\item $ R $ is quasifree.
\item There exists a $ \G $-equivariant pro-linear map $\phi: R \rightarrow \Omega_{\G^{(0)}}^2(R) $ satisfying
$$
\phi(x y)=\phi(x) y+x \phi(y)-d x d y
$$
for all  $x, y \in R $.
\item There exists a $ \G $-equivariant pro-linear map $ \nabla: \Omega_{\G^{(0)}}^1(R) \rightarrow \Omega_{\G^{(0)}}^2(R) $ satisfying
$$
\nabla(x \omega) = x \nabla(\omega), \quad \nabla(\omega x) = \nabla(\omega) x-\omega d x
$$
for all $ x \in R $ and $ \omega \in \Omega_{\G^{(0)}}^1(R) $.
\item The $ R $-bimodule $ \Omega_{\G^{(0)}}^1(R) $ is projective in $ \pro(\G \operatorname{-Mod}) $.
\end{enumerate}
\end{thm}

We note that the list in Theorem \ref{thm: characterization of quasi-free algebras} can be extended with further equivalent characterisations of quasifreeness in the same way as 
in \cite[Theorem 6.5]{VOIGT_equivariantperiodiccyclichomology}. 

\begin{lemma} \label{trivialquasifree}
The trivial $ \G $-algebra $ \ccinfty(\G^{(0)}) $ is quasifree.
\end{lemma}

\begin{proof}
Let $ f \in \ccinfty(\G^{(0)}) $ and define $ \phi(f) = 2 f d \chi_U d\chi_U - df d\chi_U $, where $ \chi_U $ is the characteristic function of a compact open 
subset $ U \subseteq \G^{(0)} $ such that $ \chi_U f = f $. This does not depend on the choice of $ U $, and one checks that $ \phi $ satisfies condition (2) in 
Theorem \ref{thm: characterization of quasi-free algebras}. 
\end{proof}

\begin{prop} \label{TAquasifree}
Let $ A $ be any pro-$ \G $-algebra. Then the periodic tensor algebra $ \T A $ is quasifree.
\end{prop}

Proposition \ref{TAquasifree} is proved by constructing a $ \G $-equivariant splitting homomorphism for the canonical projection $ \T \T A \rightarrow \T A $, compare  \cite[Proposition 6.10]{VOIGT_equivariantperiodiccyclichomology}.

\subsection{The Hodge tower and the equivariant $ X $-complex}

Let us now return to the paramixed complex $ \Omega_\G(A) $ of $ \G $-equivariant differential forms over a pro-$ \G $-algebra $ A $. Following Cuntz and Quillen ~\cite{CUNTZ_QUILLEN_nonsingularity}, we define the $ n $-th level of the Hodge tower associated to $ \Omega_\G(A) $ by 
\begin{equation*}
\theta^n \Omega_\G(A) = \bigoplus_{j = 0}^{n - 1} \Omega^j_\G(A) \oplus \Omega_\G^n(A)/b(\Omega^{n + 1}_\G(A)).
\end{equation*}
It is easy to check that the operators $ d_\G $ and $ b_\G $ descend to $ \theta^n \Omega_\G(A) $, and hence the same is true for $\kappa_\G $ and $ B_\G $. Using the natural grading into even and odd forms we see that 
$ \theta^n \Omega_\G(A) $ together with the boundary operator $ B_\G + b_\G $ becomes a pro-paracomplex of $ \G $-anti-Yetter-Drinfeld modules. 

For $ m \geq n $ there exists a natural chain map $ \theta^m \Omega_\G(A) \rightarrow 
\theta^n\Omega_\G(A) $ given by the obvious projection. By definition, the Hodge tower $ \theta \Omega_\G(A) $ of $ A $ is the projective limit 
of the projective system $ (\theta^n \Omega_\G(A))_{n \in \mathbb{N}} $ obtained in this way. 

\begin{dfn} 
Let $ A $ be a pro-$ \G $-algebra. The equivariant $ X $-complex 
$ X_\G(A) $ of $ A $ is the pro-paracomplex $ \theta^1 \Omega_\G(A) $. Explicitly, we have  
\begin{equation*}
\xymatrix{
{X_\G(A) \colon \ }
{\Omega^0_\G(A)\;} \ar@<.5ex>@{->}[r]^-{\natural d_\G} &
{\;\Omega^1_\G(A)/b_\G(\Omega^2_\G(A))} 
\ar@<.5ex>@{->}[l]^-{b_\G} 
} 
\end{equation*}
where $ \natural: \Omega^1_\G(A) \rightarrow \Omega^1_\G(A)/ b_\G(\Omega^2_\G(A)) $ denotes the canonical projection. 
\end{dfn}

Let us point out that the equivariant $ X $-complex $ X_\G(A) $ is typically not a chain complex but only a paracomplex. A notable exception is the case when $ A = \ccinfty(\G^{(0)}) $ is the trivial $ \G $-algebra. 

\begin{lemma} \label{XA}
The equivariant $ X $-complex $ X_\G(\ccinfty(\G^{(0)})) $ of the trivial $ \G $-algebra $ \ccinfty(\G^{(0)}) $ identifies canonically with
$$
\xymatrix{
{\OG\;} \ar@<.5ex>@{->}[r]  &
{\;0} 
\ar@<.5ex>@{->}[l]
} 
$$
that is, it is equal to the trivial supercomplex $ \OG[0] $.
\end{lemma}

\begin{proof} 
By definition of the equivariant $ X $-complex, the even part of  $ X_\G(\ccinfty(\G^{(0)})) $ is given by $ \OG \otimes_{\ccinfty(\G^{(0)})} \ccinfty(\G^{(0)}) \cong \OG $. 

Every element in the odd part of $ X_\G(\ccinfty(\G^{(0)})) $ can be represented as a linear combination of terms of the form $ \chi_U \otimes d\chi_V $ and $ \chi_U \otimes \chi_V d\chi_V $ for compact open bisections $ U \subseteq \G $ and compact open subsets $ V \subseteq \G^{(0)} $. Moreover, the canonical map $ T $ of $ \OG \otimes_{\ccinfty(\G^{(0)})} \ccinfty(\G^{(0)}) \cong \OG $ equals the identity, which implies that the Hochschild operator $ b_\G: \Omega^2_\G(\ccinfty(\G^{(0)})) \rightarrow \Omega^1_\G(\ccinfty(\G^{(0)})) $ satisfies
$$
b_\G(\chi_U \otimes \langle \chi_V \rangle d\chi_V d \chi_V) = - \chi_U \otimes \langle \chi_V \rangle d(\chi_V) \chi_V + \chi_U \otimes \chi_V d\chi_V.   
$$
We therefore obtain 
\begin{align*}
\chi_U \otimes \chi_V d\chi_V &= \chi_U \otimes \chi_V d(\chi_V \chi_V) \\
&= \chi_U \otimes \chi_V d(\chi_V) \chi_V + \chi_U \otimes \chi_V d(\chi_V) \\
&= 2 \chi_U \otimes \chi_V d\chi_V
\end{align*}
in $ X_\G(\ccinfty(\G^{(0)})) $ and hence $ \chi_U \otimes \chi_V d\chi_V = 0 $. 
Similarly, 
\begin{align*}
\chi_U \otimes d\chi_V &= \chi_U \otimes d(\chi_V \chi_V) \\
&= \chi_U \otimes d(\chi_V) \chi_V + \chi_U \otimes \chi_V d(\chi_V) \\
&= 2 \chi_U \otimes \chi_V d\chi_V = 0,  
\end{align*}
and we conclude that the odd part of $ X_\G(\ccinfty(\G^{(0)})) $ vanishes as claimed. 
\end{proof}

The central result regarding the equivariant $ X $-complex 
is the following theorem, compare \cite[Theorem 8.6]{VOIGT_equivariantperiodiccyclichomology}. 

\begin{thm} \label{homotopyeq} 
For any pro-$ \G $-algebra $ A $ the equivariant $ X $-complex $ X_\G(\T A) $ and the Hodge tower $ \theta\Omega_\G(A) $ are homotopy equivalent as pro-paracomplexes of $ \G $-anti-Yetter-Drinfeld modules. 
\end{thm}

The proof of Theorem \ref{homotopyeq} is a direct translation of the proof in the group equivariant case, building on the relations in Lemma \ref{lemma: properties of bounderies}. We will not spell out the details. 

\subsection{Bivariant equivariant periodic cyclic homology}

We are now ready to define bivariant equivariant periodic cyclic homology. 

\begin{dfn} \label{defhpg}
Let $ \G $ be an ample groupoid and let $ A $ and $ B $ be pro-$ \G $-algebras. The bivariant equivariant periodic cyclic homology of $ A $ and $ B $ is 
$$
HP_*^\G(A,B) = H_*(\Hom_{A(\G)}(X_\G(\T(A \otimes_{\ccinfty(\G^{(0)})} \K_\G)), X_\G( \T (B\otimes_{\ccinfty(\G^{(0)})} \K_\G)))).
$$
\end{dfn}

We point out that the Hom-complex in this definition is indeed an ordinary supercomplex, so that one can take its homology in the standard way. In order to explain this let us write $ \partial_A $ and $ \partial_B $ for the differentials of the equivariant $ X $-complexes in the source and the target, respectively. Then the differential in the Hom-complex is given by 
\begin{equation*}
\partial(\phi) = \phi \partial_A - (-1)^{|\phi|} \partial_B \phi
\end{equation*}
for a homogeneous element $ \phi $, and  we have
\begin{equation*}
\partial^2(\phi) =
\phi\, \partial_A^2 + (-1)^{|\phi|} (-1)^{|\phi| - 1} \partial_B^2\,\phi 
= \phi(\id - T) - (\id - T)\phi = T \phi - \phi T. 
\end{equation*}
Hence the relation $ \partial^2(\phi) = 0 $ is a consequence of the crucial commutation property in Lemma \ref{automaticcommutation}.

It follows directly from the definition that $ HP^\G_* $ is a bifunctor, contravariant in $ A $ and covariant in $ B $. We call 
$$
HP^\G_*(B) = HP_*^\G(\ccinfty(\G^{(0)}), B), \qquad HP^*_\G(A) = HP^\G_*(A, \ccinfty(\G^{(0)}))
$$
the $ \G $-equivariant periodic cyclic homology of $ B $, and the $ \G $-equivariant periodic cyclic cohomology of $ A $, respectively. 
Every $ \G $-equivariant algebra homomorphism $ f: A \rightarrow B $ induces naturally an element $ [f] \in HP^\G_0(A,B) $. We have an associative product 
$$ 
HP^\G_*(A,B) \times HP^\G_*(B,C) \rightarrow HP^\G_*(A,C), \quad (x,y) \mapsto x \cdot y 
$$ 
induced by the composition, and this generalises the composition of $ \G $-equivariant homomorphisms $ f: A \rightarrow B, g: B \rightarrow C $ in the sense that $ [f] \cdot [g] = [g \circ f] $. In particular, we obtain a natural ring structure on $ HP^\G_*(A,A) $ for every $ \G $-algebra $ A $ with unit element $ [\id] $. Note also that if $ \G^{(0)} $ is a singleton, or equivalently, if the groupoid $ \G $ is a discrete group, then the above constructions reduce to the theory developed in \cite{VOIGT_thesis}, \cite{VOIGT_equivariantperiodiccyclichomology}. 

Let us write $ \OG^\G \subset M(\OG) = \cinfty(\G_{ad}) $ for the subalgebra consisting of all functions which are constant along the orbits of the conjugation action. 
It is straightforward to check that the Hom-complex in Definition \ref{defhpg} is a complex of $ \OG^\G $-modules via the obvious multiplication action in the source or target, and hence the periodic cyclic homology groups $ HP^\G_*(A,B) $ are $ \OG^\G $-modules in a natural way. 
Using this module structure we can single out the contributions to $ HP^\G_* $ from the different conjugacy classes in $ \G_{ad} $. More precisely, for an ideal $ I $ in $ \OG^\G $ we define the \emph{localisation} of $ HP^G_*(A,B) $ at $ I $ by
$$
HP^\G_*(A,B)_I = HP^\G_*(A,B)/I \cdot HP^\G_*(A,B).
$$
The most obvious localisation is to take the ideal of all functions vanishing at the units $ \G^{(0)} \subseteq \G_{ad} $. We write $ HP^\G_*(A,B)_{[1]} $ for this localisation, following the notation introduced in  \cite{CRAINIC_MOERDIJK_homologyetalegroupoids} for groupoid homology. Let us note that there are analogous localisation structures in the Hochschild and cyclic homology of Steinberg algebras, see  \cite{ARNONE_CORTINAS_MUKHERJEE_homologysteinberg}. 

Localisation can help to analyse the structure of equivariant periodic cyclic homology. For instance, let us consider a discrete groupoid $ \G $ and show how to reduce the calculation of $ HP^\G_* $ to the group equivariant theory in this case. Recall that $ \G^x_x $ denotes the stabiliser group of $ x \in \G^{(0)} $ and 
write $ \mathfrak{m}_x \subseteq \ccinfty(\G^{(0)}) $ for the maximal ideal of all functions vanishing at $ x $. If $ A $ is a pro-$ \G $-algebra 
then $ A_x = A/\mathfrak{m}_x \cdot A $ is naturally a pro-$ \G_x^x $-algebra. 

\begin{prop}
Let $ \G $ be a discrete groupoid and let $ A, B $ be pro-$ \G $-algebras. Then we have a canonical isomorphism 
$$
HP^\G_*(A,B) \cong \bigoplus_{[x] \in \G \backslash \G^{(0)}} HP^{\G^x_x}_*(A_x, B_x),
$$
where each $ x $ is an arbitrary representative of the orbit $ [x] \in \G \backslash \G^{(0)} $. 
\end{prop}

\begin{proof}
Every discrete groupoid can be written as a disjoint union of transitive groupoids, and it is easy to see that this induces a corresponding direct sum decomposition of $ HP^\G_*(A,B) $ via localisation. Therefore it suffices to consider the case that $ \G $ is transitive. 

In this case, given any $ x \in \G^{(0)} $ one obtains an equivalence between the category of $ \G $-anti-Yetter-Drinfeld modules and the category of $ \G^x_x$-anti-Yetter-Drinfeld modules by sending a $ \G $-anti-Yetter-Drinfeld module $ M $ to $ \chi \cdot M $, where $ \chi \in M(\OG)) = \cinfty(\G_{ad}) $ denotes the characteristic function of $ \G_x^x $. 
Applying this functor, or rather its extension to the corresponding pro-categories, to the Hom-complex defining $ HP^\G_*(A,B) $
yields the desired isomorphism. 
\end{proof}

\section{Homotopy invariance, stability and excision} \label{section:properties}

In this section we show that $ HP^\G_* $ satisfies similar homological properties as equivariant $ KK $-theory, compare \cite{BOENICKE_PROIETTI_categoricalbc}.  

\subsection{Homotopy invariance} \label{section:homotopy}

We establish first that $ HP^\G_* $ is homotopy invariant with respect to smooth homotopies in both variables. 

Let $ A, B $ be pro-$ \G $-algebras. By definition, a smooth $ \G $-equivariant homotopy between $ \G $-equivariant algebra homomorphisms $ \phi_0, \phi_1: A \rightarrow B $ is a $ \G $-equivariant algebra homomorphism $ \Phi: A \rightarrow B[0,1] $ such that $ \Phi_i = ev_i \,\Phi $ equals $ \phi_i $ for $ i = 0,1 $. Here $ B[0,1] = B \otimes C^\infty[0,1] $ is the algebra of smooth functions on the unit interval with values in $ B $, equipped with the $ \G $-action which treats $ C^\infty[0,1] $ as a dummy term, and $ ev_i: B[0,1] \rightarrow B $ is evaluation at $ i $. 

Recall from Section \ref{section:epch} that $ \theta^n \Omega_\G(A) $ denotes the $ n $-th level of the Hodge tower, and that $ \theta^1 \Omega_\G(A) = X_\G(A) $ is the 
equivariant $ X $-complex. We have canonical projection maps $ \xi_n: \theta^n \Omega_\G(A) \rightarrow \theta^{n - 1} \Omega_\G(A) $ for all $ n \geq 1 $. 

\begin{lemma} \label{lemma: homotopy equivalence between theta^2 and X_G}
Let $ A $ be a quasifree pro-$ \G $-algebra. Then the map $ \xi_2: \theta^2 \Omega_\G(A) \rightarrow X_\G(A) $ is a homotopy equivalence of pro-paracomplexes of $ \G $-anti-Yetter-Drinfeld modules.
\end{lemma}

\begin{proof}
Since $A$ is quasifree there exists by Theorem \ref{thm: characterization of quasi-free algebras} a $ \G $-equivariant pro-linear 
map $ \nabla: \Omega^1_{\G^{(0)}}(A) \rightarrow \Omega^2_{\G^{(0)}}(A)$ such that
$$
\nabla(a \omega) = a \nabla(\omega) \text{  and  } \nabla(\omega a) = \nabla(\omega) a -\omega da
$$
for all $ a \in A $ and $\omega \in \Omega^1_{\G^{(0)}}(A) $. We extend $ \nabla $ to forms of higher degree by setting 
$$
\nabla\left(\langle a^0 \rangle da^1 \cdots da^n\right)=\nabla\left(\langle a^0 \rangle da^1\right) da^2 \cdots da^n.
$$ 
Then we have
$$
\nabla(a \omega)=a \nabla(\omega), \quad \nabla(\omega \eta)=\nabla(\omega) \eta+(-1)^{|\omega|} \omega d \eta
$$
for $ a \in A $ and $ \omega, \eta \in \Omega_{\G^{(0)}}(A) $. Moreover we set $ \nabla(a) = 0$ for $ a \in \Omega^0_{\G^{(0)}}(A) = A $.

One then obtains a map $ \nabla_\G: \Omega_\G^n(A) \rightarrow \Omega_\G^{n+1}(A)$ of pro-$ \G $-anti-Yetter-Drinfeld modules by setting 
$$
\nabla_\G(f\otimes \omega)=f\otimes \nabla(\omega).
$$
We will use $ \nabla_\G $ to construct an inverse of $ \xi_2 $ up to homotopy. 
Firstly, an explicit computation gives $ \left[b_\G, \nabla_\G\right] = \id $ on $ \Omega^n_{\G}(A) $ for $ n \geq 2 $. Since $\left[b_\G, \nabla_\G\right]$ commutes 
with $ b_\G $ this equality holds on $ b_\G(\Omega_\G^2(A)) \subseteq \Omega^1_\G(A) $ as well. As a consequence, we obtain a well-defined 
map $ \nu: X_\G(A) \rightarrow \theta^2 \Omega_\G(A) $ by setting $ \nu = \id - [\nabla_\G, B_\G + b_\G] $, noting that $ [\nabla_\G, B_\G] $ increases the degree of 
differential forms by $ 2 $. 

Using Lemma \ref{automaticcommutation} with the fact that $ \nabla_\G $ is a map of pro-$ \G $-anti-Yetter-Drinfeld modules one checks that $ \nu $ is a chain map with 
respect to $ \partial = B_\G + b_\G $. Explicitly, we have 
$$
\begin{array}{ll}
\nu = \id-\nabla_\G d & \text { on } \Omega_\G^0(A) \\
\nu = \id-\left[\nabla_\G, b_\G\right] = \id - b_\G \nabla_\G & \text { on } \Omega_\G^1(A)/b_\G\left(\Omega_\G^2(A)\right),
\end{array}
$$
and this implies $ \xi_2 \nu = \id $. Moreover, by construction 
$ \nu \xi_2 = \id - \left[\nabla_\G, B_\G + b_\G\right]$ is homotopic to the identity.
\end{proof}

Let again $ A, B $ be pro-$ \G $-algebras and let $ \Phi: A \rightarrow B[0,1] $ be a $ \G $-equivariant homotopy. The derivative $ \Phi': A \rightarrow B[0,1] $ is a $ \G $-equivariant pro-linear map satisying  $ \Phi'(ab) = \Phi'(a)\Phi(b) + \Phi(a)\Phi'(b) $ for all $ a,b \in A $. 

We define $ \eta: \Omega_\G^n(A) \rightarrow \Omega_\G^{n - 1}(B) $ by
$$
\eta(f\otimes a^0da^1 \cdots da^n) = \int_0^1 f \otimes \Phi_t(a^0)\Phi'_t(a^1) d\Phi_t(a^2) \cdots d \Phi_t(a^n) dt
$$
for $ n > 0 $ and $ \eta = 0 $ on $ \Omega_\G^0(A) $. Using the fact that $ \Phi' $ is a derivation with respect to $ \Phi $ one computes $ \eta b_\G + b_\G \eta = 0 $ 
on $ \Omega_\G^n(A) $ for all $ n \geq 0 $. 
In particular, we have  $ \eta b_\G(\Omega_\G^3(A)) \subseteq b_\G(\Omega_\G^2(B)) $, and hence we obtain a $ \G $-equivariant pro-linear 
map $ \eta: \theta^2 \Omega_\G(A) \rightarrow X_\G(B) $.

\begin{lemma}\label{lemma: induced homotopy between theta^2 and X_G}
Let $ \Phi: A \rightarrow B[0,1] $ be a $ \G $-equivariant homotopy between pro-$ \G $-algebras $ A, B $. Then we have $ X_\G(\Phi_1) \xi_2 - X_\G(\Phi_0) \xi_2 = \partial \eta + \eta \partial $, where $ \eta: \theta^2 \Omega_\G(A) \rightarrow X_\G(B) $ is the map defined above. Hence the chain maps $ X_\G(\Phi_t) \xi_2: \theta^2 \Omega_\G(A) \rightarrow X_\G(B)$ for $ t = 0,1 $ are homotopic.
\end{lemma} 

\begin{proof}
Recall that the boundary operator is $ \partial = B_\G + b_\G $. For $ j = 0 $ we have
$$
[\partial, \eta](f \otimes a)=\eta(f \otimes d a)=\int_0^1 f \otimes \Phi_t'(a) d t=f \otimes \Phi_1(a)-f \otimes \Phi_0(a) .
$$
For $ j = 1 $ we get
\begin{align*}
[\partial, \eta] & (\chi_U \otimes a^0 d a^1) = d_\G \eta(\chi_U \otimes a^0 d a^1) + \eta B_\G(\chi_U \otimes a^0 d a^1) \\
&= \int_0^1 (\chi_U \otimes d(\Phi_t(a^0) \Phi_t'(a^1)) + \chi_U \otimes \Phi_t'(a^0) d \Phi_t(a^1) \\
&\quad- \chi_U \otimes \Phi_t'(\chi_{U^{-1}} \cdot a^1) d \Phi_t(a^0)) dt \\
&=\int_0^1 b_\G(\chi_U \otimes d \Phi_t(a^0) d \Phi_t'(a^1))+\frac{\partial}{\partial t}(\chi_U \otimes \Phi_t(a^0) d \Phi_t(a^1)) dt
\end{align*}
for any compact open bisection $ U \subseteq \G $. 
Since the first term vanishes in $ X_\G(B) $ we conclude
$$
[\partial, \eta] (\chi_U \otimes a^0 d a^1) = \chi_U \otimes \Phi_1(a^0) d \Phi_1(a^1) - \chi_U \otimes \Phi_0(a^0) d \Phi_0(a^1). 
$$
Finally, on $ \Omega_\G^2(A) / b_\G(\Omega_\G^3(A))$ we have $ \partial \eta + \eta \partial = \eta b_\G + b_\G \eta = 0 $, with the last equality due to the calculation just before this Lemma.    
\end{proof}

We are now ready to state and prove the following result. 

\begin{thm}[Homotopy invariance] \label{thm: homotopy invariance}
Let $ A $ and $ B $ be pro-$ \G $-algebras and let $ \Phi: A \rightarrow B[0,1] $ be a $ \G $-equivariant homotopy. Then the elements $ \left[\Phi_0\right] $ and $ \left[\Phi_1\right] $ in $ HP_0^\G(A,B) $ are equal. More generally, if $ A $ is a quasifree pro-$ \G $-algebra then the elements $ \left[\Phi_0\right] $ and $ \left[\Phi_1\right] $ 
in $ H_0(\Hom_{A(\G)}(X_\G(A), X_\G(B))) $ are equal.    
\end{thm}

\begin{proof}
The second part of the Theorem follows directly by combining Lemma \ref{lemma: homotopy equivalence between theta^2 and X_G} and 
Lemma \ref{lemma: induced homotopy between theta^2 and X_G}.

In order to show that the first part of the Theorem can be viewed as a special case of the second, assume that $ \Phi: A \rightarrow B[0,1] $ is a $ \G $-equivariant homotopy. 
We tensor $ A $ and $ B $ with $ \K_\G $ to obtain a $ \G $-equivariant 
homotopy $ \Phi \otimes_{\ccinfty(\G^{(0)})} \K_\G: A \otimes_{\ccinfty(\G^{(0)})} \K_\G \rightarrow (B\otimes_{\ccinfty(\G^{(0)})} \K_\G)[0,1] $. 
Passing to the periodic tensor algebras we obtain a $ \G $-equivariant algebra homomorphism $ \T(\Phi \otimes_{\ccinfty(\G^{(0)})} \K_\G): \T(A \otimes_{\ccinfty(\G^{(0)})} \K_\G) \rightarrow \T((B \otimes_{\ccinfty(\G^{(0)})} \K_\G)[0,1]) $. 

Consider the $ \G $-equivariant pro-linear map
\begin{align*}
l: B\otimes_{\ccinfty(\G^{(0)})} \K_\G \otimes C^\infty[0,1]&\rightarrow \T(B \otimes_{\ccinfty(\G^{(0)})}\K_\G) \otimes C^\infty[0,1] \\
l(b\otimes T\otimes f) &= \sigma(b \otimes T) \otimes f,
\end{align*}
where $ \sigma: B \otimes_{\ccinfty(\G^{(0)})} \K_\G \rightarrow \T(B \otimes_{\ccinfty(\G^{(0)})} \K_\G) $ is the standard $ \G $-equivariant pro-linear splitting. 
Then $ l $ is a lonilcur, and we get an associated $ \G $-equivariant homomorphism 
$$
[[l]]: \T((B\otimes_{\ccinfty(\G^{(0)})}\K_\G)[0,1]) \rightarrow \T(B \otimes_{\ccinfty(\G^{(0)})} \K_\G)[0,1]
$$
by the universal property of the periodic tensor algebra from Proposition \ref{prop: universal property of tensor algebra}. Consider the $ \G $-equivariant homotopy 
$$
\Psi = [[l]] \T(\Phi\otimes_{\ccinfty(\G^{(0)})} \K_\G): \T(A\otimes_{\ccinfty(\G^{(0)})} \K_\G) \rightarrow \T(B\otimes_{\ccinfty(\G^{(0)})}\K_\G)[0,1]
$$
and note that $ \Psi_t = \T(\Phi_t \otimes_{\ccinfty(\G^{(0)})} \K_\G) $ for all $ t \in [0,1] $.
Since $ \T(A\otimes_{\ccinfty(\G^{(0)})} \K_\G) $ is quasifree we are now in the setting of the second part of the Theorem, and this concludes the proof. 
\end{proof}

We note that as an application of homotopy invariance one can show that $ X_\G(\T A) $ is homotopy equivalent to $ X_\G(A) $ if $ A $ is a quasifree pro-$ \G $-algebra, compare  \cite[Proposition 10.5]{VOIGT_equivariantperiodiccyclichomology}. 

\subsection{Stability} \label{section:stability}

Let us next show that $ HP^\G_* $ is stable in both variables with respect to tensoring with the algebra $ \K(E) $ of smoothing operators associated to a $ \G $-module $ E $ together with a $ \G $-equivariant pairing, compare section \ref{section:dgmodules}. 

We denote by $ h: E \otimes_{\ccinfty(\G^{(0)})} E \rightarrow \ccinfty(\G^{(0)}) $ the given $ \G $-equivariant pairing on $ E $ and define the twisted trace map $ ttr: \OG \otimes_{\ccinfty(\G^{(0)})} \K(E) \rightarrow \OG $ by setting 
$$
ttr(f \otimes e_1 \otimes e_2) = (\id \otimes h)(T(f \otimes e_2) \otimes e_1)
$$
for $ f \in \OG $ and $ e_1, e_2 \in E $. Explicitly, we have 
$$
ttr(\chi_U \otimes e_1 \otimes e_2) = \chi_U \otimes h(\chi_{U^{-1}} \cdot e_2 \otimes e_1) \in \OG \otimes_{\ccinfty(\G^{(0)})} \ccinfty(\G^{(0)}) \cong \OG
$$
for any compact open bisection $ U \subseteq \G $. 

\begin{lemma} \label{twistedtrace}
The twisted trace map introduced above satisfies
$$
ttr(\chi_U \otimes L_0 L_1) = ttr(\chi_U \otimes (\chi_{U^{-1}}\cdot L_1) L_0)
$$
for any compact open bisection $ U \subseteq \G $ and $ L_0, L_1 \in \K(E) $.  
\end{lemma}

\begin{proof}
It suffices to prove the claim for $ L_0 = e_1 \otimes e_2, L_1 = e_3 \otimes e_4 $ for any $ e_1, e_2, e_3, e_4 \in E $. In this case we obtain 
\begin{align*}
ttr(\chi_U \otimes L_0 L_1) = \chi_U \otimes  h(\chi_{U^{-1}} \cdot e_4 \otimes e_1) h(e_2 \otimes e_3) 
\end{align*}
and 
\begin{align*}
ttr(\chi_U \otimes (\chi_{U^{-1}}\cdot L_1) L_0) = \chi_U \otimes h(\chi_{U^{-1}} \cdot e_2 \otimes \chi_{U^{-1}} \cdot e_3) h(\chi_{U^{-1}} \cdot e_4 \otimes e_1) . 
\end{align*}
Moreover, note that $ h(\chi_{U^{-1}} \cdot e_2 \otimes \chi_{U^{-1}} \cdot e_3) = \chi_{U^{-1}} \cdot h(e_2 \otimes e_3) $ by $ \G $-equivariance of the pairing. 

It is therefore enough to observe that $ \chi_U \otimes f = \chi_U \otimes \chi_{U^{-1}} \cdot f $ for all $ f \in \ccinfty(\G^{(0)}) $, or equivalently, that the canonical map $ T $ of $ \OG \otimes_{\ccinfty(\G^{(0)})} \ccinfty(\G^{(0)}) $ equals the identity. 
\end{proof}

We say that a $ \G $-equivariant pairing $ h $ on a $ \G $-module $ E $ is $ \G $-admissible if there exists a $ \G $-equivariant linear embedding $ \ccinfty(\G^{(0)}) \rightarrow E $ such that the restriction of $ h $ to $ \ccinfty(\G^{(0)}) \subseteq E $ agrees with the canonical isomorphism $ \ccinfty(\G^{(0)}) \otimes_{\ccinfty(\G^{(0)})} \ccinfty(\G^{(0)}) \cong \ccinfty(\G^{(0)}) $. One checks that such an embedding induces a $ \G $-equivariant algebra homomorphism $ \iota: \ccinfty(\G^{(0)}) \rightarrow \K(E) $ in this case. 
More generally, if $ A $ is a $ \G $-algebra then we obtain a $ \G $-equivariant algebra homo\-morphism $ \iota_A: A \cong A \otimes_{\ccinfty(\G^{(0)})} \ccinfty(\G^{(0)}) \rightarrow A \otimes_{\ccinfty(\G^{(0)})} \K(E) $ by tensoring the identity on $ A $ with $ \iota: \ccinfty(\G^{(0)}) \rightarrow \K(E) $. 

\begin{thm} \label{StabLemma} 
Let $ A $ be a pro-$ \G $-algebra, and let $ E $ be a $ \G $-module equipped with a $ \G $-admissible $ \G $-equivariant bilinear pairing. Then the class 
$$
[\iota_A] \in H_0 \Hom_{A(\G)}(X_\G(\T A), X_\G(\T(A \otimes_{\ccinfty(\G^{(0)})} \K(E)))) 
$$
is invertible.
\end{thm}

\begin{proof}
We have to find an inverse for $ [\iota_A] $. First observe that the canonical $ \G $-equivariant linear map $ A \otimes_{\ccinfty(\G^{(0)})} \K(E) \rightarrow 
\T A \otimes_{\ccinfty(\G^{(0)})} \K(E) $ is a lonilcur and hence induces 
a $ \G $-equivariant homomorphism $ \lambda_A: \T(A \otimes_{\ccinfty(\G^{(0)})}
\K(E)) \rightarrow \T A \otimes_{\ccinfty(\G^{(0)})} \K(E) $.
Define $ tr_A: X_\G(\T A \otimes_{\ccinfty(\G^{(0)})} \K(E))
\rightarrow X_\G(\T A) $ by
\begin{equation*}
tr_A(f \otimes x \otimes L) = ttr(f \otimes L) \otimes x
\end{equation*}
and
\begin{align*}
tr_A(f \otimes (x_0 \otimes L_0) d(x_1 \otimes L_1)) &= 
ttr(f \otimes L_0 L_1) \otimes x_0 dx_1 \\
tr_A(f \otimes d(x_1 \otimes L_1)) &= 
ttr(f \otimes L_1) \otimes dx_1.
\end{align*}
Here, we use the twisted trace $ ttr $, see Lemma \ref{twistedtrace}. By construction, $ tr_A $ is a map of $ \G $-anti-Yetter-Drinfeld modules. 
We have
\begin{align*}
tr_A d_\G(f \otimes x \otimes L) &= tr_A(f \otimes d(x \otimes L)) \\
&= ttr(f \otimes L) \otimes dx \\
&= d_\G(ttr(f \otimes L) \otimes x) \\
&= d_\G tr_A(f \otimes x \otimes L), 
\end{align*}
and for a compact open bisection $ U \subseteq \G $ we calculate 
\begin{align*}
b_\G tr_A&(\chi_U \otimes (x_0 \otimes L_0) d(x_1\otimes L_1)) = b_\G(ttr(\chi_U \otimes L_0 L_1) \otimes x_0dx_1) \\
&= ttr(\chi_U \otimes L_0 L_1) \otimes (x_0x_1 - (\chi_{U^{-1}}\cdot x_1)x_0) \\
&= ttr(\chi_U \otimes L_0 L_1) \otimes x_0x_1 
- ttr(\chi_U \otimes (\chi_{U^{-1}}\cdot L_1) L_0) \otimes (\chi_{U^{-1}}\cdot x_1)x_0 \\ 
&= tr_A(\chi_U \otimes (x_0x_1\otimes L_0 L_1) - \chi_U \otimes (\chi_{U^{-1}} \cdot x_1)x_0 \otimes (\chi_{U^{-1}}\cdot L_1)L_0) \\
&= tr_A b_\G(\chi_U \otimes (x_0\otimes L_0) d(x_1\otimes L_1)), 
\end{align*}
using the twisted trace property from Lemma \ref{twistedtrace}. Similarly one checks
$$
b_\G tr_A(\chi_U \otimes d(x_1\otimes L_1)) = tr_A b_\G(\chi_U \otimes d(x_1\otimes L_1)). 
$$
It follows that $ tr_A $ is a chain map of paracomplexes. 

We define $ \tau_A = tr_A X_\G(\lambda_A) $ and claim that $ [\tau_A] $ is an inverse for $ [\iota_A] $. From the definitions one immediately computes 
$ [\iota_A] \cdot [\tau_A] = \id $. 
It thus remains to show that $ [\tau_A] \cdot
[\iota_A] = \id $. Consider the $ \G $-equivariant homomorphisms $ i_j: A \otimes_{\ccinfty(\G^{(0)})} \K(E) \rightarrow A \otimes_{\ccinfty(\G^{(0)})} \K(E) \otimes_{\ccinfty(\G^{(0)})} \K(E) $ for $ j = 0,1 $ given by 
\begin{align*}
i_0 = \id \otimes \iota, \quad 
i_1 = (\id \otimes \sigma) i_0. 
\end{align*}
Here we use the canonical identification 
$$ 
A \otimes_{\ccinfty(\G^{(0)})} \K(E) \cong A \otimes_{\ccinfty(\G^{(0)})} \K(E) \otimes_{\ccinfty(\G^{(0)})} \ccinfty(\G^{(0)}) 
$$ 
in the definition of $ i_0 $, and the tensor flip automorphism $ \sigma $ of $ \K(E) \otimes_{\ccinfty(\G^{(0)})} \K(E) $ given by $ \sigma(L_1 \otimes L_2) = L_2 \otimes L_1 $ in the definition of $ i_1 $. 

We calculate $ [i_0] \cdot [\tau_{A\otimes \K(E)}] = \id  $ 
and $ [i_1] \cdot [\tau_{A \otimes \K(E)}] =
[\tau_A] \cdot [\iota_A] $. Let us show that the maps $ i_0 $ and $ i_1 $ are $ \G $-equivariantly homotopic. To this end observe that $ \K(E) \otimes_{\ccinfty(\G^{(0)})} \K(E) \cong \K(E \otimes_{\ccinfty(\G^{(0)})} E)$ as $ \G $-algebras and denote by $ \Sigma $ the flip automorphism of $ E \otimes_{\ccinfty(\G^{(0)})} E $ given by $ \Sigma (e \otimes f) = f \otimes e $. For $ t \in [0,1] $ we then obtain a $ \G $-equivariant linear endomorphism $ \Sigma_t $ of $ E \otimes_{\ccinfty(\G^{(0)})} E $ given by
$$ 
\Sigma_t = \cos(\pi t/2) \id + i \sin(\pi t/2) \Sigma,
$$
and we note that $ \Sigma_t $ is invertible with inverse $ \Sigma_t^{-1} = \cos(\pi t/2) \id - i \sin(\pi t/2) \Sigma $. Conjugation with $ \Sigma_t $ defines $ \G $-equivariant algebra automorphism $ \sigma_t $ of 
$$ 
\K(E \otimes_{\ccinfty(\G^{(0)})} E) = (E \otimes_{\ccinfty(\G^{(0)})} E) \otimes_{\ccinfty(\G^{(0)})} (E \otimes_{\ccinfty(\G^{(0)})} E). 
$$ 
The family $ (\sigma_t)_{t \in [0,1]} $ depends smoothly on $ t $, and by construction we have 
$ \sigma_0 = \id $ and $ \sigma_1 = \sigma $. Now define $ h_t: A \otimes_{\ccinfty(\G^{(0)})} \K(E) \rightarrow 
A \otimes_{\ccinfty(\G^{(0)})} \K(E) \otimes_{\ccinfty(\G^{(0)})} \K(E) $ by $ h_t = (\id \otimes \sigma_t) i_0 $ for $ t \in [0,1] $. Then each $ h_t $ is a $ \G $-equivariant algebra homomorphism, and by construction
$ h_j = i_j $ for $ j = 0,1 $. Since the family $ (h_t)_{t \in [0,1]} $ depends again smoothly on $ t $ we have thus constructed a $ \G $-equivariant smooth homotopy between $ i_0 $ and $ i_1 $. According to Theorem \ref{thm: homotopy invariance} we obtain $ [i_0] = [i_1] $, and hence $ [\tau_A] \cdot [\iota_A] = \id $ as required. 
\end{proof}

In order to discuss the implications of Theorem \ref{StabLemma} for the stability properties of the functor $ HP^\G_* $ we need some preparations. 

\begin{lemma} \label{automaticequivariance}
Let $ E,F $ be $ \G $-modules and assume that $ \phi: E \rightarrow F $ is a $ \ccinfty(\G^{(0)}) $-linear isomorphism. 
Then 
$$ 
\phi^T = T_F^{-1} \circ (\id \otimes \phi) \circ T_E: \ccinfty(\G) \stackrel{r,\id}{\otimes} E \rightarrow \ccinfty(\G) \stackrel{r,\id}{\otimes} F 
$$ 
is an isomorphism of $ \G $-modules. 
\end{lemma}

\begin{proof}
The map $ T_F^{-1} \circ (\id \otimes \phi) \circ T_E $ is $ \G $-linear by construction, and it is clearly bijective. 
\end{proof}

Now assume that $ E, F $ are $ \G $-modules equipped with $ \G $-equivariant pairings $ h_E, h_F $, respectively. We say that a $ \ccinfty(\G^{(0)}) $-linear map $ \phi: E \rightarrow F $ is isometric if 
$$ 
h_F(\phi(e_1), \phi(e_2)) = h_E(e_1, e_2)  
$$
for all $ e_1, e_2 \in E $. 

\begin{lemma} \label{automaticequivariancepairing}
Let $ E, F $ be $ \G $-modules equipped with $ \G $-equivariant pairings, and assume that $ \phi: E \rightarrow F $ is a $ \ccinfty(\G^{(0)}) $-linear isometric isomorphism. Then $ \phi^T: \ccinfty(\G) \stackrel{r,\id}{\otimes} E \rightarrow \ccinfty(\G) \stackrel{r,\id}{\otimes} F $ is an isometric isomorphism of $ \G $-modules. 
\end{lemma}

\begin{proof}
The $ \G $-module $ \ccinfty(\G) $ is equipped with the regular pairing $ \lambda $, introduced in the discussion following Definition \ref{defpairing}, so that the tensor product pairing on $ \ccinfty(\G) \stackrel{r,\id}{\otimes} E $ is given by 
$$ 
h_{\ccinfty(\G) \otimes E}(f \otimes e_1, g \otimes e_2) = \lambda(f,g) h_E(e_1, e_2)
$$
for $ f,g \in \ccinfty(\G) $ and $ e_1, e_2 \in E $. If $ U, V \subseteq \G $ are compact open bisections, then using $ \lambda(\chi_U, \chi_V) = \lambda(\chi_{U \cap V} = \chi_{r(U \cap V)} $ we obtain
\begin{align*}
h_{\ccinfty(\G) \otimes F}&(\phi^T(\chi_U \otimes e_1), \phi^T(\chi_V \otimes e_2)) \\
&= h_{\ccinfty(\G) \otimes F}(\chi_U \otimes \chi_U \cdot \phi(\chi_{U^{-1}} \cdot e_1), \chi_V \otimes \chi_V \cdot \phi(\chi_{V^{-1}} \cdot e_2)) \\
&= \lambda(\chi_{U \cap V}) h_F(\chi_U \cdot \phi(\chi_{U^{-1}} \cdot e_1), \chi_V \cdot \phi(\chi_{V^{-1}} \cdot e_2)) \\
&= \lambda(\chi_{U \cap V}) \chi_{U \cap V} \cdot h_F(\phi(\chi_{(U \cap V)^{-1}} \cdot e_1), \phi(\chi_{(U \cap V)^{-1}} \cdot e_2)) \\
&= \lambda(\chi_{U \cap V}) \chi_{U \cap V} \cdot h_E(\chi_{(U \cap V)^{-1}} \cdot e_1, \chi_{(U \cap V)^{-1}} \cdot e_2) \\
&= h_{\ccinfty(\G) \otimes E}(\chi_U \otimes e_1, \chi_V \otimes e_2)
\end{align*} 
as required.  
\end{proof}

Let us say that a $ \G $-equivariant pairing $ h $ on a $ \G $-module $ E $ is $ \ccinfty(\G^{(0)}) $-regular if there exists a direct sum decomposition $ E \cong \ccinfty(\G^{(0)}) \oplus F $ of the underlying $ \ccinfty(\G^{(0)}) $-modules which is orthogonal with respect to $ h $, such that the restriction of $ h $ to $ \ccinfty(\G^{(0)}) \otimes_{\ccinfty(\G^{(0)})} \ccinfty(\G^{(0)}) \subseteq E \otimes_{\ccinfty(\G^{(0)})} E $ is the canonical isomorphism with $ \ccinfty(\G^{(0)}) $. Here orthogonality means that elements from opposite summands pair to zero under $ h $. The regular pairing on $ \D(\G) $ is an example of a $ \ccinfty(\G^{(0)}) $-regular pairing which is typically not $ \G $-admissible. 

\begin{thm}[Stability]
Let $ E $ be a $ \G $-module equipped with a $ \ccinfty(\G^{(0)}) $-regular $ \G $-equivariant pairing. 
Then there exists an invertible element in 
$$ 
HP^\G_0(A, A \otimes_{\ccinfty(\G^{(0)})} \K(E))
$$ 
for any pro-$ \G $-algebra $ A $. It follows that we have natural isomorphisms 
$$
HP^\G_*(A \otimes_{\ccinfty(\G^{(0)})} \K(E), B) \cong HP^\G_*(A, B) \cong HP^\G_*(A, B \otimes_{\ccinfty(\G^{(0)})} \K(E))
$$
for all pro-$ \G $-algebras $ A, B $. 
\end{thm}

\begin{proof}
Let us write $ E^{\oplus \infty} $ for the countable direct sum of copies of $ E $. By assumption, we obtain an isometric isomorphism $ E^{\oplus \infty} \cong \ccinfty(\G^{(0)}) \oplus E^{\oplus \infty} $ of $ \ccinfty(\G^{(0)}) $-modules, where we consider the canonical pairing on the first summand on the right hand side, and the given pairing on each copy of $ E $. If we equip $ \ccinfty(\G^{(0)}) $ with the trivial $ \G $-module structure and $ E^{\oplus \infty} $ with the one induced from $ E $, then the pairing on $ \ccinfty(\G^{(0)}) \oplus E^{\oplus \infty} $ becomes $ \G $-admissible. However, we note that the resulting $ \G $-module $ \ccinfty(\G^{(0)}) \oplus E^{\oplus \infty} $ is typically not isomorphic to $ E^{\oplus \infty} $ as a $ \G $-module. 

Due to Lemma \ref{automaticequivariancepairing} there exists an isometric isomorphism 
\begin{align*}
\DG \otimes_{\ccinfty(\G^{(0)})} E^{\oplus \infty} \cong \DG \otimes_{\ccinfty(\G^{(0)})} (\ccinfty(\G^{(0)}) \oplus E^{\oplus\infty}) 
\end{align*}
of $ \G $-modules with respect to the inner products described above. Using Theorem \ref{StabLemma} we therefore obtain 
\begin{align*}
X_\G(\T(A &\otimes_{\ccinfty(\G^{(0)})} \K_\G)) = X_\G(\T(A \otimes_{\ccinfty(\G^{(0)})} \K(\DG))) \\ 
&\simeq X_\G(\T(A \otimes_{\ccinfty(\G^{(0)})} \K(\DG \otimes_{\ccinfty(\G^{(0)})} (\ccinfty(\G^{(0)}) \oplus E^{\oplus\infty})))) \\
&\cong X_\G(\T(A \otimes_{\ccinfty(\G^{(0)})} \K(\DG \otimes_{\ccinfty(\G^{(0)})} E^{\oplus\infty}))) \\
&\cong X_\G(\T(A \otimes_{\ccinfty(\G^{(0)})} \K(\DG \otimes_{\ccinfty(\G^{(0)})} E) \otimes_{\ccinfty(\G^{(0)})} \K(\ccinfty(\G^{(0)})^{\oplus\infty}))) \\
&\simeq X_\G(\T(A \otimes_{\ccinfty(\G^{(0)})} \K(\DG \otimes_{\ccinfty(\G^{(0)})} E))) \\
&\cong X_\G(\T(A \otimes_{\ccinfty(\G^{(0)})} \KG \otimes_{\ccinfty(\G^{(0)})} \K(E))). 
\end{align*}
This yields the assertion. 
\end{proof}

Let $ \G $ be a proper ample groupoid such that $ \G \backslash \G^{(0)} $ is paracompact. According to Proposition \ref{smoothcutoff} there exists a locally constant cut-off function $ c $ for $ \G $. It follows that for any $ f \in \ccinfty(\G^{(0)}) $ the function $ s^*(c) r^*(f): \G \rightarrow \mathbb{C} $ given by 
$$
s^*(c) r^*(f)(\alpha) = c(s(\alpha)) f(r(\alpha)) 
$$
has compact support and is thus contained in $ \ccinfty(\G) $. 

Now let $ E, F $ be $ \G $-modules and let $ \phi: E \rightarrow F $ be a $ \ccinfty(\G^{(0)}) $-linear map. From Lemma \ref{automaticequivariance} we obtain a $ \G $-equivariant linear map $ \phi^T = T_F^{-1}(\id \otimes \phi) T_E: \ccinfty(\G) \stackrel{r,\id}{\otimes} E \rightarrow  \ccinfty(\G) \stackrel{r,\id}{\otimes} F $. 
Recall moreover from Lemma \ref{integrationmap} that the integration map $ \lambda: \ccinfty(\G) \rightarrow \ccinfty(\G^{(0)}) $ is $ \G $-equivariant with respect to the left multiplication action on $ \ccinfty(\G) = \DG $. 
Hence we obtain a linear map $ \phi^\G: E \rightarrow F $ by defining 
$$ 
\phi^\G(f \otimes e) = (\lambda \otimes \id) \phi^T(s^*(c) r^*(f) \otimes e), 
$$
using the canonical identification $ X \cong \ccinfty(\G^{(0)}) \otimes_{\ccinfty(\G^{(0)})} X $ for $ X = E,F $. 

\begin{lemma} \label{averaging}
Let $ \G $ be a proper ample groupoid with $ \G \backslash \G^{(0)} $ paracompact and let $ E, F $ be $ \G $-modules. If 
$ \phi: E \rightarrow F $ is a $ \ccinfty(\G^{(0)}) $-linear map then $ \phi^\G: E \rightarrow F $ is a $ \G $-equivariant linear map. 
\end{lemma}

\begin{proof}
For a compact open bisection $ U \subseteq \G $ and a compact open set $ V \subseteq \G^{(0)} $ we have $ \chi_U \cdot \chi_V = \chi_{r(U \cap s^{-1}(V))} = \chi_{U \cdot V} $. 
Using this we calculate
\begin{align*}
\lambda(\chi_U * (s^*(c) r^*(\chi_V)))(x) &= \chi_U \cdot \lambda(s^*(c) r^*(\chi_V))(x) \\
&= \sum_{\gamma \in \G^x} \chi_U(\gamma) \lambda(s^*(c) r^*(\chi_V))(\gamma^{-1} \cdot x) \\
&= \sum_{\beta \in \G^{\gamma^{-1} \cdot x}} \sum_{\gamma \in \G^x} \chi_U(\gamma) c(s(\beta)) \chi_V(r(\beta)) \\
&= \sum_{\beta \in \G^x} \sum_{\gamma \in \G^x} \chi_U(\gamma) c(s(\gamma^{-1}\beta)) \chi_V(r(\gamma^{-1}\beta)) \\
&= \sum_{\gamma \in \G^x} \chi_U(\gamma) \chi_V(r(\gamma^{-1})) \allowdisplaybreaks[2] \\
&= \chi_{r(U \cap s^{-1}(V))}(x) \\
&= \sum_{\beta \in \G^x} cs(\beta) \chi_{r(U \cap s^{-1}(V))}(x) \\
&= \lambda(s^*(c) r^*(\chi_{U \cdot V}))(x)
\end{align*}
for all $ x \in \G^{(0)} $. For $ e \in E $ we then compute 
\begin{align*}
\phi^\G(\chi_U \cdot (\chi_V \otimes e)) &= \phi^\G(\chi_{U \cdot V} \otimes \chi_U \cdot e) \\ 
&= (\lambda \otimes \id) \phi^T(s^*(c) r^*(\chi_{U \cdot V}) \otimes \chi_U \cdot e) \\
&= (\lambda \otimes \id) \phi^T(\chi_U * (s^*(c) r^*(\chi_V)) \otimes \chi_U \cdot e) \\
&= \chi_U \cdot (\lambda \otimes \id) \phi^T(s^*(c) r^*(\chi_V) \otimes e) \\
&= \chi_U \cdot \phi^\G(\chi_V \otimes e)
\end{align*}
as required, noting that $ \phi^T $ is $ \ccinfty(\G) $-linear. 
\end{proof}

We remark that if the map $ \phi $ in Lemma \ref{averaging} is already $ \G $-equivariant then $ \phi^\G = \phi $. Indeed, in this case we have $ \phi^\G(f \otimes e) = \lambda(s^*(c)) f \otimes \phi(e) = f \otimes \phi(e) $ for all $ f \in \ccinfty(\G^{(0)}) $ and $ e \in E $. In a similar way we get $ (\phi \psi)^\G = \phi^\G \psi $ and $ (\theta \phi)^\G = \theta \phi^\G $ if $ \psi, \theta $ are $ \G $-equivariant linear maps. 

\begin{prop}
Let $ \G $ be a proper ample groupoid with $ \G \backslash \G^{(0)} $ paracompact. Then we have a natural isomorphism 
$$
HP_*^\G(A,B) \cong H_*(\Hom_{A(\G)}(X_\G(\T A), X_\G(\T B)))
$$
for all $ \G $-algebras $ A, B $. 
\end{prop}

\begin{proof} 
The integration map defines a $ \G $-equivariant surjection $ \lambda: \DG \rightarrow \ccinfty(\G^{(0)}) $, compare Lemma \ref{integrationmap}. 
Moreover, the extension of functions by zero induces a $ \ccinfty(\G^{(0)}) $-linear inclusion map $ \iota: \ccinfty(\G^{(0)}) \rightarrow \DG $. Since 
$$
\lambda \iota(f)(x) = \sum_{\alpha \in \G^x} \iota(f)(\alpha) = f(x) 
$$
we see that $ \ccinfty(\G^{(0)}) $ is a direct summand of the $ \ccinfty(\G^{(0)}) $-module $ \DG $. According to Lemma \ref{averaging} it follows that the 
trivial $ \G $-module $ \ccinfty(\G^{(0)}) $ is a direct summand of $ \DG $ in the category of $ \G $-modules as well. We conclude that the regular pairing on $ \DG $ 
is $ \G $-admissible, so that the claim follows from Theorem \ref{StabLemma}. 
\end{proof}

\subsection{Excision} \label{section:excision}

In the final part of this section we show that equivariant periodic cyclic homology satisfies excision in both variables. The key ideas go back to \cite{MEYER_thesis}, and the argument follows closely the proof in the group equivariant case \cite{VOIGT_equivariantperiodiccyclichomology}. For this reason we will be rather brief and only sketch the main strategy. 

We consider an admissible extension 
\begin{equation*} \label{eqext}
   \xymatrix{
     K\;\; \ar@{>->}[r]^{\iota} & E \ar@{->>}[r]^{\pi} & Q 
     }
\end{equation*}
of pro-$ \G $-algebras, with a fixed $ \G $-equivariant pro-linear splitting $ \sigma: Q \rightarrow E $ for the quotient homomorphism $ \pi: E \rightarrow Q $. 

Let $ X_\G(\T E:\T Q) $ be the kernel of the map $ X_\G(\T\pi): X_\G(\T E) \rightarrow X_\G(\T Q) $ induced by $ \pi $. The splitting $ \sigma $ yields a direct sum decomposition $ X_\G(\T E) = X_\G(\T E:\T Q) \oplus X_\G(\T Q) $ of pro-$ \G $-anti-Yetter-Drinfeld modules. The resulting extension 
$$ 
   \xymatrix{
      X_\G(\T E:\T Q)\;\; \ar@{>->}[r] & X_\G(\T E) \ar@{->>}[r] & X_\G(\T Q) 
     }
$$ 
of paracomplexes induces long exact sequences in homology in both variables. Moreover, there is a canonical map $ \rho: X_\G(\T K) \rightarrow X_\G(\T E:\T Q) $ of paracomplexes of pro-$ \G $-anti-Yetter-Drinfeld modules. The key step in the proof of the excision theorem is the following result. 

\begin{thm} \label{Excision2} 
The map $ \rho: X_\G(\T K) \rightarrow X_\G(\T E:\T Q) $ is a homotopy equivalence.
\end{thm}

As a consequence of Theorem \ref{Excision2} one obtains excision in both variables for $ \G $-equivariant periodic cyclic homology.

\begin{thm}[Excision] \label{Excision} 
Let $ A $ be a pro-$ \G $-algebra and let $ 0 \rightarrow K \rightarrow E \rightarrow Q \rightarrow 0 $ be
an extension of pro-$ \G $-algebras which is admissible as an extension of pro-$ \ccinfty(\G^{(0)}) $-modules. Then there are two natural exact sequences
\begin{equation*}
\xymatrix{
 {HP^\G_0(A,K)\;} \ar@{->}[r] \ar@{<-}[d] &
      HP^\G_0(A,E) \ar@{->}[r] &
        HP^\G_0(A,Q) \ar@{->}[d] \\
   {HP^\G_1(A,Q)\;} \ar@{<-}[r] &
    {HP^\G_1(A,E)}  \ar@{<-}[r] &
     {HP^\G_1(A,K)} \\
}
\end{equation*}
and
\begin{equation*}
\xymatrix{
    {HP^\G_0(Q,A)\;} \ar@{->}[r] \ar@{<-}[d] &
       HP^\G_0(E,A) \ar@{->}[r] &
          HP^\G_0(K,A) \ar@{->}[d] \\
    {HP^\G_1(K,A)\;} \ar@{<-}[r] &
      {HP^\G_1(E,A)}  \ar@{<-}[r] &
        {HP^\G_1(Q,A)} \\
}
\end{equation*}
The horizontal maps in these diagrams are induced by the maps in the extension.
\end{thm} 

We point out that in Theorem \ref{Excision} we only require that the given extension is admissible as an extension of pro-$ \ccinfty(\G^{(0)}) $-modules, or equivalently, that there exists a $ \ccinfty(\G^{(0)}) $-pro-linear splitting for the quotient homomorphism $ \pi: E \rightarrow Q $. 

Let us indicate how Theorem \ref{Excision2} implies Theorem \ref{Excision}. Tensoring the extension given in Theorem \ref{Excision} with $ \KG $ yields an extension 
\begin{equation*}\label{eqext2}
   \xymatrix{
     K \otimes_{\ccinfty(\G^{(0)})} \KG\;\; \ar@{>->}[r] & E \otimes_{\ccinfty(\G^{(0)})} \KG \ar@{->>}[r] & Q \otimes_{\ccinfty(\G^{(0)})} \KG 
     }
\end{equation*}
of pro-$ \G $-algebras which is admissible as an extension of pro-$ \ccinfty(\G^{(0)}) $-modules. This is not quite enough to apply Theorem \ref{Excision2}, however, using the same argument as in the proof of Lemma \ref{automaticequivariance} we obtain a $ \G $-equivariant pro-linear splitting for the quotient map $ E \otimes_{\ccinfty(\G^{(0)})} \KG \rightarrow Q \otimes_{\ccinfty(\G^{(0)})} \KG $. Hence the hypotheses of Theorem \ref{Excision2} are indeed satisfied, and Theorem \ref{Excision} follows by considering long exact sequences in homology.

\section{The Green-Julg Theorem} \label{section:greenjulg}

In this section we establish a homological analogue of the Green-Julg theorem for the equivariant $ K $-theory of proper groupoids due to Tu \cite[Proposition 6.25]{TU_novikovhyperbolique}. More precisely, we shall prove the following result. 

\begin{thm} \label{Green-Julg} 
Let $ \G $ be a proper ample groupoid such that $ \G \backslash\G^{(0)} $ is paracompact. Then there exists a natural isomorphism 
$$
HP_*^\G(\ccinfty(\G^{(0)}),A)\cong HP_*^{\ccinfty(\G\backslash\G^{(0)})}(\ccinfty(\G\backslash \G^{(0)}), A \rtimes \G)
$$
for all $ \G $-algebras $ A $. 
\end{thm}

Here $ HP^{\ccinfty(\G \backslash \G^{(0)})}_* $ is our notation for the equivariant periodic cyclic theory associated to the totally disconnected locally compact 
space $ \G \backslash \G^{(0)} $, viewed as an ample groupoid. The crossed product $ A \rtimes \G $ becomes an essential $ \ccinfty(\G\backslash\G^{(0)}) $-module via 
pointwise multiplication on the copy of $ \DG $. That is, we have 
$$
(h \cdot f)(\alpha) = h(r(\alpha)) f(\alpha) = f(\alpha) h(s(\alpha)) 
$$
for all $ h \in \ccinfty(\G\backslash\G^{(0)}), f \in \DG $. 

For the proof of Theorem \ref{Green-Julg} we need some preparations. The basic strategy is to use localisation at the $ \G $-orbits in $ \G^{(0)} $ on both sides in order to reduce the claim to the Green-Julg theorem for finite groups established in \cite{VOIGT_thesis}, \cite{VOIGT_equivariantperiodiccyclichomology}. 

Let $ x \in \G^{(0)} $. By slight abuse of notation, we shall write $ [x] $ both for the $ \G $-orbit $ \G \cdot x \subseteq \G^{(0)} $ and the class of $ x $ under the quotient 
map $ \G^{(0)} \rightarrow \G \backslash \G^{(0)} $. The subset $ [x] \subseteq \G^{(0)} $ is closed in $ \G^{(0)} $, 
and since $ \G $ is proper it is discrete in the subspace topology. 

For $ [x] = \G \cdot x \in \G \backslash \G^{(0)} $ consider the maximal ideal $ \mathfrak{m}_{[x]} \subseteq \ccinfty(\G \backslash \G^{(0)}) $ of all functions vanishing at $ [x] $. Given an essential $ \ccinfty(\G \backslash \G^{(0)}) $-module $ M $ we write 
$$
M_{[x]} = M/\mathfrak{m}_{[x]} \cdot M 
$$
for the localisation of $ M $ at $ [x] $. Since functions in $ \ccinfty(\G \backslash \G^{(0)}) $ are locally constant this identifies indeed canonically with the usual ring-theoretic localisation at the maximal ideal $ \mathfrak{m}_{[x]} $. 

\begin{lemma} \label{localisationexact}
Let $ 0 \rightarrow K \rightarrow E \rightarrow Q \rightarrow 0 $ be a short exact sequence of essential $ \ccinfty(\G \backslash \G^{(0)}) $-modules. Then the localised sequence $ 0 \rightarrow K_{[x]} \rightarrow E_{[x]} \rightarrow Q_{[x]} \rightarrow 0 $ is exact for all $ [x] \in \G \backslash \G^{(0)}$. 
\end{lemma}

\begin{proof}
This is a well-known property of localisations; we shall give a direct argument in our special situation here for the convenience of the reader. Note that if $ M $ is an essential $ \ccinfty(\G \backslash \G^{(0)}) $-module and $ m = \sum f_i \cdot m_i \in \mathfrak{m}_{[x]} \cdot V $ with $ f_i \in \mathfrak{m}_{[x]}, m_i \in M $ then there exists $ f \in \mathfrak{m}_{[x]} $ such that $ f \cdot m = m $. For this it suffices to take $ f = \chi_U $ where $ U $ is the union of the supports of the $ f_i $.  

Let us write $ \iota: K \rightarrow E $ and $ \pi: E \rightarrow Q $ for the maps in the short exact sequence. It is evident that $ \pi_{[x]} $ is surjective. If $ e \in E $ represents an element of $ \ker(\pi_{[x]}) $, then $ \pi(e) \in \mathfrak{m}_{[x]} \cdot Q $, 
and we can find $ f \in \mathfrak{m}_{[x]} $ such that $ \pi(e - f \cdot e) = 0 $. 
By exactness we thus have $ e = f \cdot e + \iota(k) $ for some $ k \in K $, and hence $ \iota(k) $ and $ e $ represent the same element in $ E_{[x]} $. It follows that $ \ker(\pi_{[x]}) = \im(\iota_{[x]}) $. Finally, assume that $ k \in K $ represents an element of  $ \ker(\iota_{[x]}) $. Then we have $ \iota(k) \in \mathfrak{m}_{[x]} \cdot E $, and we find $ f \in \mathfrak{m}_{[x]} $ such that $ \iota(k) = f \cdot \iota(k) = \iota(f \cdot k) $. Since $ \iota $ is injective this means that $ k = f \cdot k $ is contained in $ \mathfrak{m}_{[x]} \cdot K $, so that $ k $ represents $ 0 \in K_{[x]} $. 
\end{proof}

Let us also state the following version of the well-known local-to-global principle from commutative algebra, compare \cite[Chapter 3]{ATIYAH_MACDONALD_commutativealgebra}. 

\begin{lemma} \label{localtoglobal}
Let $ \phi: M \rightarrow N $ be a homomorphism of essential $ \ccinfty(\G \backslash \G^{(0)}) $-modules. The $ \phi $ is an isomorphism iff $ \phi_{[x]}: M_{[x]} \rightarrow N_{[x]} $ is an isomorphism for all $ [x] \in \G \backslash \G^{(0)} $. 
\end{lemma} 

\begin{proof} 
Again we briefly sketch the proof for the convenience of the reader. 
If $ \phi $ is an isomorphism then clearly all localised maps $ \phi_{[x]} $ are isomorphisms. 

For the converse, let us first show that if an essential $ \ccinfty(\G \backslash \G^{(0)}) $-module $ V $ satisfies $ V_{[x]} = 0 $ for all $ [x] \in \G \backslash \G^{(0)} $ then $ V = 0 $. 
For this let $ v \in V $ and choose a compact open set $ K \subseteq \G \backslash \G^{(0)} $ such that $ \chi_K \cdot v = v $. Let $ [x] \in K $ and note that $ v = 0 $ in $ V_{[x]} $ means that $ v = f \cdot v $ for some $ f \in \mathfrak{m}_x $. Since the function $ f $ is locally constant we can thus find an open neighborhood $ U $ of $ [x] $ such that $ \chi_U \cdot v = 0 $. Since this works for all $ [x] \in K $ we get $ v = \chi_K \cdot v = 0 $. 

With this in place, we can apply Lemma \ref{localisationexact} to the short exact sequences $ 0 \rightarrow \ker(\phi) \rightarrow M \rightarrow \im(\phi) \rightarrow 0 $ and $ 0 \rightarrow \im(\phi) \rightarrow N \rightarrow N/\im(\phi) \rightarrow 0 $ to obtain canonical isomorphisms $ \ker(\phi_{[x]}) \cong \ker(\phi)_{[x]} $ and $ \im(\phi_{[x]}) \cong \im(\phi)_{[x]} $, as well as  $ N_{[x]}/\im(\phi_{[x]}) \cong (N/\im(\phi))_{[x]} $. This finishes the proof. 
\end{proof}

In our construction below we need an averaging procedure in the setting of $ \G $-anti-Yetter-Drinfeld modules. More precisely, let $ F \in A(\G) $ and consider the linear 
map $ \kappa(F): \OG \rightarrow A(\G) = \OG \rtimes \G $ given by $ \kappa(F)(f)(\alpha, \beta) = f(\alpha) F(\alpha, \beta) $. Then $ \kappa(F) $ is $ \OG $-linear, but in 
general not $ \G $-equivariant. In order to remedy this consider 
$$
\kappa^\G(F)(f)(\alpha, \beta) = \sum_{\gamma \in \G^{r(\alpha)}} f(\alpha) F(\gamma^{-1} \alpha \gamma, \gamma^{-1} \beta).  
$$
Since $ F $ has compact support this is a well-defined function on $ \G_{ad} \times_{\pi, r} \G $, and using that $ \G $ is proper we see that $ \kappa^\G(F)(f) $ is in fact 
contained in $ A(\G) $, so that we obtain again a linear map $ \kappa^\G(F): \OG \rightarrow A(\G) $. 
We calculate 
\begin{align*}
(\chi_U \cdot \kappa^\G(F)(f))(\alpha, \beta) &= \sum_{\delta \in \G^{r(\alpha)}}\sum_{\gamma \in \G^{s(\delta)}} \chi_U(\delta) f(\delta^{-1} \alpha \delta) F(\gamma^{-1} \delta^{-1} \alpha \delta \gamma, \gamma^{-1} \delta^{-1} \beta) \\ 
&= \sum_{\gamma \in \G^{r(\alpha)}} \sum_{\delta \in \G^{r(\alpha)}} \chi_U(\delta) f(\delta^{-1}\alpha \delta) F(\gamma^{-1} \alpha \gamma, \gamma^{-1} \beta) \\
&= \kappa^\G(F)(\chi_U \cdot f)(\alpha, \beta) 
\end{align*}
for a compact open bisection $ U \subseteq \G $, and it follows that $ \kappa^\G(F) $ is $ \G $-equivariant. Since it is also $ \OG $-linear this means that $ \kappa^\G(F) $ is a 
map of $ \G $-anti-Yetter-Drinfeld modules. It is straightforward to check that this construction is compatible with the right $ \DG $-action in the sense that 
$$ 
\kappa^\G(F \cdot \chi_U)(f) = (\kappa^\G(F)(f)) \cdot \chi_U 
$$ 
for all $ f \in \OG, F \in A(\G) $. Hence we obtain a right $ \DG $-linear map $ \kappa^\G $ from $ A(\G) $ 
into $ \Hom_{A(\G)}(\OG, A(\G)) $, which is also $ \ccinfty(\G \backslash \G^{(0)}) $-linear with respect to the obvious actions on the source and target. 
If $ M $ is a $ \G $-module this yields an induced 
map $ \kappa^\G \otimes \id: \OG \otimes_{\ccinfty(\G^{(0)})} M \cong A(\G) \otimes_{\DG} M \rightarrow 
\Hom_{A(\G)}(\OG, A(\G) \otimes_{\DG} M) \cong \Hom_{A(\G)}(\OG, \OG \otimes_{\ccinfty(\G^{(0)})} M)$, which we will for simplicity again denote by $ \kappa^\G $. 
We can summarise our discussion as follows. 

\begin{lemma} \label{AGaveraging}
Let $ M $ be a $ \DG $-module. Then the above construction yields a $ \ccinfty(\G \backslash \G^{(0)}) $-linear
map $ \kappa^\G: \OG \otimes_{\ccinfty(\G^{(0)})} M \rightarrow \Hom_{A(\G)}(\OG, \OG \otimes_{\ccinfty(\G^{(0)})} M) $. 
\end{lemma} 

Let $ U \subseteq \G^{(0)} $ be a $ \G $-compact open and closed subset, that is, $ U $ is an open and closed set with $ \G \cdot U \subseteq U $ such that $ \G \backslash U $ is compact. We write $ \G^U_U \subseteq \G $ for the subgroupoid consisting of all arrows with source and target in $ U $. 
Multiplication by the characteristic function $ \chi_{\G \backslash U} \in \cinfty(\G\backslash \G^{(0)}) $ yields a chain map 
$$
\Hom_{A(\G)}(\OG[0], \theta \Omega_\G(A)) \rightarrow \Hom_{A(\G^U_U)}(\O_{\G^U_U}[0], \theta \Omega_{\G^U_U}(A_U)), 
$$
where $ A_U = \chi_U \cdot A $ is the $ \G^U_U $-algebra obtained from $ A $ by restriction from $ \DG $ to $ \chi_U \DG \chi_U = \D(\G^U_U) $. 
Since $ \G \backslash \G^{(0)} $ is assumed to be paracompact it can be written as a disjoint union of compact open subsets. It follows that we obtain an isomorphism 
\begin{align*}
\Hom_{A(\G)}(\OG[0], \theta \Omega_\G(A)) &\cong \varprojlim_U \Hom_{A(\G^U_U)}(\O_{\G^U_U}[0], \theta \Omega_{\G^U_U}(A_U)) \\
&\cong \varprojlim_U \varprojlim_m \Hom_{A(\G^U_U)}(\O_{\G^U_U}[0], \theta^m \Omega_{\G^U_U}(A_U)) 
\end{align*}
by taking the inverse limit of the above chain maps over all $ \G $-compact open 
sets $ U \subseteq \G^{(0)} $. In a similar way we obtain an isomorphism of chain complexes
\begin{align*}
\Hom_{\ccinfty(\G \backslash \G^{(0)})}(&\ccinfty(\G \backslash \G^{(0)})[0], \theta \Omega_{\ccinfty(\G \backslash \G^{(0)})}(A \rtimes \G)) \\
&\cong \varprojlim_U \Hom_{\cinfty(\G \backslash U)}(\cinfty(\G \backslash U)[0], \theta \Omega_{\cinfty(\G \backslash U)}(A_U \rtimes \G_U^U)) \\
&\cong \varprojlim_U \varprojlim_m \theta^m \Omega_{\cinfty(\G \backslash U)}(A_U \rtimes \G_U^U),
\end{align*}
again taken over all $ \G $-compact open subsets $ U \subseteq \G^{(0)} $. 

We shall construct a compatible family of $ \ccinfty(\G \backslash \G^{(0)}) $-linear maps 
\begin{align*}
\gamma_\G^m: \Hom_{\ccinfty(\G \backslash \G^{(0)})}(\ccinfty(\G \backslash \G^{(0)})[0], &\theta^m \Omega_{\ccinfty(\G \backslash \G^{(0)})}(A \rtimes \G)) \\
&\rightarrow \Hom_{A(\G)}(\OG[0], \theta^m \Omega_\G(A))
\end{align*}
for $ m \in \mathbb{N} $
by taking $ \gamma_\G^m = \varprojlim_U (\gamma_\G^m)_U $, where $ (\gamma_\G^m)_U(\omega) = \theta^m \kappa^{\G_U^U}(\gamma_U(\omega)) $ and $ (\gamma_\G)_U(\omega) = \kappa^{\G_U^U}(\gamma_U(\omega)) $ are defined by 
considering $ \gamma_U: \Omega_{\cinfty(\G \backslash U)}(A_U \rtimes \G_U^U) \rightarrow \Omega_{\G^U_U}(A_U) $, given by 
\begin{align*}
\gamma_U(&\langle a_0 \rtimes \chi_{U_0} \rangle d(a_1 \rtimes \chi_{U_1}) \cdots d(a_n \rtimes \chi_{U_n}))\\
&= \chi_{U_0 \cdots U_n} \otimes \langle a_0 \rangle d(\chi_{U_0} \cdot a_1) d(\chi_{U_0} \chi_{U_1} \cdot a_2) \cdots d(\chi_{U_0} \cdots \chi_{U_{n - 1}} \cdot a_n)  
\end{align*}
for $ \omega = \langle a_0 \rtimes \chi_{U_0} \rangle d(a_1 \rtimes \chi_{U_1}) \cdots d(a_n \rtimes \chi_{U_n}) \in \Omega^n_{\cinfty(\G \backslash U)}(A_U \rtimes \G_U^U) $. 
Moreover, $ \kappa^{\G_U^U} $ is the averaging map constructed above, applied to $ \G_U^U $ and $ M = \Omega_{(\G^U_U)^{(0)}}(A_U) $.

\begin{lemma} \label{gammaboundary}
For every $ \G $-compact open set $ U \subseteq \G^{(0)} $ the map $ (\gamma_\G)_U $ is well-defined and $ \ccinfty(\G \backslash \G^{(0)}) $-linear. Moreover it 
commutes with the Hochschild and Connes boundary operators, and hence induces chain maps $ (\gamma^m_\G)_U $ for all $ m \in \mathbb{N} $ as stated. If $ V \subseteq \G^{(0)} $ is another $ \G $-compact open subset with $ V \subseteq U $ then the diagram 
\begin{center}
\begin{tikzcd}
\theta^m \Omega_{\cinfty(\G \backslash U)}(A_U \rtimes \G_U^U)
\arrow[r,"(\gamma^m_\G)_U"]
\arrow[d,""]& \Hom_{A(\G^U_U)}(\O_{\G^U_U}[0], \theta^m \Omega_{\G^U_U}(A_U)) \arrow[d,""] \\
\theta^m \Omega_{\cinfty(\G \backslash V)}(A_V \rtimes \G_V^V)
\arrow[r,"(\gamma^m_\G)_V"] & \Hom_{A(\G^V_V)}(\O_{\G^V_V}[0], \theta^m \Omega_{\G^V_V}(A_V))   
\end{tikzcd}
\end{center}
induced by restriction from $ U $ to $ V $ is commutative for all $ m \in \mathbb{N} $. 
As a consequence, $ \gamma^m_\G $ is a well-defined $ \ccinfty(\G \backslash \G^{(0)}) $-linear chain map. 
\end{lemma}

\begin{proof}
It follows easily from Lemma \ref{AGaveraging} that the assignment which sends $ \omega \in \Omega_{\cinfty(\G \backslash U)}(A_U \rtimes \G_U^U) $ 
to $ \kappa^{\G_U^U}(\gamma_U(\omega)) $ yields a well-defined $ \ccinfty(\G \backslash \G^{(0)}) $-linear 
map from $ \Omega_{\cinfty(\G \backslash U)}(A_U \rtimes \G_U^U) $ to $ \Hom_{A(\G^U_U)}(\O_{\G^U_U}[0], \Omega_{\G^U_U}(A_U)) $. 

We compute 
\begin{align*}
&b_{\G_U^U} \gamma_U(\langle a_0 \rtimes \chi_{U_0} \rangle d(a_1 \rtimes \chi_{U_1}) \cdots d(a_n \rtimes \chi_{U_n})) \\
&= (-1)^{n - 1} \chi_{U_0 \cdots U_n} \otimes \langle a_0 \rangle d(\chi_{U_0} \cdot a_1) \cdots d(\chi_{U_0 \cdots U_{n - 2}} \cdot a_{n - 1}) \cdot (\chi_{U_0 \cdots U_{n - 1}} \cdot a_n) \\
&\quad+(-1)^n \chi_{U_0 \cdots U_n} \otimes (\chi_{(U_0 \cdots U_n)^{-1}} \cdot \chi_{U_0 \cdots U_{n - 1}} \cdot a_n) \langle a_0 \rangle d(\chi_{U_0} \cdot a_1) \cdots d(\chi_{U_0 \cdots U_{n - 2}} \cdot a_{n - 1}) \\
&= (-1)^{n - 1} \chi_{U_0 \cdots U_n} \otimes \langle a_0 \rangle d(\chi_{U_0} \cdot a_1) \cdots d(\chi_{U_0 \cdots U_{n - 2}} \cdot a_{n - 1}) \cdot (\chi_{U_0 \cdots U_{n - 1}} \cdot a_n) \\
&\quad + (-1)^n \chi_{U_0 \cdots U_n} \otimes (\chi_{U_n^{-1}} \cdot a_n) \langle a_0 \rangle d(\chi_{U_0} \cdot a_1) \cdots d(\chi_{U_0 \cdots U_{n - 2}} \cdot a_{n - 1})
\end{align*} 
and 
\begin{align*} 
&\gamma_U b(\langle a_0 \rtimes \chi_{U_0} \rangle d(a_1 \rtimes \chi_{U_1}) \cdots d(a_n \rtimes \chi_{U_n})) \\
&= (-1)^{n - 1} \gamma_U(\langle a_0 \rtimes \chi_{U_0} \rangle d(a_1 \rtimes \chi_{U_1}) \cdots d(a_{n - 1} \rtimes \chi_{U_{n - 1}}) \cdot (a_n \rtimes \chi_{U_n})) \\
&\quad +(-1)^n \gamma_U((a_n \rtimes \chi_{U_n}) \langle a_0 \rtimes \chi_{U_0} \rangle d(a_1 \rtimes \chi_{U_1}) \cdots d(a_{n - 1} \rtimes \chi_{U_{n - 1}})) \\
&= (-1)^{n - 1} \chi_{U_0 \cdots U_n} \otimes \langle a_0 \rangle d(\chi_{U_0} \cdot a_1) \cdots d(\chi_{U_0 \cdots U_{n - 2}} \cdot a_{n - 1}) \cdot (\chi_{U_0 \cdots U_{n - 1}} \cdot a_n) \\
&\quad+(-1)^n \chi_{U_n U_0 \cdots U_{n - 1}} \otimes a_n \langle \chi_{U_n} 
\cdot a_0 \rangle d(\chi_{U_n U_0} \cdot a_1) \cdots d(\chi_{U_n U_0 \cdots U_{n - 2}} \cdot a_{n - 1}). 
\end{align*}
These expressions do not agree in general, but it follows from Lemma \ref{AGaveraging} that 
$$ 
b_{\G_U^U} \kappa^{\G_U^U}(\gamma_U(\omega)) = \kappa^{\G_U^U}(b_{\G_U^U} \gamma_U(\omega)) = \kappa^{\G_U^U}(\gamma_U b(\omega))  = \kappa^{\G_U^U}(\gamma_U) b(\omega)
$$  
for $ \omega = \langle a_0 \rtimes \chi_{U_0} \rangle d(a_1 \rtimes \chi_{U_1}) \cdots d(a_n \rtimes \chi_{U_n}) $ as desired. 
As a consequence, we see that that $ (\gamma_\G)_U $ induces a map $ (\gamma_\G^m)_U $ between the levels of the Hodge towers on both sides as claimed. 

Since $ \gamma_U $ is obviously compatible with the exterior differential $ d $ the same is true for $ (\gamma_\G)_U $, 
and we conclude that $ (\gamma_\G)_U $ commutes with the Connes boundary operators as well. The compatibility with the canonical restriction map from $ U $ to $ V $  
is straightforward to check, and it follows that $ (\gamma^m_\G) $ is a well-defined chain map. Linearity over $ \ccinfty(\G \backslash \G^{(0)}) $ is clear from the construction. 
\end{proof}

Let $ m \in \mathbb{N} $ and consider the supercomplexes    
\begin{align*}
C^m &= \Hom_{\ccinfty(\G\backslash\G^{(0)})}(\ccinfty(\G\backslash\G^{(0)})[0], 
\theta^m \Omega_{\ccinfty(\G\backslash\G^{(0)})}(A \rtimes \G)), \\
D^m &= \Hom_{A(\G)}(\O_\G[0], \theta^m \Omega_{\G}(A)). 
\end{align*} 
From the proof of Lemma \ref{gammaboundary} we deduce that taking inverse limits of the maps $ \gamma_\G^m $ yields a chain map $ \gamma_\G: \varprojlim C^n \rightarrow \varprojlim D^n $. 
By Lemma \ref{XA} and Lemma \ref{trivialquasifree}, we have
\begin{align*}
H_*(\varprojlim C^m) &\cong HP_*^{\ccinfty(\G\backslash\G^{(0)})}(\ccinfty(\G\backslash \G^{(0)}), A \rtimes \G), \\ 
H_*(\varprojlim D^m) &\cong HP_*^\G(\ccinfty(\G^{(0)}),A), 
\end{align*}  
since $ \ccinfty(\G \backslash \G^{(0)}) $ is quasifree as a $ \ccinfty(\G \backslash \G^{(0)}) $-algebra and $ \ccinfty(\G^{(0)}) $ is quasifree as a $ \G $-algebra. 
Moreover, the maps in homology induced by the chain maps $ \gamma^m_\G $ and $ \gamma_\G $ fit into a commutative diagram  
\begin{center}
\begin{tikzcd}
0 \arrow[r,""] &
\varprojlim^1 H_{* + 1}(C^m) 
\arrow[r,""] 
\arrow[d,"\varprojlim^1(\gamma_\G^m)_*"]& H_*(\varprojlim C^m)  
\arrow[r,""] 
\arrow[d,"(\gamma_\G)_*"]& \varprojlim H_*(C^m) 
\arrow[r,""] 
\arrow[d,"\varprojlim(\gamma_\G^m)_*"] & 0 
\\
0 \arrow[r,""] &
\varprojlim^1 H_{* + 1}(D^m) 
\arrow[r,""] & H_*(\varprojlim D^m) 
\arrow[r,""] & \varprojlim H_*(D^m) 
\arrow[r,""] & 0 
\end{tikzcd}
\end{center}
with exact rows, compare \cite[Theorem 3.5.8]{WEIBEL_homologicalalgebra}. 
Using Milnor's description of $ \varprojlim^1 $ and the $ 5 $-lemma, we thus conclude that Theorem \ref{Green-Julg} is a consequence of the following assertion. 

\begin{thm} \label{Green-Julg-local}
The chain maps $ \gamma^m_\G $ induce isomorphisms 
$$
H_*(C^m) \cong H_*(D^m)
$$
for all $ m \in \mathbb{N} $. 
\end{thm}

\begin{proof} 
Due to Lemma \ref{localisationexact} and Lemma \ref{localtoglobal} it suffices to show that the localised chain maps $ (\gamma_\G^m)_{[x]}: C^m_{[x]} \rightarrow D^m_{[x]} $ induce isomorphisms in homology for all $ [x] \in \G\backslash\G^{(0)} $. Note here that $ C^m $ is a complex of essential $ \ccinfty(\G \backslash \G^{(0)}) $-modules by construction, whereas $ D^m $ is naturally a complex of essential $ \ccinfty(\G\backslash\G_{ad}) $-modules. Using that the projection map $ r = s: \G_{ad} \rightarrow \G^{(0)} $ is $ \G $-equivariant, one obtains an injective essential algebra homomorphism $ \ccinfty(\G\backslash\G^{(0)}) \rightarrow M(\ccinfty(\G\backslash\G_{ad})) $, which allows one to view $ D^m $ as a complex of essential  $ \ccinfty(\G \backslash \G^{(0)}) $-modules as well. 

Let us fix $ x \in \G^{(0)} $. 
The restricted groupoid $ \G^{[x]}_{[x]} $ is discrete, and the localisation $ A_{[x]} $ is naturally a $ \G^{[x]}_{[x]} $-algebra. Similarly, the localisation $ A_x $ 
of $ A $ at the maximal ideal of $ \ccinfty(\G^{(0)}) $ consisting of all functions vanishing at $ x $ is naturally a $ \G^x_x $-algebra, and we note that
the isotropy group $ \G^x_x $ is finite. 

Consider first the homology of $ C^m_{[x]} $. A direct inspection shows that we have a canonical isomorphism 
$$ 
\Omega_{\ccinfty(\G\backslash\G^{(0)})}(A \rtimes \G)_{[x]} \cong \Omega(A_{[x]} \rtimes \G_{[x]}^{[x]})
$$ 
of mixed complexes. 
Using exactness of localisation we get 
\begin{align*}
H_*(C^m_{[x]}) &\cong H_*(\Hom(\mathbb{C}[0], (\theta^m \Omega_{\ccinfty(\G\backslash\G^{(0)})}(A \rtimes \G))_{[x]})) \\
&= H_*((\theta^m \Omega_{\cinfty(X)}(A \rtimes \G))_{[x]}) \\
&\cong H_*(\theta^m(\Omega_{\cinfty(X)}(A \rtimes \G)_{[x]})) \\
&\cong H_*(\theta^m(\Omega_{\ccinfty(\G \backslash \G^{(0)})}(A \rtimes \G)_{[x]})) \\
&= H_*(\theta^m \Omega(A_{[x]} \rtimes \G^{[x]}_{[x]})). 
\end{align*}
In a similar way one can describe the homology of $ D^m_{[x]} $. More precisely, using properness of $ \G $ we obtain
\begin{align*}
H_*(D^m_{[x]}) &\cong H_*(\Hom_{A(\G_{[x]}^{[x]})}(\O_{\G^{[x]}_{[x]}}[0], \theta^m \Omega_{\G^{[x]}_{[x]}}(A_{[x]}))) \\
&\cong H_*(\Hom_{A(\G_x^x)}(\O_{\G_x^x}[0], \theta^m \Omega_{\G^x_x}(A_x))) \\ 
&= H_*(\Hom_{\G_x^x}(\mathbb{C}[0], \theta^m \Omega_{\G^x_x}(A_x))) \\ 
&= H_*(\theta^m \Omega_{\G^x_x}(A_x)^{\G_x^x}) 
\end{align*}
compare also the discussion at the end of section \ref{section:epch}. 

It is a general fact about Hodge towers that a morphism $ M \rightarrow N $ of mixed complexes which induces an isomorphism in Hochschild homology also induces isomorphisms $ H_*(\theta^m M) \rightarrow H_*(\theta^m N) $ for all $ m \in \mathbb{N} $, see \cite{CUNTZ_QUILLEN_nonsingularity}. Now we have 
$$
A_{[x]} \rtimes \G_{[x]}^{[x]} \cong (A_x \rtimes \G^x_x) \otimes M_{[x]}(\mathbb{C}) , 
$$
where $ M_{[x]}(\mathbb{C}) $ denotes the algebra of matrices indexed by $ [x] $. It follows that the inclusion map $ A_x \rtimes \G^x_x \rightarrow A_{[x]} \rtimes \G_{[x]}^{[x]} $ 
induces an isomorphism
$$ 
HH_*(A_{[x]} \rtimes \G^{[x]}_{[x]}) = 
HH_*(\Omega(A_{[x]} \rtimes \G^{[x]}_{[x]})) \cong 
HH_*(\Omega(A_x \rtimes \G^x_x)) = HH_*(A_x \rtimes \G^x_x). 
$$
This reduces the claim to the Green-Julg theorem in group equivariant Hochschild homology for finite groups, see \cite{BRYLINSKI_brown}, \cite{BUES_thesis}. In fact, following through the above identifications one checks that the map 
$$
HH_*(\gamma_\G): HH_*(A_x \rtimes \G^x_x) = HH_*(\Omega(A_x \rtimes \G^x_x)) \rightarrow HH_*(\Omega_{\G^x_x}(A_x)^{\G_x^x}) = HH_*^{\G_x^x}(A_x)
$$
induced by $ \gamma_\G $ agrees up to a nonzero scalar with the isomorphism obtained from these sources, compare the proof of \cite[Theorem 4.3]{VOIGT_thesis}. This implies that $ HH_*(\gamma_\G) $ is indeed an isomorphism, thus completing the argument. 
\end{proof}

\bibliographystyle{plain}

\bibliography{cvoigt}

\def\cprime{$'$} \def\cprime{$'$} \def\cprime{$'$} \def\cprime{$'$}
  \def\cprime{$'$} \def\cprime{$'$}
  \def\polhk#1{\setbox0=\hbox{#1}{\ooalign{\hidewidth
  \lower1.5ex\hbox{`}\hidewidth\crcr\unhbox0}}} \def\cprime{$'$}
  \def\cprime{$'$} \def\cprime{$'$} \def\Dbar{\leavevmode\lower.6ex\hbox to
  0pt{\hskip-.23ex \accent"16\hss}D}
  \def\cftil#1{\ifmmode\setbox7\hbox{$\accent"5E#1$}\else
  \setbox7\hbox{\accent"5E#1}\penalty 10000\relax\fi\raise 1\ht7
  \hbox{\lower1.15ex\hbox to 1\wd7{\hss\accent"7E\hss}}\penalty 10000
  \hskip-1\wd7\penalty 10000\box7}
  \def\cfudot#1{\ifmmode\setbox7\hbox{$\accent"5E#1$}\else
  \setbox7\hbox{\accent"5E#1}\penalty 10000\relax\fi\raise 1\ht7
  \hbox{\raise.1ex\hbox to 1\wd7{\hss.\hss}}\penalty 10000 \hskip-1\wd7\penalty
  10000\box7}
\begin{thebibliography}{10}

\bibitem{ARNONE_CORTINAS_MUKHERJEE_homologysteinberg}
Guido Arnone, Guillermo Cortinas, and Devarshi Mukherjee.
\newblock Homology of {S}teinberg algebras.
\newblock {\em arXiv:2412.15112}, 2024.

\bibitem{ATIYAH_MACDONALD_commutativealgebra}
M.~F. Atiyah and I.~G. Macdonald.
\newblock {\em Introduction to commutative algebra}.
\newblock Addison-Wesley Publishing Co., Reading, Mass.-London-Don Mills, Ont.,
  1969.

\bibitem{BOENICKE_PROIETTI_categoricalbc}
Christian B\"{o}nicke and Valerio Proietti.
\newblock A categorical approach to the {B}aum-{C}onnes conjecture for
  \'{e}tale groupoids.
\newblock {\em J. Inst. Math. Jussieu}, 23(5):2319--2364, 2024.

\bibitem{BOURBAKI_generaltopology1}
Nicolas Bourbaki.
\newblock {\em Elements of mathematics. {G}eneral topology. {P}art 1}.
\newblock Hermann, Paris; Addison-Wesley Publishing Co., Reading,
  Mass.-London-Don Mills, Ont., 1966.

\bibitem{BRIX_GONZALES_HUME_LI_hausdorffcovers}
Kevin~Aguyar Brix, Julian Gonzales, Jeremy Hume, and Xin Li.
\newblock On {H}ausdorff covers for non-{H}ausdorff groupoids.
\newblock {\em arXiv:2503.23203v1}, 2025.

\bibitem{BRYLINSKI_brown}
Jean-Luc Brylinksi.
\newblock Algebras associated with group actions and their homology.
\newblock {\em Brown University preprint}, 1986.

\bibitem{BUES_thesis}
M.~Bues.
\newblock Equivariant differential forms and crossed products.
\newblock {\em PhD thesis, Harvard University}, 1996.

\bibitem{BUSS_HOLKAR_MEYER_universalproperty}
Alcides Buss, Rohit~D. Holkar, and Ralf Meyer.
\newblock A universal property for groupoid {$\rm C^*$}-algebras. {I}.
\newblock {\em Proc. Lond. Math. Soc. (3)}, 117(2):345--375, 2018.

\bibitem{BOENICKE_DELLAIERA_GABE_WILLETT_dynamicasymptoticdimension}
Christian Bönicke, Cl\'{e}ment Dell'Aiera, James Gabe, and Rufus Willett.
\newblock Dynamic asymptotic dimension and {M}atui's {HK} conjecture.
\newblock {\em Proc. Lond. Math. Soc. (3)}, 126(4):1182--1253, 2023.

\bibitem{CONNES_noncommutativedifferentialgeometry}
Alain Connes.
\newblock Noncommutative differential geometry.
\newblock {\em Inst. Hautes \'Etudes Sci. Publ. Math.}, (62):257--360, 1985.

\bibitem{CONNES_noncommutativegeometry}
Alain Connes.
\newblock {\em Noncommutative geometry}.
\newblock Academic Press Inc., San Diego, CA, 1994.

\bibitem{CRAINIC_MOERDIJK_homologyetalegroupoids}
Marius Crainic and Ieke Moerdijk.
\newblock A homology theory for \'{e}tale groupoids.
\newblock {\em J. Reine Angew. Math.}, 521:25--46, 2000.

\bibitem{CUNTZ_QUILLEN_algebraextensions}
Joachim Cuntz and Daniel Quillen.
\newblock Algebra extensions and nonsingularity.
\newblock {\em J. Amer. Math. Soc.}, 8(2):251--289, 1995.

\bibitem{CUNTZ_QUILLEN_nonsingularity}
Joachim Cuntz and Daniel Quillen.
\newblock Cyclic homology and nonsingularity.
\newblock {\em J. Amer. Math. Soc.}, 8(2):373--442, 1995.

\bibitem{CUNTZ_QUILLEN_excision}
Joachim Cuntz and Daniel Quillen.
\newblock Excision in bivariant periodic cyclic cohomology.
\newblock {\em Invent. Math.}, 127(1):67--98, 1997.

\bibitem{FARSI_KUMJIAN_PASK_SIMS_amplegroupoidshomology}
Carla Farsi, Alex Kumjian, David Pask, and Aidan Sims.
\newblock Ample groupoids: equivalence, homology, and {M}atui's {HK}
  conjecture.
\newblock {\em M\"{u}nster J. Math.}, 12(2):411--451, 2019.

\bibitem{LI_classifiablecartan}
Xin Li.
\newblock Every classifiable simple {$\rm C^*$}-algebra has a {C}artan
  subalgebra.
\newblock {\em Invent. Math.}, 219(2):653--699, 2020.

\bibitem{MATUI_homologyandtoplogicalfullgroups}
Hiroki Matui.
\newblock Homology and topological full groups of \'etale groupoids on totally
  disconnected spaces.
\newblock {\em Proc. Lond. Math. Soc. (3)}, 104(1):27--56, 2012.

\bibitem{MATUI_etalegroupoidshifts}
Hiroki Matui.
\newblock \'{E}tale groupoids arising from products of shifts of finite type.
\newblock {\em Adv. Math.}, 303:502--548, 2016.

\bibitem{MEYER_thesis}
Ralf Meyer.
\newblock Analytic cyclic homology.
\newblock {\em PhD Thesis, University of M\"unster}, 1999.

\bibitem{PATERSON_groupoidsinversesemigroups}
Alan L.~T. Paterson.
\newblock {\em Groupoids, inverse semigroups, and their operator algebras},
  volume 170 of {\em Progress in Mathematics}.
\newblock Birkh\"{a}user Boston, Inc., Boston, MA, 1999.

\bibitem{PROIETTI_YAMASHITA_homologyktheory1}
Valerio Proietti and Makoto Yamashita.
\newblock Homology and {$K$}-theory of dynamical systems {I}. {T}orsion-free
  ample groupoids.
\newblock {\em Ergodic Theory Dynam. Systems}, 42(8):2630--2660, 2022.

\bibitem{RENAULT_groupoidapproach}
Jean Renault.
\newblock {\em A groupoid approach to {$C^{\ast} $}-algebras}, volume 793 of
  {\em Lecture Notes in Mathematics}.
\newblock Springer, Berlin, 1980.

\bibitem{STEINBERG_discreteinversesemigroup}
Benjamin Steinberg.
\newblock A groupoid approach to discrete inverse semigroup algebras.
\newblock {\em Adv. Math.}, 223(2):689--727, 2010.

\bibitem{TU_feuilletagesmoyennables}
Jean-Louis Tu.
\newblock La conjecture de {B}aum-{C}onnes pour les feuilletages moyennables.
\newblock {\em $K$-Theory}, 17(3):215--264, 1999.

\bibitem{TU_novikovhyperbolique}
Jean~Louis Tu.
\newblock La conjecture de {N}ovikov pour les feuilletages hyperboliques.
\newblock {\em $K$-Theory}, 16(2):129--184, 1999.

\bibitem{TU_nonhausdorffgroupoids}
Jean-Louis Tu.
\newblock Non-{H}ausdorff groupoids, proper actions and {$K$}-theory.
\newblock {\em Doc. Math.}, 9:565--597, 2004.

\bibitem{VANDAELE_multiplierhopfalgebras}
A.~Van~Daele.
\newblock Multiplier {H}opf algebras.
\newblock {\em Trans. Amer. Math. Soc.}, 342(2):917--932, 1994.

\bibitem{VOIGT_thesis}
Christian Voigt.
\newblock Equivariant cyclic homology.
\newblock {\em PhD thesis, University of M\"unster}, 2003.

\bibitem{VOIGT_equivariantperiodiccyclichomology}
Christian Voigt.
\newblock Equivariant periodic cyclic homology.
\newblock {\em J. Inst. Math. Jussieu}, 6(4):689--763, 2007.

\bibitem{VOIGT_equivariantperiodiccyclicquantum}
Christian Voigt.
\newblock Equivariant cyclic homology for quantum groups.
\newblock In {\em {$K$}-theory and noncommutative geometry}, EMS Ser. Congr.
  Rep., pages 151--179. Eur. Math. Soc., Z\"urich, 2008.

\bibitem{VOIGT_chern}
Christian Voigt.
\newblock Chern character for totally disconnected groups.
\newblock {\em Math. Ann.}, 343(3):507--540, 2009.

\bibitem{WEIBEL_homologicalalgebra}
Charles~A. Weibel.
\newblock {\em An introduction to homological algebra}, volume~38 of {\em
  Cambridge Studies in Advanced Mathematics}.
\newblock Cambridge University Press, Cambridge, 1994.

\end{thebibliography}

\end{document}